\documentclass{amsart}
\usepackage{amsmath,amsthm,amssymb,amsfonts,pdfsync}
\usepackage{enumitem}
\usepackage{graphicx}
\usepackage{color}
\definecolor{MyGreen}{rgb}{ 0.1667 ,   0.6  ,  0.0235}
\ifx\pdfoutput\undefined \usepackage[hypertex]{hyperref} \else
\usepackage[pdftex,pdfstartview=FitH,colorlinks=true,citecolor=MyGreen]{hyperref} \fi
\numberwithin{equation}{section}
\renewcommand\setminus{\smallsetminus}
\newcommand\IGNORE[1]{}

\def\legendre#1#2{\displaystyle\left(\!\frac{\, #1\, }{\, #2\, }\!\right)}
\newcommand{\tr}{\mathrm{tr}}
\newcommand{\modu}{\mathrm{mod}\ }

\newcommand{\Cl}{\mathrm{Cl}}

\newcommand{\Ht}{\mathrm{ht}}

\newcommand{\so}{\mathrm{SO}}
\newcommand{\mR}{\mathrm{R}}
\newcommand{\SO}{\mathrm{SO}}
\newcommand{\mscS}{\mathscr{S}}
\newcommand{\transp}[1]{{\vphantom{#1}}^{\mathit t}{#1}}
\newcommand {\absolute}[1] {\left| {#1} \right|}

\newcommand{\bash}{\backslash}
\def\peter#1{\langle #1\rangle}
\def\ov#1{\overline{#1}}
\newcommand{\refs}{\eqref}
\newcommand{\ra}{\rightarrow}
\newcommand{\ndiv}{\nmid}
\newcommand{\eps}{\varepsilon}
\newcommand{\Af}{\adele_{f}}
\newcommand{\sym}{\mathrm{Sym}}
\newcommand{\diag}{\mathrm{Diag}}
\newcommand{\x}{\mathbf{x}}

\newcommand{\Pcr}{\mathscr{P}}
\newcommand{\sign}{\mathrm{sign}}

\newcommand{\T}{\mathbf{T}}

\newcommand{\mcL}{\mathcal{L}}

\newcommand{\LduR}{\mathcal{L}^{(1)}_{2}(\Rr)}
\newcommand{\LdR}{\mathcal{L}_{2}(\Rr)}

\newcommand{\End}{\mathrm{End}}

\newcommand{\Zz}{\mathbb{Z}}
\newcommand{\F}{\mathbb{F}}
\newcommand{\Fp}{\F_{p}}

\newcommand{\disc}{\mathrm{disc}}

\newcommand{\Hh}{\mathbb{H}}

\newcommand{\C}{\mathbb{C}}

\newcommand{\vol}{\mathrm{vol}}

\newcommand{\height}{\mathrm{ht}}

\newcommand{\GL}{\mathrm{GL}}

\newcommand{\PGL}{\mathrm{PGL}}

\newcommand{\SL}{\mathrm{SL}}

\newcommand{\N}{\mathbf{N}}
\newcommand{\R}{\mathbb{R}}
\newcommand{\Rr}{\mathbb{R}}
\renewcommand{\H}{\mathbb{H}}

\newcommand{\Z}{\mathbb{Z}}
\newcommand{\Q}{\mathbb{Q}}
\newcommand{\Qq}{\mathbb{Q}}
\newcommand{\Qp}{\mathbb{Q}_{p}}
\newcommand{\Zp}{\Z_{p}}
\newcommand{\Qt}{\mathbb{Q}^\times}

\newcommand{\adele}{\mathbb{A}}

\newcommand{\Pic}{\mathrm{Pic}}
\newcommand{\Reg}{\mathrm{Reg}}

\newcommand{\stab}{\mathrm{stab}}
\newcommand{\Id}{\mathrm{Id}}

\DeclareFontFamily{OT1}{rsfs}{}
\DeclareFontShape{OT1}{rsfs}{n}{it}{<-> rsfs10}{}
\DeclareMathAlphabet{\mathscr}{OT1}{rsfs}{n}{it}

\newcommand{\Cco}{\mathscr{C}_{c}}

\newcommand{\GaG}{{\Gamma\bash G}}

\def\stacksum#1#2{{\stackrel{{\scriptstyle #1}}{{\scriptstyle #2}}}}

\newcommand{\Or}{\mathscr{O}}

\newtheorem*{prop*}{Proposition}
\newtheorem*{fact*}{Fact}
\newtheorem{Theorem}{Theorem}[section]
\newtheorem{prop}[Theorem]{Proposition}
\newtheorem{thm}[Theorem]{Theorem}
\newtheorem*{thm*}{Theorem}

\newtheorem{Lemma}[Theorem]{Lemma}

\newtheorem{Proposition}[Theorem]{Proposition}
\newtheorem{Corollary}[Theorem]{Corollary}

\theoremstyle{definition}

\newtheorem{Remark}[Theorem]{Remark}
\newtheorem{rem}[Theorem]{Remark}
\newtheorem*{rem*}{Remark}

\newcommand{\be}{\begin{equation}}
\newcommand{\ee}{\end{equation}}
\newcommand{\bea}{\begin{eqnarray*}}
\newcommand{\eea}{\end{eqnarray*}}

\newcommand{\Gd}{{\mathscr{G}_{d}}}

\newcommand{\bfx}{\mathbf{x}}
\newcommand{\bfz}{\mathbf{z}}

\newcommand{\GLdZ}{\GL_{2}(\Zz)}

\newcommand{\GLdR}{\GL_{2}(\Rr)}

\newcommand{\Md}{M_{2}}

\newcommand{\MdZ}{M_{2}(\Zz)}

\newcommand{\MdQ}{M_{2}(\Qq)}
\newcommand{\MdR}{M_{2}(\Rr)}

\newcommand{\SLdZ}{\SL_{2}(\Zz)}

\newcommand{\SLdR}{\SL_{2}(\Rr)}

\newcommand{\SOdR}{\mathrm{SO}_{2}(\Rr)}

\newcommand{\PGLd}{\mathrm{PGL}_{2}}

\newcommand{\PSL}{\mathrm{PSL}}

\newcommand{\PGLdZ}{\mathrm{PGL}_{2}(\Zz)}
\newcommand{\PSLdZ}{\mathrm{PSL}_{2}(\Zz)}

\newcommand{\PGLdR}{\mathrm{PGL}_{2}(\Rr)}
\newcommand{\PSOdR}{\mathrm{PSO}_{2}(\Rr)}
\newcommand{\PSLdR}{\mathrm{PSL}_{2}(\Rr)}

\newcommand{\mcR}{\mathcal{R}}
\newcommand{\mcT}{\mathcal{T}}
\newcommand{\bfv}{\mathbf{v}}
\newcommand{\bfw}{\mathbf{w}}
\newcommand{\obfv}{{\ov\bfv}}

\begin{document}

\title[Distribution of closed geodesics]{The distribution of closed geodesics on the modular surface, and Duke's theorem}
\author[M. Einsiedler]{Manfred Einsiedler}
\address[M. E.]{Department of Mathematics, ETH Z\"urich, R\"amistrasse 101
CH-8092 Z\"urich
Switzerland}
\email{manfred.einsiedler@math.ethz.ch}
\author[E. Lindenstrauss ]{Elon Lindenstrauss}
\address[E. L.]{The Einstein Institute of Mathematics\\
Edmond J. Safra Campus, Givat Ram, The Hebrew University of Jerusalem
Jerusalem, 91904, Israel}
\email{elon@math.huji.ac.il}
\author[Ph. Michel]{Philippe Michel}
\address[Ph. M.]{EPF Lausanne, SB-IMB-TAN, Station 8, CH-1015 Lausanne, Switzerland}
\email{philippe.michel@epfl.ch}
\author[A. Venkatesh]{Akshay Venkatesh}
\address[A. V.]{Department of Mathematics, building 380, Stanford, CA 94305, USA}
\email{akshay@math.stanford.edu}
\thanks{M.~E.~acknowledges the support by the Clay Mathematics Institute as a Research
Scholar, by the NSF (grant 0554373) and the SNF (grant 200021-127145).
E.~L.~acknowledges the support of NSF (grants
DMS-0554345 and 0800345), the ISF (grant 983/09) and the European Research Council (Advanced Research Grant 267259). Ph.~M.~was partially supported by the SNF (grant 200021-125291) and the the European Research Council (Advanced Research Grant 228304). A.~V.~was supported by the Clay Mathematics Institute and by NSF Grant DMS-0903110.}
\date{February 2011}
\begin{abstract}
We give an ergodic theoretic proof of a theorem of Duke about equidistribution of closed geodesics
on the modular surface. The proof is closely related to the work of Yu.\ Linnik and B.\ Skubenko,
who in particular proved this equidistribution under an additional congruence assumption on the discriminant.
We give a more conceptual treatment using entropy theory, and show how to use positivity of the discriminant as
a substitute for Linnik's congruence condition.
\end{abstract}
\maketitle
\setcounter{tocdepth}{1}
\tableofcontents

\section{Introduction}
A non-zero integer $d$ is called a {\em discriminant} if it can be represented in the form
$$d=b^2-4ac,\ a,b,c\in\Zz,$$
or equivalently if $d$ is the discriminant
of the binary quadratic form with integral entries
\be\label{binaryquad}q(x,y)=ax^2+bxy+cy^2.
\ee
It is easy to see that $d$ is a discriminant if and only if $d\equiv 0,1(\modu 4)$.
A discriminant $d$ is {\em fundamental} if $d$ is either square-free (in which case $d$ is congruent to $1$ modulo $4$) or $d/4$ is a square-free integer congruent to $2,3(\modu 4)$. Equivalently: $d$ is fundamental if it is the discriminant of the ring of integers of a quadratic field.

The study of integral binary quadratic forms goes back at least to the Greeks.  Significant breakthroughs were accomplished by Gauss. In his Disquitiones arithmeticae he studied the set of
$\GLdZ$-orbits of such forms, where $\GLdZ$ acts via the linear change of variables: for $\gamma=\left(\begin{array}{cc}u & v \\w & z\end{array}\right)\in\GLdZ$
\be\label{SLaction}\gamma.q(x,y)=\frac{1}{\det(\gamma)}q((x,y)\gamma)=\frac{1}{\det(\gamma)}q(ux+wy,vx+zy).
\ee
This action preserves the discriminant and Gauss proved that the set of $\GLdZ$-orbits of integral binary quadratic forms of a given discriminant is {finite},
see \cite[pg.\ 128]{Cassels} for an accessible and more general treatment. Let
\begin{align*}
\mR_{\disc}(d)&=\{q(x,y)=ax^2+bxy+cy^2,\ a,b,c\in\Z,\ \disc(q)=d,\ \gcd(a,b,c)=1\}\\
&\simeq \{(a,b,c)\in\Zz^3,\ \disc(a,b,c)=b^2-4ac=d,\ \gcd(a,b,c)=1\}
\end{align*}
denote the set of forms of discriminant $d$ with coprime coefficients, and let
$$[\mR_{\disc}(d)]=\GLdZ\bash\mR_{\disc}(d)$$
be the set of orbits; its cardinality is the {\em class number} and is noted $h(d)$. Gauss also showed that the set $[\mR_{\disc}(d)]$ could be given an additional structure of an abelian group (the law of composition of quadratic forms), leading to the notion of {\em class group} of quadratic forms of discriminant $d$. Nowadays these venerable and beautiful results are usually interpreted in terms of the theory of quadratic fields and ideal class groups.
We will recall this connection below.

\subsection{Linnik and Skubenko equidistribution theorems}
In the late 50's, Linnik studied more refined properties of the set of representations $\mR_{\disc}(d)$, in particular their distribution properties.

 Let
$$V_{\disc,\pm 1}(\R)=\{(a,b,c)\in\R^3,\ b^2-4ac=\pm 1\};$$
this is a one-sheeted hyperboloid in the $+1$ case and a two-sheeted hyperboloid in the $-1$ case, and is identified with the set of real binary quadratic form with discriminant
 $\pm 1$. In both cases $V_{\disc,\pm 1}(\R)$ is invariant under the natural action of $\GLdR$ extending \refs{SLaction} and has
one  orbit.


The set of representation $\mR_{\disc}(d)$ projects on $V_{\disc,\pm 1}(\Rr)$ (with $\pm1=\sign(d)$) by a homothety
$$|d|^{-1/2}\mR_{\disc}(d)\subset V_{\disc,\pm 1}(\Rr),$$
and Linnik studied how this set is distributed when $d\ra\infty$. These hyperboloids carry a natural $\GLdR$-invariant measure $\mu_{\disc,\pm1}$ defined, for any open set $\Omega\subset V_{\disc,\pm 1}(\Rr)$, as the Lebesgue measure in $\Rr^3$ of the solid cone emanating from the origin and ending at $\Omega$, i.e.\
$$\mu_{\disc,\pm}(\Omega)=\mu_{\Rr^3}(\mathcal{C}(\Omega))$$
where
$$\mathcal{C}(\Omega)=\{ r.\bfx,\ \bfx\in\Omega,\ r\in[0,1]\}.$$

Using an original argument of ergodic theoretic flavor, Linnik \cite[Chap. V]{Linnikbook} established the following equidistribution statement for {\em negative discriminants}.

\begin{thm}[Linnik] \label{linnik-thm}
	Let $p>2$ be a fixed prime. As $d\ra-\infty$ amongst the negative discriminants such that
$\legendre{d}{p}=1$, the set  $$|d|^{-1/2}\mR_{\disc}(d)\subset V_{\disc,-1}(\Rr),$$
becomes equidistributed with respect to $\mu_{\disc,-1}$, in the following sense:
 for any two continuous compactly supported functions $\varphi_{1},\varphi_{2}$ on $V_{\disc,-1}(\Rr)$ such that
 the integral $\mu_{\disc,-1}(\varphi_{2})\not=0$ we have
$$\frac{\sum_{x\in\mR_{\disc}(d)}\varphi_{1}(|d|^{-1/2}x)}{\sum_{x\in\mR_{\disc}(d)}\varphi_{2}(|d|^{-1/2}x)}\ra \frac{\mu_{\disc,-1}(\varphi_{1})}{\mu_{\disc,-1}(\varphi_{2})}\mbox{ as } d\ra-\infty.$$
In particular, $\sum_{x\in\mR_{\disc}(d)}\varphi_{2}(|d|^{-1/2}x)\not=0$ if $d$ as above is large enough.
\end{thm}

Building on Linnik's ergodic method
Skubenko \cite{skubenko} (see also \cite[Chap.\ VI.]{Linnikbook}) proved the analogous statement for {\em positive discriminants}:

\begin{thm}[Skubenko] \label{skubenko-thm}\label{skubenkopage}
Let $p>2$ be a fixed prime. As $d\ra+\infty$ amongst the positive discriminants such that
 $\legendre{d}{p}=1$, the set  $$|d|^{-1/2}\mR_{\disc}(d)\subset V_{\disc,+1}(\Rr),$$
becomes equidistributed with respect to $\mu_{\disc,+1}$, in the following sense:
 for any two continuous compactly supported functions $\varphi_{1},\varphi_{2}$ on $V_{\disc,+1}(\Rr)$ such that the integral $\mu_{\disc,+1}(\varphi_{2})\not=0$ we have
$$\frac{\sum_{x\in\mR_{\disc}(d)}\varphi_{1}(|d|^{-1/2}x)}{\sum_{x\in\mR_{\disc}(d)}\varphi_{2}(|d|^{-1/2}x)}\ra \frac{\mu_{\disc,+1}(\varphi_{1})}{\mu_{\disc,+1}(\varphi_{2})}\mbox{ as } d\ra-\infty.$$
In particular, $\sum_{x\in\mR_{\disc}(d)}\varphi_{2}(|d|^{-1/2}x)\not=0$ if $d$ as above is large enough.
\end{thm}

We refer to Figure \ref{fig-skubenko} for an illustration of the case $d=377$.
\begin{figure}
\includegraphics[scale=0.4]{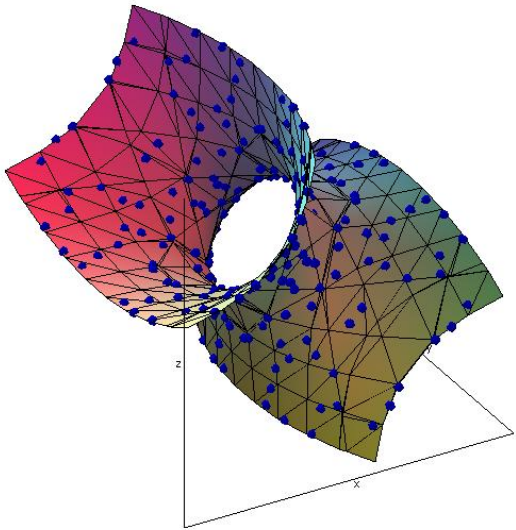}
\caption{\label{fig-skubenko}The distribution of $377^{-1/2}\mR_{\disc}(377)$ viewed on the one-sheeted hyperboloid, note that $h(377)=1$.
}
\end{figure}

The condition  $\legendre{d}{p}=1$ for some fixed prime $p$ is equivalent to the condition that
\begin{center}\em the fixed prime $p$ splits in the quadratic field $\Q(\sqrt d)$.\end{center}
This condition (which we shall refer to as {\em Linnik's condition}) was an essential input for Linnik's ergodic method but, as was pointed out by Linnik himself, it should not be necessary for the equidistribution theorem to hold. It is only thirty years later that this condition was removed in the beautiful work of Duke \cite{Duke}.

\subsection{Duke's theorem}
A key point of Duke's approach is to reformulate the prior theorems
in a dual form:  in terms of equidistribution of ``Heegner points'' (for negative $d$) or of closed geodesics (for positive $d$) on the modular surface $Y_{0}(1):=\SLdZ\bash\H$.

Assuming that $d$ is not a square, one associates to any $(a,b,c)\in\mR_{\disc}(d)$ the geodesic corresponding to the geodesic semi-circle in the upper half plane whose end points are
\begin{equation}\label{xabc}
x_{a,b,c,\pm}=\frac{-b\pm\sqrt d}{2a}.
\end{equation}
We lift this geodesic in the obvious way to the unit tangent bundle of $\Hh$ and then project it to a geodesic orbit on the unit tangent bundle $\T^1(Y_{0}(1))$. This geodesic orbit, which we denote by $\gamma_{[a,b,c]}$, is compact and depends only on the  $\SLdZ$-orbit of $(a,b,c)$. We obtain in this way a collection of $h(d)$ closed geodesics
$$\Gd=\bigcup_{[a,b,c]}\gamma_{[a,b,c]}\subset\T^1(Y_0(1)),$$
see Figure \ref{fig-Duke} for the case $d=377$.
\begin{figure}
\includegraphics[scale=0.4]{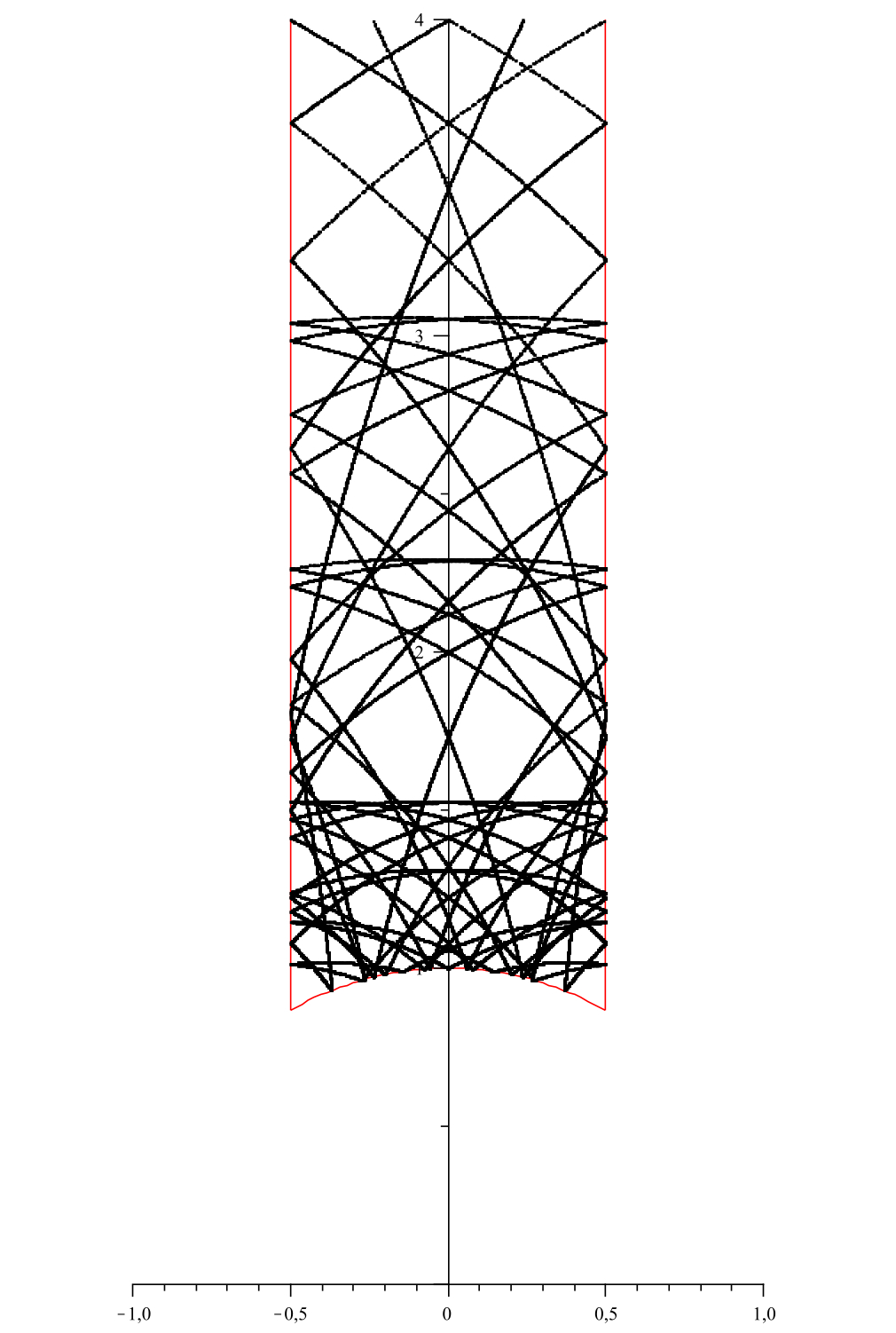}
\caption{\label{fig-Duke}The distribution of $\mathcal{G}_{377}$ projected on the fundamental domain of $\SLdZ\bash\H$, note that $h(377)=1$.}
\end{figure}
This collection of compact orbits of the geodesic flow then carries a natural probability measure invariant under the geodesic flow
which we denote by $\mu_{d}$. Let $\mu_{\mathrm{L}}$ be the Liouville (Haar) probability measure on $\T^1(Y_{0}(1))$, then Duke's theorem (as extended by Chelluri \cite{Chelluri} to the unit tangent bundle)
gives the following:

\begin{thm}[Duke]\label{duke-thm} As $d\ra+\infty$ amongst the positive fundamental discriminants, the set  $\Gd$
becomes equidistributed on the unit tangent bundle $\T^1(Y_{0}(1))$ with respect to the measure $\mu_{\mathrm{L}}$: for any
continuous compactly supported function $\varphi$ on $\T^1(Y_{0}(1))$
$$\int_{\Gd}\varphi(t)d\mu_{d}(t)\ra \int_{\T^1(Y_{0}(1))}\varphi(u)d\mu_{\mathrm{L}}(u).$$
\end{thm}
The equivalence of the equidistribution statement in Theorem~\ref{skubenko-thm} and Theorem~\ref{duke-thm} will be explained in  \S \ref{subsec:duality}.

The restriction to fundamental discriminants is not essential; indeed all the proofs extend to the general case, including the one we present here.   Duke's proof is fundamentally different from Linnik's; it does not rely on ergodic theory but on
harmonic analysis of the modular surface $\SLdZ\bash\H$, that is on the theory of automorphic forms supplemented by deep arguments from analytic number theory and in particular a breakthrough of Iwaniec \cite{Iwaniec}.

\medskip
In this paper we give a new proof of Duke's theorem in the case of positive discriminant. Our proof is strongly influenced by Linnik's ergodic method, and may be seen as a modern incarnation of Linnik's original ideas, and we use the positivity of the discriminant as a substitute to Linnik's condition that Skubenko relied on in his work.

  There are two main ingredients in the proof:
\begin{enumerate}
\item \emph{Linnik's Basic Lemma} --- An upper bound on the number of nearby pairs of points in the projection of $\mR_{\disc}(d)$ to $V_{\disc,-1}(\Rr)$ (as this set is infinite, the quantity to be bounded needs some additional interpretation), which eventually reduces to an upper bound on the number of ways a given binary quadratic form can be represented by a ternary quadratic form.
\item The \emph{uniqueness of measure of maximal entropy} for the flow corresponding to the one parameter group $a _ t = \begin{pmatrix} e ^ t& \\& e ^ {- t} \end{pmatrix}$ on $\SL_2( \Z) \backslash \SL_2 (\R)$.
\end{enumerate}

\noindent
We have made an effort to present both of these main ingredients in a self-contained way, as each relies on some well-known results that are unfortunately well-known in essentially disjoint circles of mathematicians.

\medskip

The second of these two ingredients replaces a more explicit but less conceptual argument of Linnik and Skubenko. The uniqueness of the measure of maximal entropy for this action is well-known (both in the cocompact and finite volume case) and in the cocompact case dates back to work of R.~Bowen \cite{Bowen}. However the version we give here is new in that it allows us to control how much weight $\Gd$ gives to small neighborhoods of the cusp in $\SL_2 ( \Z) \backslash \H$: essentially, we give a finitary version of the uniqueness of measure of maximal entropy in the \emph{noncompact} quotient $\SL _ 2 (\Z) \backslash \SL _ 2 (\R)$.
 This finitary version is the content of Theorem~\ref{Measure}, and involves a careful analysis of how much entropy can be carried by $a_t$-invariant measures that give disproportionately high weight to the cusp.  A cleaner version of the relationship between entropy and mass in the cusp (although not directly applicable for our main purposes)
 is given in Theorem \ref{entropy-cusp}. We believe these results are of independent interest, and will likely have other applications; it also raises some interesting new questions (see e.g.~\cite{EKadirov}).

We mention that another modern exposition of Linnik's method in a similar context (distribution of integer points on spheres) by J. Ellenberg and two of us (Ph.M. and A.V.) has appeared already in \cite{EMV}.  In that work Linnik's Basic Lemma is again a central ingredient, complemented by a different argument to convert the upper bounds provided by the Basic Lemma to equidistribution (i.e.\ both upper and lower bounds on number of points in specified regions). The reader may wish to compare these two complementary approaches.

\subsection{Notation.} \label{sec:notation}
We collect here some notation that is used throughout the paper:

The group $\SLdR$ acts transitively on the upper-half plane model $\Hh$ of the hyperbolic plane by fractional linear transformations and the stabilizer of the point $i$ is the compact subgroup $\SOdR$. The resulting identification
$$\Hh\simeq \SLdR/\SOdR$$
descends to an identification of $\Hh$ with $\PSLdR/\PSOdR$;
moreover the action of $\PSLdR$ on the unit tangent bundle $\Hh$ is simply transitive.
If we let $p \in T^1(\Hh)$ be the tangent vector pointing up at $i$,
then $g \mapsto gp$ gives an identification $\PSLdR\simeq T^{1}\Hh$.
Taking the quotient by $\PSLdZ$ we obtain an identification with the unit tangent bundle of the modular curve\footnote{Actually the modular curve has singularities at the points $i$ and $j=\frac{1+\sqrt{-3}}{2}$ owing to the fact that these points have non-trivial stabilizers in $\PSLdZ$, we will ignore this minor point.}
$\PSLdZ\bash\PSLdR\simeq T^{1}(\PSLdZ\bash\Hh).$

We shall make use of another identification of the quotient $\PSLdZ\bash\PSLdR$, namely with the space of lattices in $\R^2$ up to homothety.
Indeed, the space of lattices $\LdR$ is identified with $\GLdZ\bash\GLdR$ via $g\mapsto \Zz^2.g$; the same map also identifies  the space $[\LdR]$ of  lattices up to homothety with $\PGLdZ\bash \PGLdR$ and  the set  $\mcL^{(1)}_{2}(\Rr)=X$ of lattices of covolume one with $\SLdZ\bash\SLdR =\PSLdZ \bash \PSLdR$. Finally, the  sets $[\LdR]$ and $\LduR$ are also identified via the map $[L]\mapsto \vol(L)^{-1/2}.L$.

Thus the following spaces are identified: $$X \simeq  \PSLdZ\bash\PSLdR\simeq T^{1}(\PSLdZ\bash\Hh) \simeq [\LdR] \simeq \LduR.$$
When we speak of ``the lattice corresponding to $x \in X$,''
we have in mind always the image of $x$ under the isomorphism $X \simeq \LduR$.

We take the following fundamental domain
$$\mscS=\{(z,v)\in \H\times S^1,\ |\Re z|\leq 1/2,\ |z|\geq 1\} \subset T^1(\H) \simeq \PSLdR$$
for $\PSLdZ=\Gamma$.

 Fix an arbitrary left-invariant Riemannian metric $d$ on $\PSL_2(\R)$.
 It descends to a metric on $X$, denoted $d_X$ or simply $d$ for short.
Explicitly we have
\begin{equation} \label{distancelemma}
 d_{X}(\PSLdZ g_1,\PSLdZ g_2)=\min_{\gamma\in\PSLdZ}d(g_1,\gamma g_2)
\end{equation}

The geodesic curves on $T^{1}(\Hh)$ --- which in the upper half-plane are circles and lines intersecting the real axis in a normal angle --- correspond to the orbits of the right $A$-orbits in $\PSLdR$ where $A=\{a_t\}$ is the diagonal subgroup of $\PSLdR$. By a slight abuse, we shall use $A$ to refer to the diagonal subgroup
of all three groups: $\GLdR, \PGLdR$ and $\SLdR$.

\medskip
\noindent
{\bf Acknowledgements:} The authors would like to thank Peter Sarnak for encouragement and many helpful conversations. A.V. would also like to thank Jordan Ellenberg for many discussions on the topic of quadratic forms. The authors also thank Menny Aka, Asaf Katz, Ilya Khayutin, Lior Rosenzweig for carefully going over a preliminary version of this paper.

\newcommand{\tran}[1]{{\vphantom{#1}}^{\mathit t}{#1}}

\section{Representations by the discriminant, orbits and quadratic fields}\label{orbitsection}
In this section we explain in greater detail the relationship between Skubenko's equidistribution theorem and Duke's and connect these questions to the arithmetic of real quadratic fields. Along the way we will find a few equivalent ways in which to describe compact $A$-orbits in $\mathscr{G}_d$.
Building on that we prove in \S \ref{subsec:duality} the equivalence between Skubenko's and
Duke's formulations.


\subsection{Overview of the bijections}\label{subsec:summary}

Recall that we have previously associated to any element of $[\mR_{\disc}(d)]$ --
i.e.\ to any $\GL_2(\Z)$ orbits in $\mR_{\disc}(d)$ --
 a closed geodesic on $\SL_2(\Z) \backslash \H$.
On the other hand, as discussed in \S \ref{sec:notation},
a closed geodesic in $\mathscr{G}_d$  corresponds to a closed $A$-orbit on
the space $X$.

Write $\Or_d := \Z[\frac{d+\sqrt{d}}{2}]$ for the order of discriminant $d$.

We shall show below that the following sets are in natural bijection to each other:
\begin{enumerate}
\item[i.]  $[\mR_{\disc}(d)]$, the set of  $\GLdZ$-orbits of primitive representations
in $\mR_{\disc}(d)$.

\item[ii.] The set of $\GLdZ$-conjugacy classes of ring embeddings $\iota:\Or_d\hookrightarrow \MdZ$ which are {\em optimal}, i.e.\ for which the embedding cannot be extended to an embedding of a strictly bigger order $\Or \gneq \Or_d$ with image in $\MdZ$.

\item[iii.] $\Cl(\Or_{d})=$ the set of $K^\times$-homothety classes of proper $\Or_{d}$-ideals.
%

\end{enumerate}

In the case of a fundamental discriminant the above objects and their bijections are a bit easier to explain. In fact, if $d$ is a fundamental discriminant, then every representation is primitive, every embedding is optimal, and every $\Or_d$-ideal is proper. In reading the remainder of the section the reader may first specialize to this case, or even continue reading with Section \ref{spacingsection} and only refer to the portions of this section as needed for the remainder of the paper.

%

\subsection{Discriminant and quadratic fields}
We establish the bijections of \S \ref{subsec:summary}.

Before beginning, we note that the sequence of maps
\be
\label{SymtoMo}
ax^2+bxy+cy^2\mapsto \left(\begin{array}{cc} a & b/2 \\ b/2 & c\end{array}\right)\mapsto \left(\begin{array}{cc} b & -2a \\ 2c & -b\end{array}\right)
\ee
defines an isometry between the spaces of (real) binary quadratic forms, symmetric $2\times 2$ real matrices and trace zero $2\times 2$ real  matrices, where each of those is equipped with a quadratic form:
$$(\mathcal{Q}(\Rr^2),\disc)\simeq(\sym_{2}(\Rr),-4\det)\simeq (M^0_{2}(\Rr),-\det).$$
The action of $\GLdZ$ in \refs{SLaction} is the restriction of the following action of $\GLdR$ on $\mathcal{Q}(\Rr^2)$:
$$g.q(x,y)=\frac{1}{\det (g)}q((x,y)g)=\frac{1}{\det (g)}q(ux+wy,vx+zy),\ g=\left(\begin{array}{cc}
u & v \\
w & z
\end{array}\right),
$$
which intertwines with the actions
$$g.(ax^2+bxy+cy^2)\longleftrightarrow \frac{1}{\det(g)}g\left(\begin{array}{cc} a & b/2 \\ b/2 & c\end{array}\right)\tran{g}
\longleftrightarrow g\left(\begin{array}{cc} b & -2a \\ 2c & -b\end{array}\right)g^{-1} .$$
Observe that these actions factor through $\PGLdR$.
They also induce an isomorphism  between $\PGLdZ$
and the group  of orthogonal transformations of $(\mathcal{Q}(\Rr^2),\disc)$
preserving the integral quadratic forms.

 Let $d$ be a discriminant which is not a perfect square; let $(a,b,c)\in\mR_{\disc}(d)$ be a representation, and let
\be
\label{mabcdef}m=m_{a,b,c}=\left(\begin{array}{cc}  b& -2a \\ 2c  & -b\end{array}\right)
\ee
 be the trace zero matrix associated to it via the map
\refs{SymtoMo}. Since $$m^2=d\cdot \Id$$
this defines an embedding of the quadratic field ($d$ is not a square) $K=\Qq(\sqrt d)$ into $\MdQ$
$$\iota_{m}:\left.\begin{array}{ccc} K & \mapsto & \MdQ \\
u+v\sqrt d & \mapsto & u\Id+v.m\end{array}\right.$$

\subsubsection{Representations and optimal embedding}
The integrality properties of this embedding are measured by considering
$$\Or_{m}:= \iota_{m}^{-1} ( \Md(\Z))$$ which is an order in $K$.
Let us identify which order:
Note that $\Or_{\lambda.m}=\Or_{m}$ for any
$\lambda\in\Qt$. Hence if $b^2-4ac=d$ for $a,b,c\in\Zz$
we may write $$(a,b,c)=f(a',b',c')$$ with $f\in\Zz$ and $a',b',c'\in\Zz$ coprime integers satisfying
$$\disc(a',b',c')=d'=d/f^2.$$ This reduces the discussion to the case where $(a,b,c)$ is a {\em primitive representation} of $d$ (a representation
with coprime entries).

Assuming that $(a,b,c)$ is primitive, one sees quickly that
\begin{equation}\label{order-d}
 \iota_{m}^{-1}(\Or_{m})=\Or_{d}=\Zz[\frac{d+\sqrt d}2]
\end{equation}
is the order of {\em discriminant} $d$. If \eqref{order-d} holds, we say that $\iota_{m}$ defines an optimal embedding of $\Or_{d}$ into $\MdZ$.
We obtain in that way a bijection between
$$\hbox{the set of  $\GLdZ$-orbits of primitive representations $[\mR_{\disc}(d)]$ }$$
and
$$\hbox{the set of $\GLdZ$-conjugacy classes of optimal embeddings $\iota:\Or_{d}\hookrightarrow \MdZ$.}$$

\subsubsection{Embeddings and ideal classes}
Let us recall, that a lattice $I\subset K$ is a proper $\Or_{d}$-ideal, iff
$$\Or_{I}:=\{\lambda\in K,\ \lambda.I\subset I\}=\Or_{d}.$$
Then there is a bijection between
$$\hbox{  the set of $\GLdZ$-conjugacy classes of optimal embeddings of $\Or_{d}$ }$$
and the set of {\em proper ideal classes} of $\Or_{d}$
$$\Cl(\Or_{d})=\hbox{ the set of $K^\times$-homothety classes of proper $\Or_{d}$-ideals}.$$

This bijection goes as follows \cite{McDLatimer}:
Given a proper $\Or_{d}$-ideal $I\subset K$, one may choose a $\Zz$-basis $I=\Zz.\alpha+\Zz.\beta$ which gives an identification
$$\theta:\left.\begin{array}{ccc} I& \mapsto & \Zz^2 \\
u\alpha+v\beta & \mapsto & (u,v)\end{array}\right.$$
This identification induces the embedding
$$\iota: K\hookrightarrow \MdQ$$
defined by $$\iota(\lambda)(u,v)=\theta(\lambda.(u\alpha+v\beta)),$$
(or in other terms, such that $\theta(\lambda.x)=\theta(x)\iota(\lambda)$).

Since $\Or_{d}.I\subset I$, one has $\iota(\Or_{d})\Zz^2\subset\Zz^2$, that is $\iota(\Or_{d})\subset\MdZ$ and the fact that $I$ is a proper $\Or_{d}$-ideal is equivalent to the fact that $\iota$ is an optimal embedding of $\Or_{d}$.
If we replace the $\Zz$-basis $(\alpha,\beta)$ by another basis $(\alpha',\beta')$ then $\iota$ is replaced by a $\GLdZ$-conjugate. Finally if $I$ is replaced by an ideal in the same class $I'=\lambda.I$ $\lambda\in K^\times$, then the corresponding $\GLdZ$-conjugacy classes coincide: $[\iota_{I'}]=[\iota_{I}]$.

The inverse of the map
$$[I]\mapsto [\iota_{I}]$$
is as follows: given an optimal embedding $\iota:K\mapsto \MdQ$ of $\Or_{d}$, let $e_{1}=(1,0)\in\Zz^2$ be the first vector of the standard basis\footnote{We could have chosen any primitive vector in $\Zz^2$.} of $\Zz^2$, then the map
$$\theta:\left.\begin{array}{ccc} K& \mapsto & \Qq^2 \\
\lambda & \mapsto &e_{1}.\iota(\lambda) \end{array}\right.$$
is an isomorphism of $\Qq$-vector spaces; next define the lattice $I=\theta^{-1}(\Zz^2)$ in $K$ which is invariant under multiplication by
 $\Or_{d}$. In other words, $I$ is an $\Or_{d}$-ideal and $I$ being proper is equivalent to $\iota$ being optimal.

\subsubsection{The Picard group of the order $\Or_d$}
We now recall the definition and basic properties of the Picard group for an order $\Or_d$ in a quadratic field.

The product of two $\Or _ d$-ideals $I$ and $J$ gives another $\Or _ d$-ideal \[I\cdot J = \left\{ \lambda  \lambda ': \lambda \in I, \lambda ' \in J \right\};\]
and clearly this operation respects the equivalence relation introduced above on $\Or _ d$-ideals.
An $\Or _ d$-ideal $I$ is \emph{invertible} if there is some $\Or _ d$-ideal $J$ so that $I\cdot J=\Or_d$. An $\Or _ d$-ideal $I$ is \emph{locally principal} if for any prime $p$,
$$I_{p}:=I \otimes_{\Zz}\Zz_{p}=\lambda_{p} (\Or_d) _ {p},$$
where $(\Or_d) _ {p}=\Or _ d \otimes_\Zz \Zz_p$ and $\lambda_{p}$ is an element of $(K \otimes_{\Qq}\Qp)^\times$. Both properties depend only on the ideal class $[I]$ and not on $I$ itself.

For general orders $\Or$ in number fields and $\Or$-ideals $I$, one has the following implications
\begin{equation*}
\text{$I$ is locally principal} \Longrightarrow \text{$I$ is invertible} \Longrightarrow \text{$I$ is proper}
.\end{equation*}
We shall make use of the following property of orders in quadratic number fields:

\begin{prop}\label{prop-picard} For the orders $\Or_d$ in quadratic number fields the inverse implication \[
\textup{$I$ is proper} \Longrightarrow \textup{$I$ is locally principal}\] holds for $\Or_d$-ideals $I$. In particular, the set of proper ideal classes $\Cl(\Or_d)$, endowed with the composition law induced by forming the product of two lattices, has the structure of an abelian group.
\end{prop} This nice special feature of quadratic orders comes from the fact that in the quadratic case, orders are always monogenic (i.e. of the form $\Or = \Z [x]$).

\begin{proof}
Recall that $\Or _ d = \Zz[x]$ for $x=\frac{d+\sqrt d}2$.
Assume now that $I$ is a proper $\Or_d$-ideal and consider the $2$-dimensional $\Fp$-vector space $I_{p}/pI_{p}\simeq I/pI$. The natural map $$(\Or_d)_{p}/p (\Or_d)_{p} \mapsto \End_{\Fp}(I_{p}/pI_{p})$$
is injective. To see this, suppose that $\lambda \in (\Or_d)_{p}$ acts trivially on $I_{p}/pI_{p}$. Then $\lambda I_{p}\subset p I_{p}$ and $\frac{\lambda}p I_{p}\subset I_{p}$ and so $\frac{\lambda}p \in \Or_{p}$ as required. It follows that $\ov x$ the image of $x$ in $ \End_{\Fp}(I_{p}/pI_{p})$
has a minimal polynomial of degree $2$ and that $I_{p}/pI_{p}$ is a cyclic $\Fp[\ov x]$-module. So there exist $\lambda_{p}\in I_{p}$
such that $I_{p}=\lambda_{p}(\Or_d)_{p}+pI_{p}$ which implies that
\begin{multline*}
	 I_{p}=
\lambda_{p}(\Or_d)_{p}+p(\lambda_{p}(\Or_d)_{p}+pI_{p})=\\
=\lambda_{p}(\Or_d)_{p}+p^2I_{p}=\lambda_{p}(\Or_d)_{p}+p^3I_{p}=\ldots=\lambda_{p}(\Or_d)_{p}.
\end{multline*}
\end{proof}

\subsection{Interpretation in terms of lattices}
Let us verify that the various descriptions of $\mathscr{G}_d$ are equivalent:

 Given $(a,b,c) \in R_{\disc}(d)$, put $h_{a,b,c}=\begin{pmatrix}b+\sqrt{d}&b-\sqrt{d}\\2c&2c\end{pmatrix}$ and $w = \begin{pmatrix}0&-1\\1&0\end{pmatrix}\in\SLdZ$.
 Then $w h_{a,b,c}$ maps $(\infty, 0)$ to $\frac{-b\pm\sqrt d}{2a}$.
 Therefore, the geodesic $\gamma_{[a,b,c]}$ on $\PSLdZ \bash \Hh$ associated to $(a,b,c)$
 after equation \eqref{xabc} is:
$$\gamma_{[a,b,c]} = w h_{a,b,c} . (0,\infty),$$ where $(0,\infty)$ is the geodesic on $\Hh$  joining $0$ and $\infty$.
Now $(0,\infty)$ corresponds, in the realization $T^1(\Hh)$,
to the $A$-orbit of the identity in $\SLdR$; therefore
$\gamma_{[a,b,c]}$ corresponds to $\SLdZ \cdot   w h_{a,b,c}A=\SLdZ \cdot h_{a,b,c} A$,
or equivalently the lattices of the form $\Z^2 \cdot h_{a,b,c} a_t \subset \mathcal{L}_2^{(1)} \ (a_t \in A)$.
 Now one calculates
\[
\frac1{\det (h_{a,b,c})} h_{a,b,c}\begin{pmatrix}0&\frac12\\\frac12&0\end{pmatrix} ^{t}h_{a,b,c}
=\frac1{\sqrt{d}}\begin{pmatrix}a&\frac{b}2\\\frac{b}2&c\end{pmatrix},
\]
which shows that in a particular basis of $\Z^2h_{a,b,c}$ the quadratic form $q_0(x,y) = xy$ takes the shape as in \eqref{special-shape}.

Since $A$ is the stabilizer subgroup of $q_0$, we have verified that $\gamma_{[a,b,c]}$ corresponds to:
\begin{quote}  The set of homothety classes of lattices $L$, such that the restriction of the quadratic form $q_{0}(x,y)=xy$ to $L$, expressed in terms of a basis $\alpha,\beta$ of $L$, take the form
\begin{equation}\label{special-shape}
 q_{0}(u\alpha+v\beta)=\vol(L)\frac{au^2+buv+cv^2}{d^{1/2}}.
\end{equation}
\end{quote}

Note that the  particular quadratic form $\frac{au^2+buv+cv^2}{\sqrt{d}}$ is not
canonically attached to the lattice $L$ because of the different choices of a basis.

Set $m_0=\begin{pmatrix}1&0\\0&-1\end{pmatrix}$ and $\iota_{0}$ to be the embedding $\iota_{0}:K\hookrightarrow \diag_{2}(\Rr)\subset\MdR$ obtained by mapping $\sqrt d$ to $d^{1/2}m_{0}$ and $\theta_{0}$ be the linear embedding $\theta_{0}: K\hookrightarrow \R^2$ given by $$\theta_{0}(\lambda)=(1,1)\iota_{0}(\lambda),\ i.e.. \ \theta_{0}(u+v\sqrt d)=(u+v|d|^{1/2},u-v|d|^{1/2}).$$

Now let us verify, as asserted in \S \ref{subsec:summary}, that the $A$-orbit of
$\theta_0(I)$ belongs to $\mathscr{G}_d$,  for any proper $\Or_d$-ideal $I$.
(We don't verify the more precise assertion that
this is exactly the element of $\mathscr{G}_d$ that corresponds to the class of $I$ under the bijection $\Cl(\Or_d) \leftrightarrow \mR_{\disc}$).
We need to verify (according to \eqref{special-shape})
that $\lambda \in I \mapsto \frac{q_0(\theta_0(\lambda))}{\vol(\theta_0(I))} \sqrt{d}$
is a quadratic form of discriminant $d$. But
$q_{0}(\theta_{0}(\lambda))=\N_{K/\Qq}(\lambda)$ is the norm;  and for any ideal $I\subset K$ that
$\vol(\theta_{0}(I))=|d|^{1/2}\N(I)$.
Here we have defined  {\em norm} $\N(I)$ of an ideal (relative to $\Or_{d}$) by the ratio of indexes
$$
 \N(I)=\frac{(\Or_{d}:\Or_{d}\cap I)}{(I:\Or_{d}\cap I)}.
$$
Now, for any ideal $I$, the map $ x \in I\mapsto \frac{\N_{K/\Qq}(x)}{\N(I)}$
is easily verified to be an integer quadratic form of discriminant $d$, as desired.




%



\subsection{A duality principle}\label{subsec:duality}
Our goal now is to show that the equidistribution statements
of Skubenko's theorem and of Duke's theorem are equivalent.

The discussion which follows is valid in great generality; but we will consider only $G=\PGLdR$, $\Gamma=\PGLdZ$,
and the diagonal torus $A$ in $G$.

Since $\PGLdR$ is identified with $\SO_{\disc}(\Rr)$, it acts transitively on $V_{\disc,+1}(\Rr)$ (by Witt's theorem) and equals the $\PGLdR$-orbit of (say) $q_{0}(x,y)=xy$; equivalently $V_{\disc,+1}(\Rr)$ is identified with the $\PGLdR$-conjugacy class of
the matrix $m_{0}$ which has $A$ as its stabilizer subgroup in $G$. Hence
$$V_{\disc,+1}(\Rr)= \PGLdR.q_{0}\simeq \PGLdR.m_{0}\simeq \PGLdR/A.$$

\subsubsection{Duality between orbits} \label{subsub:duality}
It follows from the previous discussion that each representation $(a,b,c)\in\mR_{\disc}(d)$
is identified with some class $g_{a,b,c}A/A\in G/A$ or what is the same to an orbit $g_{a,b,c}A\subset G$ for some $g_{a,b,c}\in G$ such that $$g_{a,b,c}.q_{0}=|d|^{-1/2}(a,b,c),\ q_{0}=(0,1,0).$$

As we have seen $\Gamma$ acts on $\mR_{\disc}(d)$ and the latter decomposes into a finite disjoint union of $\Gamma$-orbits, setting
$$[a,b,c]=\Gamma\bash\Gamma(a,b,c)\in
[\mR_{\disc}(d)],$$ for the orbit of $(a,b,c)$, one has
  $$\mR_{\disc}(d)=\bigsqcup_{[a,b,c]\in [\mR_{\disc}(d)]}\Gamma.(a,b,c) $$
Hence $|d|^{-1/2}.\mR_{\disc}(d)$ is identified with the collection of $\Gamma$-orbits
  $$\bigsqcup_{[a,b,c]\in [\mR_{\disc}(d)]}\Gamma g_{a,b,c} A/A\subset G/A;$$
thus the problem of the distribution of $|d|^{-1/2}.\mR_{\disc}(d)$ inside $V_{\disc,+1}(\Rr)$ is a problem
 about the distribution of a collection of $\Gamma$-orbits inside the quotient space $G/A$.

There is an almost tautological equivalence between (left) $\Gamma$-orbits on $G/A$ and (right) $A$-orbits on $\Gamma\backslash G$ given by
  \begin{equation}
\label{correspondence}
\Gamma g A/A\longleftrightarrow \Gamma g A\longleftrightarrow\Gamma\backslash \Gamma g A.
\end{equation}
This duality induces a close relationship between
 the study of the distribution of $|d|^{-1/2}.\mR_{\disc}(d)$ inside $V_{\disc,+1}(\Rr)$
and the distribution of the collection of right-$A$ orbits
 $$\Gd=\bigcup_{[a,b,c]\in[\mR_{\disc}(d)]}x_{[a,b,c]} A\subset \Gamma\backslash G$$
  inside the  homogeneous space $\Gamma\backslash G$,
with \begin{equation}\label{xabcdef} x_{[a,b,c]}=\Gamma\backslash\Gamma g_{a,b,c}.\end{equation}
 This is the ``duality principle'' alluded to at the beginning of this section. Let us make this principle a bit more precise
 by identifying the orbits in question:

Assuming that $(a,b,c)\in\mR_{\disc}(d)$ is primitive; one has
$$x_{[a,b,c]} A=\Gamma\bash \Gamma g_{a,b,c}A=\Gamma\bash \Gamma A_{a,b,c }g_{a,b,c}$$ where
$$A_{a,b,c}=g_{a,b,c}Hg_{a,b,c}^{-1}=\stab_{(a,b,c)}(G)$$
is the stabilizer of $(a,b,c)$ in $G$. That group is the group of real points of a $\Qq$-algebraic group, which we will denote by $\T_{a,b,c}$, namely the image in $\PGLd$ of the centralizer $Z_{m}$ of
$$m=m_{a,b,c}=\left(\begin{array}{cc} b & 2c \\ -2a & -b\end{array}\right).$$
 In terms of the embedding $\iota=\iota_{m_{a,b,c}}:K\hookrightarrow \MdQ$, one has
$$Z_{m}(\Qq)=\iota(K^{\times}),$$
and
$$\T(\Qq)=\iota(K^\times)/\Q^\times\Id,\ A_{a,b,c}=\T_{a,b,c}(\Rr)=\iota(K\otimes\Rr)^\times/\Rr^\times\Id,$$
and (since  $\MdZ\cap \iota(K)=\iota(\Or_{d})$),
$$\Gamma_{a,b,c}:=\Gamma\cap A_{a,b,c}=\iota(\Or_{d}^\times)/\{\pm \Id\}.$$
Alternatively, let $\iota_{0}$ denote the (real) embedding
$$\iota_{0}:\left.\begin{array}{ccc} K & \mapsto & \MdR \\
u+v\sqrt d & \mapsto & u\Id+v.d^{1/2}m_{0}\end{array}\right.$$
obtained by conjugating $\iota_{m}$ with $g_{a,b,c}^{-1}$, we have
$$\iota_{0}(K\otimes_{\Qq}\Rr)^\times\!/\Rr^\times\Id=A$$
and
$$\Gamma'_{a,b,c}:=g^{-1}_{a,b,c}\Gamma g_{a,b,c}\cap A=\iota_{0}(\Or_{d}^\times)/\{\pm \Id\}$$
so that we have homeomorphisms
\be\label{orbitabc}x_{[a,b,c]} A=\Gamma\bash g_{a,b,c} A\simeq g^{-1}_{a,b,c}\Gamma g_{a,b,c}\cap A\bash A=
 \iota_{0}(K\otimes\Rr)^\times/\Rr^\times\iota_{0}(\Or_{d}^\times).
 \ee
By Dirichlet's unit theorem, $\iota_{0}(K\otimes\Rr)^\times\!/\Rr^\times\iota_{0}(\Or_{d}^\times)$ is compact hence
  $x_{[a,b,c]}A$ is compact and since $[\mR_{\disc}(d)]$ is finite we obtain:

\begin{thm}\label{compactfact} The union of $A$-orbits $\Gd$  is compact.
 \end{thm}


\subsubsection{Duality between measures}
To consider equidistribution problems, one needs to refine 
the correspondence   \refs{correspondence} at the level of  measures.  Roughly speaking, the choice of  the counting measure
$\mu_{\Gamma}$ on $\Gamma$ and of  left-invariant Haar measure  $\mu_{A}$ on\footnote{Note that $A$ is unimodular.} $A$
 define a measure theoretic version of the correspondence \refs{correspondence}:

\begin{fact*} There exists homeomorphisms
 between the following spaces of Radon measures (relative to the weak-* topology):
   \begin{equation}
\label{correspondance2}
\begin{matrix}\hbox{left $\Gamma$-invariant}\\
\hbox{ Radon measures}\\
\hbox{$\lambda$ on $G/A$}
\end{matrix} \longleftrightarrow \begin{matrix}\hbox{left $\Gamma$, right $A$-invariant}\\
\hbox{ Radon measures}\\
\hbox{$\rho$  on $G$} \end{matrix}\longleftrightarrow \begin{matrix}\hbox{right $A$-invariant}\\
\hbox{ Radon measures}\\
\hbox{$\nu$  on $\Gamma\backslash  G$}.
 \end{matrix}
\end{equation}
These homeomorphisms are
characterized by the identities: for any $\varphi\in\Cco(G)$, one has
$$\lambda(\varphi_{A})=\rho(\varphi)=\nu(\varphi_{\Gamma})$$
where
$$\varphi_{A}(g):=\int_{A}\varphi(gh)d\mu_{A}(h),\ \varphi_{\Gamma}(g)=\sum_{\gamma\in \Gamma}\varphi(\gamma.g).$$
\end{fact*}

\noindent
See for instance \cite[\S 8.1]{BeOh} for a proof of that fact. We work out this correspondence in specific cases:

\begin{itemize}[leftmargin=*]
\item $\rho$ is a Haar measure $\mu_{G}$ on $G$, which is $G$-biinvariant as $G$ is unimodular. The correspondence \refs{correspondance2} yield the quotient measures $\nu=\mu_{\Gamma \backslash G}$ on $\Gamma \backslash G$, and $\lambda=\mu_{G/A}\varpropto \mu_{\disc,\pm 1}$ on $G/A$.  The former measure $\nu$ is finite  (i.e.\ $\Gamma$ is a lattice in $G$) and we may adjust $\mu_{G}$ so that $\mu_{\Gamma \backslash G}$ is a probability measure.
\item The sum $\lambda_d$ of Dirac measures on $G/A$ given by
\begin{align*}
\lambda_{d}&=\sum_{(a,b,c)\in\mR_{\disc}(d)}\delta_{g_{a,b,c}A/A} =\sum_{[a,b,c]}\sum_{g \in \Gamma.g_{a,b,c}}\delta_{gA/A}\\
&=\sum_{[a,b,c]}\sum_{\gamma \in \Gamma/\Gamma_{a,b,c}}\delta_{\gamma g_{a,b,c}A/A}.
\end{align*}
\begin{prop*} The measure $\nu_{d}$ on $\GaG$ corresponding to $\lambda_{d}$ under \refs{correspondance2} is the sum of the push forwards of the Haar measure $\mu_{A}$ over the set of $A$-orbits $x_{[a,b,c]}A$, $[a,b,c]\in[\mR_{\disc}(d)]$.
\end{prop*}
\end{itemize}
Indeed, set $\lambda_{[a,b,c]}=\sum_{\gamma \in \Gamma/\Gamma_{a,b,c}}\delta_{\gamma g_{a,b,c}A/A}$. Then if $S$ denotes a fundamental domain in $A$ for $\Gamma ' _ {a, b, c}$
\begin{align*}
\lambda_{[a,b,c]}(\varphi_{A})	 &=\sum_{\gamma \in \Gamma/\Gamma_{a,b,c}}\int_{A}\varphi(\gamma g_{a,b,c}h)dh = \sum_ { \gamma \in \Gamma} \int_ S \varphi(\gamma g_{a,b,c}h)dh \\
&=\int_{\Gamma'_{a,b,c}\bash A}\varphi_{\Gamma}(g_{a,b,c}h)dh=\int_{x_{[a,b,c]}A}\varphi_{\Gamma}(h)dh,
\end{align*}
hence the measure on $\Gamma \backslash G$ corresponding to $\lambda_{[a,b,c]}$ is given by the push forwards of the Haar measure $\mu_{A}$ to the periodic $A$-orbit $x_{[a,b,c]}A$, and the proposition follows.

 Let $$\vol(\Gd):=\nu_{d}(\Gd)=\sum_{[a,b,c]}\vol(x_{[a,b,c]}A),$$
 denote the total volume of this (finite) collection of (compact) $A$-orbits.
 From \refs{orbitabc} we see that the various orbits associated to primitive representations of $d$ have the same volume, namely with the correct normalization of the Haar measure of $A$
 $$\vol(x_{[a,b,c]}A)=\vol(\Rr^\times\iota_{0}(\Or_{d}^\times)\bash A)=\Reg(\Or_{d})$$
 where $\Reg(\Or_{d})$ is the {\em regulator} of $\Or_{d}$.
Therefore,
\[
\vol(\Gd)=|\Pic(\Or_{d})|\,\Reg(\Or_{d}).
\]
\IGNORE{ Decomposing $\Gd$ into orbits associated with primitive presentations,
 $$\Gd=\bigsqcup_{f^2|d}\sum_{[a,b,c]\in[\mR_{\disc}(d/f^2)]}x_{[a,b,c]}A=\bigsqcup_{f^2|d}
 \mathscr{G}^*_{d/f^2},$$
 one has
 \[\vol(\Gd)=\sum_{f^2|d}\vol(\mathscr{G}^*_{d/f^2})=\lambda\sum_{f^2|d}|\Pic(\Or_{d/f^2})|\,\Reg(\Or_{d/f^2}).\]}
If $d=\disc(\Or_{K})$ is a fundamental discriminant, the {\em Dirichlet class number formula} gives
$$\vol(\Gd)=|\Pic(\Or_{d})|\,\Reg(\Or_{d})=\lambda|d|^{1/2}L\bigl(\bigl(\frac{d}\cdot\bigr),1\bigr)$$
where $\lambda$ is some absolute constant, $(\frac{d}\cdot)$ is the Kronecker symbol and $L((\frac{d}\cdot),s)$ its associated $L$-function. Then by Siegel's theorem $L((\frac{d}\cdot),1)=|d|^{o(1)}$ as $d\ra\infty$ so that
\be\label{volestimate}\vol(\Gd)=|d|^{1/2+o(1)}.\ee
If $d=d'f^2$ with $d'$ a fundamental discriminant
\begin{equation*}
\frac {|\Pic(\Or_{d})|\,\Reg(\Or_{d}) }{ |\Pic(\Or_{d'})|\,\Reg(\Or_{d'})} = f \prod_ {p|f} \left (1- p^{-1}\left (\frac  {d '}p \right ) \right)
\end{equation*}
which
shows again that $|\Pic(\Or_{d})|\Reg(\Or_{d})=|d|^{1/2+o(1)}$ and hence \refs{volestimate}
holds in general (c.f.\ e.g.\ \cite[Sect.~9.6]{Edwards}).
  We let
 $$\mu_{d}:=\frac{1}{\vol(\Gd)}\nu_{d}.$$ This is an $A$-invariant probability measure on $\GaG$ and the above discussion shows that
Skubenko's Theorem on page \pageref{skubenkopage}
follows from the following:
 \begin{thm}\label{discgammaequid} As $d\ra\infty$ amongst the non-square discriminants, the sequence of measures $\mu_{d}$ weak-* converge to the probability measure $\mu_{\GaG}$, i.e.\ for any $\varphi_{\Gamma}\in\Cco(\GaG)$, one has
 $$\mu_{d}(\varphi_{\Gamma})=\frac{1}{\vol(\Gd)}\sum_{[a,b,c]}\int_{x_{[a,b,c]A}}\varphi_{\Gamma}(h)dh\ra \mu_{\GaG}(\varphi_{\Gamma}).$$
 \end{thm}
Indeed any continuous compactly supported function on $G/A$ is of the form $\varphi_{A}$ for $\varphi\in\Cco(G)$, hence by Theorem~\ref{discgammaequid}
\begin{multline*}
\lambda_{d}(\varphi_{A})=\nu_{d}(\varphi_{\Gamma})=\vol(\Gd)\mu_{d}(\varphi_{\Gamma})\\=
\vol(\Gd)(\mu_{\GaG}(\varphi_{\Gamma})+o(1))=
\vol(\Gd)(\mu_{G/A}(\varphi_{A})+o(1)).
\end{multline*}

\section{Spacing properties of torus orbits}\label{spacingsection}

In this section, we show that the various distinct orbits $x_{[a,b,c]}A\subset\Gd$ are in a suitable sense {\em well spaced} from each other; the main result is Proposition
\ref{all-info-2}.
 Recall that
$$\Gd=\bigsqcup_{[a,b,c]\in[\mR(d)]}x_{[a,b,c]}A.$$
where $x_{[a,b,c]}$ is defined in \eqref{xabcdef}.


\subsection{Ideal classes are controlling the time spent near the cusp.}
The space $X$ is not compact and this is measured through a height function (normalized to be invariant under scaling): given, for $L=\Zz^2.g\subset \Rr^2$ a lattice, by
$$\Ht(L)=\Bigl(\frac{\min_{x\in L-\{0\}}\|x\|}{\vol(L)^{1/2}}\Bigr)^{-1}=\Bigl(\frac{\min_{x\in\Zz^2-\{0\}}\|xg\|}{|\det(g)|^{1/2}}\Bigr)^{-1}.$$
where $\|.\|$ denote the Euclidean norm.
This continuous function is proper.  Indeed, if $x \in X$
and $(z,v) \in \mscS$ any representative, then the height $\Ht(x)$ and the imaginary part $\Im(z)$ satisfy $\Im(z)=\Ht(x)^2$.
For any $H>1$ let $X_{\geq H}$
 denote the set of all $x\in X$ with $\Ht(x)\geq H$.

In this section we evaluate explicitly how big the height of a lattice in $\Gd$ could be.

\begin{Proposition}\label{class-group}  Suppose the proper integral ideal
$J \subset \Or_d$ corresponds to $[a,b,c] \in \mR_{\disc}(d)$ under the bijection of \S \ref{subsec:summary}.
Then $x_{[a,b,c]}A \cap X_{\geq H} $ is nonempty if and only if $J^{-1}$ is equivalent to an ideal $I\subset\Or_{d}$ of norm
 $\leq\frac12 H^{-2}d^{1/2}$. Moreover, this defines a bijection between
 connected component $\Gd\cap X_{\geq H}$ and  proper $\Or_{d}$-ideal $I\subset \Or_d$ of norm $\leq\frac12 H^{-2}d^{1/2}$.
\end{Proposition}

Even though the above does not control escape of mass for $\mu_d$ as $d\rightarrow\infty$ it does give an upper bound
for $\mu_d(X_{\geq H})$, see Proposition \ref{all-info-1}, which we will use in our proof of Duke's theorem.
Note that Proposition \ref{prop-picard} guarantees that there is an inverse $J^{-1}$ to the proper ideal $J$.

\begin{Remark}\label{rem1/4} Applying this result to $H=d^{1/4}$ we see that $\Gd\cap X_{\geq d^{1/4}}$ is empty (as there are no ideals of norm $< 1$). This implies that $\Gd$ is pre-compact.
\end{Remark}

\begin{proof}
Note that, if we identify $x \in X$ with a lattice $L$ of covolume $1$, then
$x A \cap X_{\geq H}$ is nonempty  if and only if there is some nonzero vector $(u,v)\in L$ with $|uv|\leq \frac12H^{-2}$.


 Therefore (using the explicit bijection of \S \ref{subsec:summary}) the $A$-orbit defined by $J$ intersects $X_{\geq H}$, if and only if $J$ contains an element $\lambda$ with $$|\N(\lambda)|\leq \frac12H^{-2}\N(J)d^{\frac12}.$$  Recall that $\N(J^{-1})=\N(J)^{-1}$ by standard properties of the norm. It follows that the $A$-orbit defined by $J$ intersects $X_{\geq H}$ if and only if $\N(\lambda J^{-1})\leq \frac12H^{-2}d^{\frac12}$ for some $\lambda\in J$ (so that $\lambda J^{-1}\subset\Or_{d}$).

Finally, notice that for $H>1$ there is, in a lattice $L'\in X_{\geq H}$, up to sign, only one primitive nonzero vector of length $\leq H^{-1}\vol(L')^{1/2}$ (which is a simple volume computation). Therefore, fixing $J$, in the above argument, a connected component of $\theta_{0}(J).A\cap X_{\geq H}$ corresponds to a unique primitive element $\lambda\in J$ with $|\N(\lambda)|\leq \frac12H^{-2}\N(J)d^{\frac12}$ (up to sign) and we can associate to this connected component the ideal $I=\lambda J^{-1}\subset\Or_{d}$ of norm $\leq \frac{1}2H^{-2}d^{\frac12}$.
\end{proof}

\begin{Proposition}\label{all-info-1}
There is ``not too much mass high in the cusp'' in the sense that
$$\mu_{d}(X_{\geq H}) \ll_{\varepsilon} d^{\varepsilon} H^{-2}$$ for all $\varepsilon>0$
and $H\geq 1$.
\end{Proposition}

Note that to make this estimate useful, we will set later $H=d^{\varepsilon}$ for some $\varepsilon>0$.

\proof
 We note first that in any orbit in $\Gd$ the maximal height achieved is $\leq d^{\frac14}$ (see  Remark \ref{rem1/4}). This implies that for $H>1$ any connected component of $\Gd\cap X_{\geq H}$
 has length $\ll\log ({d})$. Indeed such a component corresponds (in the upper-half plane model)
   to the segment of some oriented geodesic circle (i.e.\ a half circle centered on the real line) made of whose points which have imaginary part between $H$ and $d^{1/4}$: the hyperbolic length of such a segment is bounded by $\ll\log({d^{\frac14}}/H)$.

   Therefore, by Proposition \ref{class-group} $$\vol(\Gd\cap X_{\geq H})\ll  \log(d) N_{\leq H}(d)$$ where $N_{\leq H}(d)$ is the number of proper ideals $I\subset \Or_{d}$ of norm $\N(I)\leq \frac12 H^{-2}d^{\frac12}$. Recall that for any $n\in\mathbb{N}$ the number of
 proper ideals in $\Or_{d}$ of norm equal to $n$ is bounded by the number of divisors of $n$ and so by $\ll_\eps n^\eps$. By summing over all $1\leq n\leq\frac12 H^{-2}d^{\frac12}$ we get that $N_{\leq H}(d)\ll_\eps (H^{-2}d^{\frac12})^{1+\eps}$. Together with \eqref{volestimate} this proves the proposition.
\qed

\subsection{Linnik's basic lemma and representing binary quadratic forms by ternary forms}

 Following Linnik we will derive the ``basic lemma'' from representation numbers of quadratic forms:
Let $q,Q$ be two integral non-degenerate quadratic forms on $\Zz^m$ and $\Zz^n$ respectively. Assuming that $m\leq n$, {\em a representation} of $q$ by $Q$ is an isometric embedding of quadratic lattices
$$\iota:(\Zz^m,q)\hookrightarrow (\Zz^n,Q)$$
in other terms a $\Zz$-linear map $\iota:\Zz^m\ra\Zz^n$ such that for $\bfx\in\Zz^m$
$$Q(\iota(\bfx))=q(\bfx).$$
For instance a representation $\bfx\in\Zz^n$ of an integer $d\in\Zz$ by a quadratic form $Q$ on $\Zz^n$ may be viewed
as the isometric embedding
$$\iota_{\bfx}:\left.\begin{array}{ccc}
(\Zz,dx^2)&\hookrightarrow&(\Zz^n,Q)\\
n &\ra  & n\bfx
\end{array}\right..$$
Let $\mR_{Q}(q)$ be the set of such representations: The group $\Gamma=\SO_{Q}(\Zz)$ acts on $\mR_{Q}(q)$ (for $\gamma\in\Gamma$, $\gamma.\iota=\gamma\circ\iota$) and the quotient
$\Gamma\bash\mR_{Q}(q)$ is finite.

We are interested here in evaluating $|\Gamma\bash\mR_{Q}(q)|$ in the codimension one case (i.e.. when $n-m=1$). More precisely, we will need to show that, in this case, $|\Gamma\bash\mR_{Q}(q)|$
is rather small. The simplest evidence come from the case $m=1,\ n=2$ : the representations of an integer by a binary quadratic form. For instance it is well know that for $d\not=0$ the number of integral solutions to $xy = d$
(i.e. the number of divisors of $d$) is bounded by $O_{\eps}(d^{\eps})$. Similarly the number of representations of an integer as a sum of two squares satisfies the same bound; indeed, for  any binary integral quadratic form $Q$ one has $|\Gamma\bash\mR_{Q}(d)|\ll_{q}|d|^\eps$ for any $\eps>0$. The following is a version of this claim for $m=2,\ n=3$, where in the case of non-fundamental discriminants the estimate is not as strong.

\begin{Proposition}\label{2-3reps}
Let $Q$ be an integral ternary quadratic form, and let $$q(x,y) = a x^2+b  xy+c y^2$$ an integral binary quadratic form, both non-degenerate.  Assume that $f^2|\gcd(a,b,c)$ is the greatest common square divisor of $a,b,c$. Then the number $N$ of embeddings
of $(\Z^2, q)$ into $(\Z^3, Q)$, modulo the action of $\so_Q(\Z)$, is $\ll_{Q,\eps} f\max(|a|,|b|,|c|)^{\eps}$.
\end{Proposition}

When $Q=x^2+y^2+z^2$ is the "sum of three squares" quadratic form such a bound is a consequence of an explicit formula on the number of representations due to Venkov \cite{Venkov} (assuming $a$ square-free). This bound was later generalized by Pall  \cite[Thm. 5]{Pall}. We provide a self-contained treatment in Appendix \ref{Appendix-A}.

Let
\begin{align*}
\peter{(a,b,c),(a',b',c')}_{\disc}&=\disc(a+a',b+b',c+c')-\disc(a,b,c)-\disc(a',b',c')\\
&=2bb'-4ac'-4a'c
\end{align*}
be the polarization inner product associated with the quadratic form $\disc$.
We will apply Proposition \ref{2-3reps} to the pair
$$Q=\disc,\ q(x,y)=dx^2+\ell xy+dy^2,$$
and note that $q(x,y)$ is non-degenerate if an only if $\ell\neq \pm 2d$. Hence we obtain:

\begin{Corollary} \label{QF} Let $\Gamma = \mathrm{SO}_{\disc}(\Z)$.
Then for any two integers $d,\ell$ with  $\ell\not=\pm 2d$, the number of $\Gamma$-orbits on pairs
 \begin{multline*}  \bigl\{((a,b,c), (a',b',c'))\in\Z^3\times\Zz^3:\\ \disc(a,b,c)=\disc(a',b',c')
= d, \langle (a,b,c), (a',b',c') \rangle_\disc = \ell \bigr\}
\end{multline*}
is $\ll_{\eps} f ( \max(|d|,|\ell|) )^{\eps} $, where $f^2$ is the largest square factor of $\gcd(d, \ell)$.
\end{Corollary}

We now translate the information obtained about quadratic forms above to Linnik's basic lemma, which we phrase in the geometric context.  This falls short from equidistribution but will suffice
as the arithmetic input to the ergodic arguments later.

\begin{Proposition}[Basic lemma]\label{all-info-2}
We have
$$\mu_{d} \times \mu_{d} \{ (x,y) \in X_{\leq H}^2: d_{X}(x,y) \leq \delta\} \ll_\varepsilon H^4 \delta^3 d^{\varepsilon}$$
whenever $d^{-\frac{1}4}\leq\delta\leq \frac13H^{-2}$ and $\varepsilon>0$.
\end{Proposition}

Note that the exponent $3$ of $\delta^3$ is optimal, and suggests that $\mu_d$ is $3$-dimensional in the appropriate scale. The trivial exponent is $1$, which follows from $A$-invariance of $\mu_d$.

\proof
We start by indicating the relationship between $\delta$-close tuples in $(\Gd\cap X_{\leq H})^2$ and the representation of the binary quadratic form
$ q(x,y)=dx^2+\ell xy+dy^2$  by the ternary quadratic form $\disc$.

From \eqref{distancelemma}, $g_1,g_2\in\PSLdR$ are such that $x_i=\Gamma g_i\in \Gd\cap X_{\leq H}$ for $i=1,2$ and $d_{X}(x_1,x_2)<\delta$, then we may
assume
\be\label{long-list-of-prop}
g_1\in\mscS, \quad g_2 \in \mscS', \quad \Gamma g_1 \in X_{\leq H} \quad \text{and} \quad d(g_1,g_2)<\delta,
\ee
where
$\mscS'$ is some slightly bigger set containing the fundamental domain $\mscS$ in its interior. For concreteness we take
$$\mscS'=\{(z,v)\in \H\times S^1,\ |\Re z|\leq 1,\ \Im z\geq 1/2\}.$$
This clearly shows that the matrix entries of both $g_i$ are controlled, i.e.\ $\|g_i\|\ll H$ where
$$
\|g\|=\tr(g\transp{g})^{1/2}.
$$
Moreover,  we may associate to $g_i$ the primitive integral quadratic form,
$$
 q_i(x,y)=\sqrt{d}[g_i. q_0](x,y)={a_ix^2+b_ixy+c_iy^2},\ b_i^2-4a_ic_i=d,\ \gcd(a_i,b_i,c_i)=1.
$$

We have to consider two different possible cases. Either $q_1=q_2$ (i.e.\ $g_2\in g_1A$) or $q_1\neq q_2$.

The total mass for the first case is easy to estimate by $\ll_\epsilon d^{1/2+\epsilon}\delta$ before normalization by the total volume, which gives after the normalization
that
$$
 \mu_{d} \times \mu_{d} \{ (\Gamma g_1,\Gamma g_1 h) \in X_{\leq H}^2: h\in A,d(\Id,h) \leq \delta\} \ll_\varepsilon \delta d^{-1/2} d^{\varepsilon}\leq \delta^3 d^{\varepsilon}
$$
since $ d^{-1/4}\leq \delta$.

Henceforth we assume $q_1\neq q_2$. Since $\|g_i\| \ll H$, we have
\be\label{coeffquadratic}
  \max(|a_i|,|b_i|,|c_i|) \ll d^{1/2} H^2.
\ee
Also by assumption $g_2=g_1h$ with $d(h,\Id)<\delta$. This shows that $q_2=\sqrt{d}g_1. (h.q_0)$ where
$\|h.q_0-q_0\| \ll \delta$.
Therefore,
\be\label{diff-equation}
 \max(|a_1-a_2|,|b_1-b_2|,|c_1-c_2|)\ll d^{1/2}H^2\delta.
\ee
We now define $$q(u,v)=\disc(u(a_1,b_1,c_1)+v(a_2,b_2,c_2))=du^2+\ell uv+dv^2.$$
From the bound \eqref{diff-equation} on the difference of the vectors we know
$$
|q(1,-1)|=|2d-\ell | \ll d H^4\delta^2.$$

In order to apply Corollary \ref{QF} on $q$, we need to check that $q$ is not degenerate, i.e.\ that $\ell\neq\pm 2d$.
Indeed, if $\ell=\pm 2d$ then
$$
 d(a_2\mp a_1)^2=q(a_2,-a_1)=\disc\bigl(a_2(a_1,b_1,c_1)-a_1(a_2,b_2,c_2)\bigr)=(a_2b_1-a_1b_2)^2,
$$
which contradicts the assumption that $d$ is not a perfect square.
 Therefore $\ell\not=\pm 2d$. In this case we may apply Corollary  \ref{QF} to obtain
the bound
\[
N_{\ell,d}=\left|\SO_{\disc}(\Zz)\bash\{(\Zz^2,dx^2+\ell xy+dy^2)\hookrightarrow (\Zz^3,\disc)\}\right|\ll f\max(d,\ell)^\eps
\]
on the number $N_{\ell,d}$ of inequivalent ways in which the quadratic form $dx^2+\ell xy+dy^2$ can be represented, where $f^2|\gcd(d,\ell)$ is the greatest square divisor.
Note that the group $\operatorname{SO}_{\disc}$ is rationally equivalent to $\operatorname{PGL}_2$, and so up to isogeny rationally equivalent to $\SL_2$. Therefore, $\operatorname{SO}_\disc(\Zz)$ is commensurable to the image of $\Gamma=\SL_2(\Zz)$ and we may also use $\Gamma$ instead of $\SO_\disc(\Zz)$ in the above estimate.

 Let
\[
 \Gamma (q_1^{(1)},q_2^{(1)}), \ldots, \Gamma (q_1^{(k)},q_2^{(k)})
\]
be a complete list of diagonal $\Gamma$-orbits of pairs of quadratic forms which can be written as
$$
 q_i^{(j)}(x,y)=\sqrt{d}g_i^{(j)}. q_0(x,y)
$$
with $g_1^{(j)},g_2^{(j)}$ satisfying \eqref{long-list-of-prop}

The number $k$ of these diagonal $\Gamma$-orbits of quadratic forms is bounded by
 \begin{align*}
 k&\leq\sum_{\ell= 2d-L}^{2d+L}N_{\ell,d}=\sum_{f^2|d}\sideset{}{'}\sum_\stacksum{|2d-\ell|\le L}{f^2|\ell,\ \ell\neq\pm 2d} N_{\ell,d}
 \\
&\ll_\eps \sum_{f^2|d}\sideset{}{'}\sum_\stacksum{|2d-\ell|\le L}{f^2|\ell,\ \ell\neq\pm 2d}fd ^\epsilon
\ll_\epsilon\sum_{f^2|d}f\frac{d^{1+\epsilon}H^4\delta^2}{f^2}\ll_{\eps}
 d^{1+2\eps}\delta^2 H^4.
\end{align*}
where $L  \ll d H^4\delta^2$ and $\sum'$ denotes a sum over $\ell$ for which $\frac{(d,\ell)}{f^2}$ is square-free.
\medskip

\IGNORE{Observe also that, if we choose $\delta \ll d^{-1/2} H^{-2}$, so that the upper bound in \eqref{diff-equation}
is less than $1$, then we must have $q_1=q_2$. Therefore, since $q_1^{(j)}\neq q_2^{(j)}$
this shows
$$d(g_1^{(j)},g_2^{(j)})\gg d^{-1/2}H^{-2}.$$ }
We claim that for  $q_1^{(j)}\neq q_2^{(j)}$ we have
\be\label{new-lower-bound}
d(g_1^{(j)}a_{t},g_2^{(j)}A)\gg d^{-1}.
\ee
Indeed suppose $d(g_1^{(j)}a_{t},g_2^{(j)}a_{t'})\leq c d^{-1}$ (for some constant $c$ determined in a moment). Then we may find some $\gamma\in\Gamma$ with $\gamma g_1^{(j)}a_t\in\mathscr{S}$, which also implies $\gamma g_2^{(j)}a_{t'}\in\mathscr{S}'$. By Remark \ref{rem1/4} we have $\Gd \subset X_{\leq H'}$ for $H'=  d^{1/4}$. Hence by choosing $c$ appropriately
the upper bound in \eqref{diff-equation} (applied for $H'=d^{1/4}$ and $\delta=cd^{-1}$) is less than one, which gives a contradiction.

Writing $g_2=g_1\exp v$ for some $v=v^-+v^++v_A\in\mathfrak{sl}_2(\R)$, with $v^-,v^+,v_A$ eigenvectors of $\operatorname{Ad}_{a_t}$ with eigenvalues $e^{-t},e^t,1$ respectively, the estimate \eqref{new-lower-bound} implies that both $\|v^-\|, \|v^+\|\gg d^{-1}$.
It follows that for any $j$ the inequality
\be\label{set-defined}
d(g_1^{(j)} a_t, g_2^{(j)}A)<1
\ee
can hold only for $t$ in some interval $I_j$ of length $\ll \log d$.

\medskip

\noindent{\bf Claim:} For each pair $(g^{(j)}_1, g^{(j)}_2)$ there is an interval
$I_j\subset \R$
of length $\ll_\eps d^\eps$ with the following property:

 If $(x_1,x_2)\in (\Gd\cap X_{\leq H})^2$ with $d(x_1,x_2)<\delta$
 have representatives $(g_1, g_2)$ satisfying
\eqref{long-list-of-prop}
  for which the associated forms $q_i = \sqrt{d} g_i. q_0$ are different, then $x_1=\Gamma g_1^{(j)} a_{t}$ for some $j$ and some $t\in I_j$.
\medskip

Indeed,  $(\gamma.q_1,\gamma.q_2)=(q_1^{(j)},q_2^{(j)})$ for some $\gamma\in\Gamma$ and some
 $j\in[1,k]$ and so $g_1=\gamma^{-1}g_i^{(j)}a_{t}$ resp.\ $g_2\in\gamma^{-1}g_2^{(j)}A$.
By assumption on $g_1,g_2$ we have $d(g_1^{(j)}a_{t},g_2^{(j)}A)<\delta$.

Using the claim and a fixed Haar measure of $A$ (i.e.\ before normalization) we get that the measure of the collection of points $(x_1,x_2)\in(\Gd\cap X_{\le H})^2$, which can be represented as $x_i=\Gamma g_i$ with $g_i$ as in \eqref{long-list-of-prop} and for which the associated quadratic forms are different, is
\[
   \ll \sum_{j=1}^k |I_j| \delta\ll_\epsilon d^\eps \delta k \ll_\eps  d^{1+2\eps} H^4 \delta^3.
\]
Therefore, by dividing the above by the total volume of $(\Gd)^2$, the claim (together with the analysis of the case $q_1=q_2$)   implies the proposition.
 \qed

\section{An ergodic theoretic proof of Duke's theorem}\label{ergodicsection}
\subsection{Entropy and the unique measure of maximal entropy.}

A basic underlying concept in our proof is that of entropy. We recall that if $\Pcr$ is a partition of the probability space $(X, \nu)$,
the entropy of $\Pcr$ is defined as $$H_{\nu}(\Pcr) := \sum_{S \in \Pcr} - \nu(S) \log \nu(S).$$
It is clear that $H_\nu(\Pcr)=H_\nu(T^{-1}\Pcr)$ if $T:X\rightarrow X$ preserves $\nu$ --- below we will use
this fact without explicit reference.
We note for future reference that entropy is controlled by an $L^2$-norm
\begin{equation} \label{Entropy-L2}
    H_{\nu}(\Pcr) \geq  -\log \left( \sum_{S \in \Pcr}  \nu(S)^2 \right)
\end{equation}
as one easily sees from convexity of the logarithm map.
Moreover, entropy has the following basic
subadditivity property: if $\Pcr_1, \Pcr_2$ are two partitions, then
\begin{equation}\label{subadditivity}
 H_{\nu}(\Pcr_1 \vee \Pcr_2) \leq H_{\nu}(\Pcr_1) + H_{\nu}(\Pcr_2),
\end{equation}
where $\vee$ denotes common refinement.

If $T$ is a measure-preserving transformation of $(X,\nu)$, then the measure theoretic entropy
of $T$ is defined as:
\begin{equation} \label{entropydef}
 h_\nu(T) = \sup_{\Pcr} \lim_{n\rightarrow \infty} \frac{H_{\nu}(\Pcr \vee T^{-1} \Pcr \vee \dots \vee T^{-(n-1)} \Pcr)}{n}
\end{equation}
where the supremum
is taken over all finite partitions of $X$.
We also note that the limit in the definition exists and is equal to the infimum because the sequence
\[
a_n =  H_{\nu}(\Pcr \vee T^{-1} \Pcr \vee \dots \vee T^{-(n-1)} \Pcr)
\]
is subadditive (i.e.\ $a_{n+m}\leq a_n + a_m$).

A key role in our argument is played by the fact that the uniform measure on $\Gamma \backslash \SL _ 2 (\R)$ for any lattice $\Gamma$ can be distinguished using entropy, as it is the \emph{unique} measure of maximal entropy:

\begin{Theorem} \label{max-entropy-main}
 Let $X=\Gamma\backslash\SL_2(\R)$ be a quotient by a lattice $\Gamma<\SL_2(\R)$, and let $T$
 denote the time-one-map of the geodesic flow, i.e.\ right translation \[T(x)=x\begin{pmatrix}e^{1/2}&0\\0&e^{-1/2}\end{pmatrix}.\] Then for any invariant measure $\nu$ the entropy satisfies
 $h_\nu(T)\leq 1$ where equality holds if and only if $\nu=\mu_X$ is the $\SL_2(\R)$-invariant
 probability measure on $X$.
\end{Theorem}

The inequality $h_\nu(T)\leq 1$ is not hard and can be proved in many ways. Identifying the uniform measure as the unique measure where this maximum is attained is somewhat more delicate. We give a self-contained treatment in Appendix~\ref{Appendix-B}.

\subsection{Proof of Duke's theorem, an outline}\label{X compact}

Let  $T:X\rightarrow X$ denote the time-one-map of the geodesic flow as in Theorem~\ref{max-entropy-main}.
Recall that
\begin{eqnarray*}
 U^-&=&\biggl\{\begin{pmatrix} 1 & t \\ 0 & 1\end{pmatrix}: t\in\R\biggr\}\mbox{ resp.}\\
 U^+&=&\biggl\{\begin{pmatrix} 1 & 0 \\ t & 1\end{pmatrix}: t\in\R\biggr\}
\end{eqnarray*}
are the stable, resp.\ unstable horocycle subgroups. The orbits of these two subgroups give the foliation into stable
and unstable manifolds in the following sense. If $u=u(t)\in U^-$, then the distance between $T^n(x)$ and $T^n(xu)$ converges rapidly to zero:
 \begin{multline*}
 d(T^n(x),T^n(xu))=d\biggl(x\begin{pmatrix}e^{n/2}&0\\0&e^{-n/2}\end{pmatrix},xu\begin{pmatrix}e^{n/2}&0\\0&e^{-n/2}\end{pmatrix}\biggr)\\
 \leq d\biggl(\begin{pmatrix}1&0\\0&1\end{pmatrix}, \begin{pmatrix}e^{-n/2}&0\\0&e^{n/2}\end{pmatrix} u \begin{pmatrix}e^{n/2}&0\\0&e^{-n/2}\end{pmatrix}\biggr)
 \\ = d\biggl(\begin{pmatrix}1&0\\0&1\end{pmatrix}, \begin{pmatrix} 1 & e^{-n}t \\ 0 & 1\end{pmatrix}\biggr).
\end{multline*}

To give an outline of our argument, it is perhaps preferable to simplify the situation. In our proof, the noncompact nature of our space $X$ is a significant complication, so instead of considering the quotient $\SL _ 2 (\R) \backslash \SL _ 2 (\R)$ for the purposes of this outline let us consider a compact quotient $\hat X = \Gamma \backslash \SL _ 2 (\R)$ on which we have a sequence of $T$-invariant probability measures $\mu _ d$ satisfying the following simplified version of the conclusion of Corollary~\ref{all-info-2}
\begin{equation} \label{simplified Linnik lemma}
\mu_{d} \times \mu_{d} \{ (x,y) \in \hat X ^2: d_{\hat X}(x,y)
\leq \delta \} \ll_\varepsilon \delta^3 d^{\varepsilon} \qquad \text{for $\delta > d ^ {- 1 / 4}$}.
\end{equation}

Let $r>0$ be an injectivity radius of $\hat X$ so that for any $x \in \hat X$ the map $B_r^{G}(e)\rightarrow \hat X$  sending $g$ to $xg$ is injective (with $G={\SL_2(\R)}$, and $B_r^G$ denoting a ball of radius $r$ in $G$).
Also assume $\eta<\frac 1er$ is small enough so that $B_\eta^{G}(e)$ is an injective
image under the exponential map of a neighborhood of $0$ in the Lie algebra.

Let $\Pcr$ be a finite measurable partition all of whose elements have ``diameter smaller than $\eta$'', i.e.  if  $x$ and $y=xg$ with $g\in B_r^{G}$ belong to the same element of $\Pcr$, then $g\in B_\eta^{G}$. Assume that
the same holds as well for $T^i(x)$ and $T^i(y)$ for $i=-N,\ldots, 0,1,\ldots,N$.
Then $d(T(x),T(y))<\eta$ and $d(e,a^{-1}ga)<r$ so that $a^{-1}ga\in B_\eta^{G}(e)$.
Repeating this implies that
\[
  g\in B_N=\bigcap_{n=-N}^N \begin{pmatrix}e^{1/2}&\\&e^{-1/2}\end{pmatrix}^{-n}
  B_\eta^{G}(e)\begin{pmatrix}e^{1/2}&\\&e^{-1/2}\end{pmatrix}^n.
\]
We define a Bowen $N$-ball to be the translate $xB_N$ for some $x\in X$.

Notice that the set $B_N$ is ``tube-like'': it has
width at most $e^{-N}\eta$ along the stable and unstable directions, but is
of length $\eta$ in the direction $A$ of the geodesic flow.
The above shows that every element of the partition
\begin{equation}\label{refined partition equation}
\Pcr^{[-N,N]}=\bigvee_{n=-N}^NT^{-n}\Pcr
\end{equation}
is contained in a single Bowen $N$-ball. Together we conclude that
$$
\bigcup_{S \in  \Pcr^{[-N,N]}} S \times S \subset \bigcup_{i=1}^k \{(x,ya_i):d(x,y)<re^{-N}\}
$$
where $k \ll e^{N}$ and $a_1,\ldots,a_k \in B_r^A(1)$ are chosen to be $\delta$-dense --
that is to say, the union of the $\delta$-neighbourhoods around $a_i$ cover
$B_r^A(1)$.

Together with \eqref{simplified Linnik lemma} this shows that
$$
\sum_{S \in  \Pcr^{[-N,N]}}\mu_d(S)^2 \ll_\varepsilon e^{-2N}d^\varepsilon
$$
whenever $\delta=\eta e^{-N}\geq d^{-\frac{1}{4}}$ or equivalently $N \leq \frac{1}{4}\log d+\log r$.
We choose $N=\lfloor \frac{1}{5}\log d\rfloor$
(the ``extra space'' will be useful in supressing a $d^\varepsilon$).
Using \eqref{Entropy-L2} we have
$$
 H_{\mu_d}\bigl( \Pcr^{[-N,N]}\bigr)\geq (2-6\varepsilon)N
$$
for large enough $d$.

In this statement we cannot yet let $d\rightarrow\infty$ to get
a statement about a weak$^*$ limit $\mu$, because $N$ is a function of $d$, and so
the size of $\Pcr^{[-N,N]}$ increases with $d$.  Thus let $N_0\geq 1$
be any fixed integer:  $[-N,N]$ can be covered by $\lceil\frac{N}{N_0}\rceil$ many
translates of $[-N_0,N_0]$. This in turn shows that $\Pcr^{[-N,N]}$ can be obtained
as a refinement of the $\lceil\frac{N}{N_0}\rceil$ partitions
\[
\Pcr^{[-N,-N+2N_0]}, \Pcr^{[-N + 2 N _ 0,-N+4N_0]}, \dots
\]
(in the obvious generalization of the notation \eqref{refined partition equation}).
By subadditivity \eqref{subadditivity} (and invariance) this implies
\[
  H_{\mu_d}\bigl( \Pcr^{[-N_0,N_0]}\bigr)\geq (2-7\varepsilon)N_0
\]
for large enough $d$.
By choosing the original partition $\Pcr$ such that $\mu(\partial S)=0$ for all $S\in \Pcr$
and some weak$^*$ limit $\mu$ of the sequence $\mu_d$ we can now take the limit as $d\rightarrow\infty$
to obtain
$$
  H_{\mu}\bigl( \Pcr^{[-N_0,N_0]}\bigr)\geq (2-7\varepsilon)N_0\mbox{ for all $\varepsilon>0$ and $N_0\geq 1$},
$$
i.e.\ that $h_\mu(T)\geq 1$. Theorem \ref{max-entropy-main} can now be invoked to show that $\mu$ must be the  $\SL_2(\R)$-invariant measure on $X$.

We remark that the analysis above works only in the cocompact case; for e.g.
$\Gamma = \SLdZ$, there is no global injectivity radius; and no matter how fine one takes the partition $\Pcr$, to cover a single atom of the partition $\Pcr^{[-N,N]}$ one typically needs exponentially many Bowen $N$-balls.

\subsection{Proof of Duke's theorem, controlling the time spent near the cusp.}\label{non-compact}
Passing from the cocompact to the nonuniform case raises two difficulties:
\begin{itemize}
\item[(i)] Why is such a weak$^*$ limit a probability measure (indeed, why can't such a sequence of measures $\mu _ d$ converge to the zero measure)?  \item[(ii)] The proof outline presented in \S\ref{X compact} used heavily the relation between Bowen $N$-balls and atoms of the partition $\Pcr^{[-N,N]}$ for a finite partition $\Pcr$. How can we adapt  this argument to the nonuniform situation where  in general many Bowen $N$-balls are needed to cover a partition element $S \in\Pcr^{[-N,N]}$?
\end{itemize}

It turns out that these two difficulties are not unrelated, and to handle them
one needs
to control the time an orbit spends in the neighborhood of the cusp, so that this problem
is related to controlling the escape of mass.
What is needed is the following finitary version of the uniqueness of measure of maximal entropy:

\begin{Theorem} \label{Measure}
Suppose $\mu_i$ is a sequence of $A$-invariant measures on $X$, and suppose
there is a a constant $r>0$ and a sequence $\delta_i \rightarrow 0 $ such that for all sufficiently small $\varepsilon>0$
the ``heights'' $H_i=\delta_i^{-\varepsilon}$ satisfy
\begin{enumerate}
\item $\mu_i ( X_{\geq H_i}) \rightarrow 0$, as $i \rightarrow \infty$;
\item $\mu_i \times \mu_i (\{(x,y) \in X_{\leq H_i} \times X_{\leq H_i}: d(x,y)<\delta_i \}
    \ll_{\varepsilon} \delta_i^{3- 5\varepsilon}$.
\end{enumerate}
Then $\mu_i \rightarrow \mu_X$, the $\SL_2(\R)$-invariant measure on $X$, as $i \rightarrow \infty$.
\end{Theorem}

Clearly, this, Proposition \ref{all-info-1}, and Proposition \ref{all-info-2} with $\delta=d^{-\frac{1}{4}}$
are sufficient to prove Duke's theorem.   Apart
from the ideas already discussed in the last section, the main additional step
is:

\begin{Proposition}\label{cover-lemma}
 Fix a height $M\geq 1$. Let $N\geq 1$ and
 consider a subset $V\subset[-N,N]$.
 Then the set
 \begin{multline*}
  Z(V)=\Bigl\{x\in T^N X_{< M}\cap T^{-N} X_{< M}: \mbox{ for all $n\in[-N,N]$ we have}\\
   T^n(x)\in X_{\geq M}\Leftrightarrow n\in V\Bigr\}
 \end{multline*}
 can be covered by $\ll_{M}e^{2N-\frac{1}{2}|V|}$ Bowen $N$-balls.
 Moreover, $Z(V)$ is nonempty for only  $\ll_M e^{\frac{2\log\log M}{\log M}N}$
 different sets $V \subset [-N,N]$.
\end{Proposition}

In words, $Z (V)$ is the set of points $x \in X$ so that their trajectory $T ^ {- N} x$, $T ^ {- N + 1} x$, \dots, $T ^ N x$ between times $- N$ and $N$ begins and ends below height $M$ and are above height $M$ precisely at the time specified by the set $V$.
So the content of the Proposition is that orbits that spend a lot of time in a neighborhood of the cusp in fact
can be covered by relatively few tube-like sets.
Later we will turn this
into the statement that those orbits have relatively little mass.

Note that as the size of $V$ grows the number of Bowen $N$-balls needed to cover $Z (V)$ decreases, though even if $V=[-N-1,N+1]$  it is still exponential --- indeed $\asymp e^N$, which is essentially the square root of the estimate we get for $V = \emptyset$ .

We defer the proof of the Proposition \ref{cover-lemma}
to the next section.
A purely ergodic theoretic formulation of this phenomena is that a lot of mass near the cusp
for an invariant probability measure results in a significantly smaller entropy for the geodesic flow.
We will give such a formulation in Theorem \ref{entropy-cusp}; it implies in particular that:

\begin{quote}
Given a sequence $T$-invariant probability measures $\mu_i$ with entropies $h_{\mu_i}(T)\geq c$, any weak   weak$^*$ limit $\mu$ satisfies $\mu(X)\geq 2c-1$.
\end{quote}

We will discuss in Remark \ref{exp-1/2} why $c=1/2$ is the critical point for this phenomenon.

\subsection{Controlling escape of mass, and maximal entropy}
We proceed to the proof of Theorem~\ref{Measure}, and start by showing that mass cannot escape, using assumption (2). We will use (1) of that theorem which gives a mild control on how fast mass could possibly escape to be able to apply the covering argument in Proposition \ref{cover-lemma}.
That (2) can replace entropy in that argument is not surprising since we have already seen in Section \ref{X compact}
a relationship between this  assumption and entropy.

\begin{Lemma}\label{non-escape}
 Let $\mu_i$ be a sequence of $T$-invariant measures as in Theorem~\ref{Measure}. Let $\mu$
 be a weak$^*$ limit of any subsequence of $\mu_i$. Then
 \[
  \mu(X_{<M})\geq 1-\frac{2\log\log M}{\log M}
 \]
 for every sufficiently large $M$, and so $\mu$ is a probability measure.
\end{Lemma}

\begin{proof}
 Fix some $\kappa>\frac{2\log\log M}{\log M}$. We will show that $\mu(X_{<M})\geq 1-\kappa$.

 We set $N_i=\lceil -\log \delta_i\rceil$ and $H_i=\delta_i^{-\epsilon}$ for some $\epsilon>0$
 determined below (more precisely:  before the final displayed equation of this proof) in terms of $\kappa$. Notice that a geodesic trajectory of a point $x\in X_{\leq H_i}$ will visit
 $X_{<M}$ in less than $2\log H_i-2\log M\leq2\epsilon N_i$ steps either in the future or in the
 past. Hence
 \[
  \bigcup_{n=-\lfloor2\epsilon N_i\rfloor}^{\lfloor2\epsilon N_i\rfloor}T^{-n}X_{<M}\supset X_{\leq H_i}
 \]
 and so this union contains most of the $\mu_i$-mass according to the assumption (1)
 of Theorem ~\ref{Measure}.

 Let $N'_i=N_i+\lfloor2\epsilon N_i\rfloor$.
 Then $T^{N_i'}X_{\leq H_i}\cap T^{-N_i'}X_{\leq H_i}$ is contained in the union of $\ll (\epsilon N_i)^2$
  many sets of the form
 $T^{N_i'+n_-}X_{<M}\cap T^{-N_i'+n_+}X_{<M}$ where $|n_-|,|n_+|\leq 2\epsilon N_i$.
 We apply this to the set
 \[
  X_\kappa=\biggl\{x\in T^{N_i'}X_{\leq H_i}\cap T^{-N_i'}X_{\leq H_i}:
  \frac{1}{2N_i'+1}\sum_{n=-N'_i}^{N_i'}1_{X_{\geq M}}(T^nx)>\kappa\biggr\}
 \]
 consisting of points that spend an unexpected high portion of $[-N_i',N_i']$ above $M$.

 We wish to estimate $\mu_i(X_\kappa)$. $X_\kappa$ is also a union of sets of the form
 \[Z'=X_\kappa\cap T^{N_i'+n_-}X_{<M}\cap T^{-N_i'+n_+}X_{<M}\]
 with $n_-,n_+$ as before.  It suffices to estimate $\mu_i(Z')$
 for some fixed $n_-, n_+$.   Replacing $Z'$ by an appropriate
 shift $Z:=T^k Z'$ we may consider instead $Z\subset T^{N}X_{<M}\cap T^{-N}X_{<M}$ where $N\in [N_i,N_i+4\epsilon N_i]$.
 Adjusting the condition on the ``average time spent above $M$'' appropriately,
 \[
  Z\subseteq\biggl\{x\in T^{N}X_{<M}\cap T^{-N}X_{<M}:
   \frac{1}{2N+1} \sum_{n=-N}^{N}1_{X_{\geq M}}(T^nx)>\kappa-O(\epsilon)\biggr\}.
 \]
To the right-hand set we apply Proposition \ref{cover-lemma}; which
shows that $Z$ is covered   by
 \[
  \ell\ll_M e^{\frac{2\log\log M}{\log M}N}e^{2N-(\kappa-O(\epsilon))N}\leq
   e^{2N_i+\frac{2\log\log M}{\log M}N_i-\kappa N_i+O(\epsilon) N_i}
 \]
 many Bowen $N$-balls. Because $N \geq N_i$,
we may also cover $Z$ by $\ell$ many Bowen $N_i$-balls $S_1,\ldots,S_\ell$.

 Since Bowen $N_i$-balls have thickness $\leq e^{-N_i}\leq\delta_i$ along stable and unstable
 horocycle directions and thickness $\ll 1$ along $A$, we get that
 \[
  \bigcup_{j=1}^\ell S_j\times S_j\subset\bigcup_{j=1}^k\{(x,ya_j):d(x,y)<\delta_i\}
 \]
 where $k\ll e^{N_i}$ and $a_j\in B_1^A$ are $\delta_i$-dense. This remains true if we
 make the sets $S_j$ disjoint by replacing $S_2$ by $S_2'=S_2\setminus S_1$, $S_3$ by $S_3'=S_3\setminus (S_1\cup S_2)$, \ldots.
 By our assumption (2) we now get
 \[
  \sum_{j=1}^\ell \mu_i(S_j')^2\ll_\epsilon \delta_i^{3-5\epsilon}k\ll e^{-2N_i+5\epsilon N_i}.
 \]
 Therefore, by Cauchy-Schwarz
 \[
  \mu_i(Z)\leq\sum_{j=1}^\ell \mu_i(S_j') \leq \biggl(\sum_{j=1}^\ell \mu(S_j')^2\biggr)^{1/2}\ell^{1/2}\ll_{\epsilon,M}
  e^{\frac{\log\log M}{\log M}N_i-\frac12\kappa N_i+O(\epsilon) N_i}
 \]
 Going through all possibilities for $n_-,n_+$ (of which there are $\ll e^{\epsilon N_i}$ many) this implies
 $$
  \mu_i(X_\kappa)\ll_{\epsilon,M} e^{\bigl(\frac{\log\log M}{\log M}-\frac{1}{2}\kappa +O(\epsilon)\bigr) N_i}.
 $$
 Given that we assume $\kappa>\frac{2\log\log M}{\log M}$ we can choose $\epsilon>0$ small enough
 such that the exponent in the above expression is negative so that the measure goes to zero for
 $i\rightarrow\infty$ (since $N_i\rightarrow\infty$).
 By definition of $X_\kappa$ we have
 \[
  \mu_i(X_{\geq M})=\int 1_{X_{\geq M}}\operatorname{d}\!\mu_i=
  \int\frac{1}{2N_i'+1}\sum_{n=-N_i'}^{N_i'}1_{X_{\geq M}} \operatorname{d}\!\mu_i\leq \kappa+\mu_i(X_\kappa)+2\mu_i(X_{\geq H_i}),
 \]
 which when $i\rightarrow\infty$ implies that $\mu(X_{<M})\geq 1-\kappa$ for any $\kappa>\frac{2\log\log M}{\log M}$.
 This gives the lemma.
\end{proof}

We indicated in Section \ref{X compact} how the elements of the refinement $\bigvee_{n=-N}^NT^{-n}\Pcr$ are related to Bowen $N$-ball; but that
analysis fails in the noncompact case, when trajectories visit the cusp. We now discuss the general case.

\begin{Lemma}\label{partition}
 For every $M>1$ there exists a finite partition $\Pcr$ of $X$ such that for
 every $\kappa\in (0,1)$ and every $N$, ``most elements of the refinement $\bigvee_{n=-N}^NT^{-n}\Pcr$ are controlled
 by Bowen $N$-balls'':

  There exists a set $X'\subset X$ so that:
  \begin{itemize}
  \item[-]
  $X'$  is a union of
 $S_1,\ldots, S_\ell\in \bigvee_{n=-N}^NT^{-n}\Pcr$;
 \item[-] Each such $S_j$ is contained in a union of at most
 $3^{\kappa (2N+1)}$ many Bowen $N$-balls;
 \item[-] $\mu(X')\geq 1-2\mu(X_{\geq M})\kappa^{-1}$
 for every invariant probability measure $\mu$;
 \end{itemize}
For a given $\mu$ the choice of $\Pcr$ can be made such
 that the boundaries of all sets of $\Pcr$ have zero measure.
\end{Lemma}

\begin{proof}
  We define $\Pcr=\{Q,P_1,\ldots,P_k\}$ where
 $Q=X_{\geq M}$ and $\{P_1,\ldots,P_k\}$ is a measurable partition of $X_{<M}$
 whose elements have diameter less than $\eta$ where $\eta$ is small enough in comparison
 to the injectivity radius of $X_{<M}$ (in the same sense as in the discussion in Section~\ref{X compact}).

Note that the boundary of $Q$ is a null set for every probability measure $\mu$ that is invariant under the geodesic flow. This is because every trajectory hits the boundary of $Q$ in a countable set. Also, given $\mu$ we can find for every point $x\in X_{<M}$ an $\epsilon<\eta/2$ so that the boundary has measure zero.  Applying compactness we construct $P_1,\ldots,P_k$  from the algebra generated by finitely many such balls.

 We claim that $S\in \Pcr_N=\bigvee_{n=-N}^NT^{-n}\Pcr$ has the
 property that any two points $x,y\in S$ satisfy
 $$
  T^nx \in X_{<M}\Leftrightarrow T^n y\in X_{<M}\mbox{ for }n\in[-N,N]
 $$
 and
 $$
  d(T^nx,T^ny)<\eta\mbox{ whenever }T^nx,T^ny\in X_{<M}\mbox{ and }n\in[-N,N].
 $$
 Therefore, the average $f(x)=\frac{1}{2N+1}\sum_{n=-N}^N 1_{X_{\geq M}(T^n x)}$ is constant on sets of $\Pcr_N$.
 We define $$X'=\{x\in T^{-N}X_{<M}:f(x)\leq \kappa\}.$$

 If $\mu$ is an invariant probability measure, invariance
 implies $\int f(x) \operatorname{d}\!\mu=\mu(X_{\geq M})$ and so
 $\mu(\{x:f(x)>\kappa\})\leq\mu(X_{\geq M})\kappa^{-1}$. Therefore,
 $X'$ has measure $\mu(X')\geq 1-\mu(X_{\geq M})-\mu(X_{\geq M})\kappa^{-1}$.

 Consider now an element $S\in \Pcr_N$ with $S\subset X'$. After taking the image of $S$ under $T^N$
 we have for any $x,y\in S'=T^N S$ that
 \begin{equation}\begin{split}
  x\in X_{<M},\ \frac{1}{2N+1}\sum_{n=0}^{2N}1_{X_{\geq M}}(T^n x)\leq\kappa\mbox{ and}\\
  d(T^nx,T^n y)<\eta\mbox{ whenever }T^nx,T^ny\in X_{<M}\mbox{ and }n\in[0,2N].
 \end{split}\end{equation}
 Let $V=\{n\in[0,2N]:T^nS'\subset X_{\geq M}\}$. We can now show inductively
 that for every $n\in [0,2N]$
 the set $S'$ is contained in a union of $3^{|[0,n-1]\cap V|}$
 many sets of the form
 \[
  xB_{2\eta e^{-n}}^{U^+}B_{2\eta}^{U^-A}\mbox{ where }x\in S'.
 \]
 We will refer to these sets as forward Bowen $n$-balls and to $x$ as its center. For $n=0$ we have nothing
 to show (for notice that we allowed a bigger radius in the subgroups $U^+$ and $U^-A$).
 Suppose the claim holds for some $n$ and let $x\in S'$ be a center of one of the forward
 Bowen $n$-balls. If $T^{n+1}x\in X_{<M}$ then $T^{n+1}S'\subset P_i$ for $i\geq 1$
 and it follows easily that any point $y=xu^+g\in S'$ with $u^+\in B_{2\eta e^{-n}}^{U^+}$
 and $g\in B_{2\eta}^{U^-A}$ satisfies $u^+\in B_{2\eta e^{-(n+1)}}^{U^+}$ (assuming again
 that $\eta$ is small enough in comparison with the injectivity radius). If $T^{n+1}x\in X_{\geq M}$
 then we can cover the forward Bowen $n$-ball by $3$ forward Bowen $(n+1)$-balls.

 Recall that for $S\subset X'$ we have $|V|\leq\kappa N$ and so by
 taking the preimages of $S'=T^{N}S$ and the forward Bowen $2N$-balls obtained the lemma follows.

\end{proof}

To prove Theorem~\ref{Measure} it remains to establish the following lemma and combine it with
Lemma \ref{non-escape} and Theorem \ref{max-entropy-main}.

\begin{Lemma}
 A weak$^*$ limit $\mu$ of a subsequence of the invariant probability measures $\mu_i$ as in Theorem~\ref{Measure}
 has maximal entropy $h_\mu(T)=1$.
\end{Lemma}

\begin{proof}
 Let $\Pcr$ be as in Lemma \ref{partition}. Set $N_i=\lceil-\log\delta_i\rceil$ and define
$$
\Pcr_{N_i}=\bigvee_{n=-N_i}^{N_i} T^{-n}\Pcr.
$$ We wish to show that $H_{\mu_i}(\Pcr_{N_i})$
 is large by using Lemma \ref{partition} and assumption (2). Let $\kappa=\mu(X_{\geq M})^{1/2}$
 for some weak$^*$ limit $\mu$
 and define $X_i$ as in Lemma \ref{partition} using $N=N_i$.

 For any $S\in\Pcr_{N_i}$
 with $S\subset X_i$ there exists a cover of $S$ consisting of $\leq 3^{\kappa( 2N_i+1)}$
 many Bowen $N_i$-balls;  so there is a partition $\mathscr{R}(S)$
 of $S$ into $\leq 3^{\kappa (2N_i+1)}$ sets, each a  subset of a Bowen $N_i$-ball. We define the
 partition $\mathcal{Q}_i$ as the partition consisting of all $S\in\Pcr_{N_i}$
 with $S\subset X\setminus X_i$ and all elements of $\mathscr{R}(S)$ for any $S \subset X_i$.
 It follows that
 \begin{equation}\label{QPcond}
  H_{\mu_i}(\mathcal{Q}_i|\Pcr_{N_i})=\sum_{S\in\Pcr_{N_i},S\subset X_i}
  \mu_i(S)H_{\mu_i|_S}(\mathcal{Q}_i)\leq\kappa (2N_i+1)\log 3.
 \end{equation}
 Also since $\mathcal{Q}_i$ is a finer partition than $\Pcr_{N_i}$
 we have
 \begin{equation}\label{addition}
  H_{\mu_i}(\mathcal{Q}_i)=H_{\mu_i}(\mathcal{Q}_i\vee\Pcr_{N_i})=H_{\mu_i}(\Pcr_{N_i})+H_{\mu_i}(\mathcal{Q}_i|\Pcr_{N_i}),
 \end{equation}
 which together with \eqref{QPcond} indicates that we wish to show that $H_{\mu_i}(\mathcal{Q}_i)$ is large.

 Here we will use the assumption (2) from Theorem \ref{Measure};
 but the elements of $\mathcal{Q}_i$ that lie outside $X_i$ can be irregularly shaped,
 requiring a further estimate:
 \begin{equation}\label{addition2}
  H_{\mu_i}(\mathcal{Q}_i)\geq H_{\mu_i}(\mathcal{Q}_i|\{X_i,X\setminus X_i\})\geq\mu_i(X_i)H_{\mu_i|_{X_i}}(\mathcal{Q}_i).
 \end{equation}
 Using \eqref{Entropy-L2} for the restriction $\mu_i|_{X_i}$ we see that
 \begin{equation}\label{entl2}
   H_{\mu_i|_{X_i}}(\mathcal{Q}_i)\geq -\log\sum_{S\in\mathcal{Q}_i, S\subset X_i}\biggr(\frac{\mu(S)}{\mu(X_i)}\biggr)^2.
 \end{equation}
 By construction of $\mathcal{Q}_i$ every $S\in\mathcal{Q}_i$ with $S\subset X_i$ is a subset of a Bowen $N_i$-ball.
 Proceeding as in Section \ref{X compact} it follows that
$$
\bigcup_{S\in \mathcal{Q}_i, S\subset X_i} S\times S\subset\bigcup_{i=1}^k\{(x,ya_i):d(x,y)<\delta_i\}
$$
where $k\ll e^{N_i}$ and $a_1,\ldots,a_k \in B_r^A(1)$ are chosen to be $\delta_i$-dense.
Together with assumption (2) of Theorem~\ref{Measure} this shows
$$
 \sum_{S\in  \mathcal{Q}_i, S\subset X_i}\mu_i(S)^2\ll_\varepsilon \delta_i^{3-5\varepsilon}e^{N_i}\ll e^{(-2+5\varepsilon)N_i}.
$$
Let $C_{\epsilon}$ be the implicit constant here, that is to say,
$$
\sum_{S\in  \mathcal{Q}_i, S\subset X_i}\mu_i(S)^2 \leq C_{\epsilon} e^{-(2+5 \varepsilon)N_i}.
$$
Then, taking into account \eqref{addition2}--\eqref{entl2},
\[
 H_{\mu_i}(\mathcal{Q}_i)\geq 2\mu_i(X_i)\log\mu_i(X_i)-\mu_i(X_i)\log C_\epsilon +\mu_i(X_i) (2-5\epsilon)N_i.
\]
Here the first two terms are bounded, so for large enough $i$
\begin{eqnarray*}
 H_{\mu_i}(\mathcal{Q}_i)&\geq&\mu_i(X_i) (2-6\epsilon)N_i\\
 &\geq&(1-2\kappa^{-1}\mu_i(X_{\geq M}))(2-6\epsilon)N_i
\end{eqnarray*}
where we also used the estimate for $X_i$ in Lemma \ref{partition}.
Combining this with \eqref{addition} and \eqref{QPcond} we get
\[
 H_{\mu_i}\biggl(\bigvee_{n=-N_i}^{N_i} T^{-n}\Pcr\biggr)\geq (1-2\kappa^{-1}\mu_i(X_{\geq M}))(2-6\epsilon)N_i-O(\kappa N_i).
\]
Now fix some integer $N_0\geq 1$.
Using subadditivity of entropy we have for any large enough $i$ that
\[
H_{\mu_i}\biggl(\bigvee_{n=-N_0}^{N_0} T^{-n}\Pcr\biggr)\geq (1-2\kappa^{-1}\mu_i(X_{\geq M}))(2-6\epsilon)N_0-O(\kappa N_0)-\epsilon N_0.
\]
This is now a statement involving only finitely many test function, namely the characteristic
functions of all elements of $\bigvee_{n=-N_0}^{N_0} T^{-n}\Pcr$ and of $X_{\geq M}$. Since
there is no escape of mass by Lemma \ref{non-escape} and since we can assume without loss of
generality that all boundaries have zero measure for the weak$^*$ limit $\mu$ by Lemma \ref{partition},
 we get the same estimate for $\mu$. Dividing by $2N_0$ and letting $N_0$ now go to infinity we arrive
 at
 \[
  h_\mu(T)\geq (1-2\mu(X_{\geq M})^{1/2})(1-3\epsilon)-O(\mu(X_{\geq M})^{1/2})-\epsilon
 \]
 for any $M\geq 1$ and $\epsilon>0$.

  Since $\mu(X_{\geq M})$ can be made arbitrarily small,
 it follows that $h_\mu(T)\geq 1$, i.e.\ $T$ has maximal entropy.
\end{proof}

\section[Trajectories spending time high in the cusp]{Trajectories spending time high in the cusp, and a proof of Proposition~\ref{cover-lemma}.}\label{high in cusp}

Apart from the characterization of the Haar measure as the unique measure of maximal entropy in Theorem \ref{max-entropy-main}
the main technical estimate needed to prove Theorem \ref{Measure}
 is Proposition \ref{cover-lemma}. We recall that this proposition states
that the set
 \begin{multline*}
  Z(V)=\Bigl\{x\in T^N X_{< M}\cap T^{-N} X_{< M}: \mbox{ for all $n\in[-N,N]$ we have}\\
   T^n(x)\in X_{\geq M}\Leftrightarrow n\in V\Bigr\}
 \end{multline*}
can be covered by $\ll_{M}e^{2N-\frac{1}{2}|V|}$ Bowen $N$-balls.

In addition to proving this, we shall also prove here the promised purely ergodic formulation
of ``high entropy inhibits escape of mass,'' namely:

\begin{Theorem}\label{entropy-cusp}
Let $T$ be the time-one-map for the geodesic flow.
There exists some $M_0$ with the property that
 $$
  h_\mu(T)\leq 1+\frac{\log\log M}{\log M}-\frac{\mu(X_{\geq M})}{2}
 $$
for any
invariant probability measure $\mu$ on $X=\SL(2,\Z)\backslash\SL(2,\R)$ for the geodesic flow and any $M\geq M_0$.
 In particular, for a sequence of $T$-invariant probability measures $\mu_i$ with entropies $h_{\mu_i}(T)\geq c$, any weak$^*$ limit $\mu$ satisfies $\mu(X)\geq 2c-1$.
\end{Theorem}

\medskip
\begin{rem}\label{exp-1/2}
	Roughly speaking $1/2$ is the critical point for Theorem \ref{entropy-cusp} because the ``upward'' and ``downward'' parts of a trajectory, that goes high in the cusp, are strongly related to each other. In fact, in the case of a $p$-adic flow this phenomenon is easy to explain.

We consider another dynamical system of similar flavor: here the space will be\footnote {For technical reasons, it is preferable to use $\PGL _ 2$ here rather than $\SL _ 2$.} \[Y=\PGL _ 2 (\Z [1 / p]) \backslash \PGL _ 2 (\R) \times \PGL _2 (\Q _ p) \] and the action will be by multiplication on the right of the $\PGL _2 (\Q _ p)$-component by $a_p=\begin{pmatrix} p& \\& 1 \end{pmatrix}$. Let $M <  \PGL _ 2 (\R) \times \PGL _2 (\Q _ p) $ be the product of $\mathrm{PO}_2(\R)$ and the group of diagonal matrices in $\PGL _ 2 (\Z _ p)$.  There is a natural right $M$-invariant projection $\pi: Y \to  \PSL _ 2 (\Z) \backslash \H $ , and on this latter space we have the Hecke correspondence which attaches to a point $\dot z \in \PSL _ 2 (\Z) \backslash \H$ a set $T_{p}(\dot z)$ of $p + 1$ new points, namely if $z \in \H$ is a representative of $\dot z$ then
\begin{equation} \label{Hecke equation}
T_p(\dot z)= \PSL_2( \Z) \backslash  \left\{ pz , z/p, (z + 1) / p, \dots, (z + p - 1) / p \right\}
.\end{equation}
The space $Y / M$ can be identified with the set of infinite sequences $\dots,y _ {-1}, y _ 0, y _ 1,\dots$ with $y _ i \in T _ p (y _ {i -1}) \setminus \left\{ y _ {i -2}\right\}$, and under this identification multiplication by $a _ p$ in the $p$-direction becomes simply the shift action.
This in particular shows that multiplication by $a _ p$ on $Y / M$ (or, with a bit more effort on $Y$) has entropy $\leq \log p$, and just like in our case this maximum is attained for the Haar measure on $Y$.
From \eqref{Hecke equation} it is clear that if $y \in \PSL (2, \Z) \backslash \H$ is high up in the cusp, precisely 1 of its $T _ p$-points will be higher in the cusp, and $p$ of these points would be lower then $y$ in the cusp. Therefore if $\dots y _ {-1}, y _ 0, y _ 1, \dots$ are a sequence of points of $\PSL (2, \Z) \backslash \H$ as above and if $y_k$ are high up in the cusp for some contiguous range of $k$'s, say $n \leq k \leq m$, then in this range given the value of $y_k$ there is only one possible way of choosing $y_{k+1}$ so that it is higher than $y_k$, and since by assumption $y _ {k + 2} \neq y _ k$ once $y _ {k + 1}$ is lower than $y _ k$, the point $y _ { k + 2}$ being in $T _ p ( y _ {k + 1})$ but excluded from being $y _  k$ which is unique point in $T _  p (y _ {k + 1})$ higher than $y _ {k + 1}$ must be lower than $ y _ k + 1$. Hence if $y _ {k + 1}$ is lower than $y _ k$ for some $k$ in the above range, then $y _ {k' + 1}$ must be lower then $y _ {k'}$ for all $k'$ in the range $k \leq k ' \leq m$. From the above discussion it follows that while the trajectory is high up in the cusp, we have a choice of which subsequent point to choose only half of the time, hence the factor~${\tfrac{1}{2}}$.
\end{rem}

\subsection{Proof of Proposition \ref{cover-lemma}: the number of possible sets $V$.} \label{count_V}

The easiest part of Proposition \ref{cover-lemma} is the final assertion, i.e.
if we write
\[
 Q_{M,N}=\bigvee_{n=-N}^{N}T^{-n}\{X_{\geq M},X_{<M}\}.
\]
then the above partition $Q_{M,N}$ has $\ll_M e^{\frac{2\log\log M}{\log M}N}$ many elements.

 We make use of the fundamental domain $\mscS \subset \PSLdR$ from \S \ref{sec:notation};
the geodesic flow  $X$ corresponds to following the geodesic
 determined by $(z,v)$ until the boundary of the fundamental region is reached, at which point one applies either
$\begin{pmatrix}1&\pm 1\\ &1\end{pmatrix}$ to shift the geodesic horizontally
 or $\begin{pmatrix} &-1\\1&\end{pmatrix}$ to reflect on the bottom boundary of the fundamental region.

The basic point in the proof is that if $x \in X$ satisfies $\Ht(x) \geq M$,
then $\Ht(T^n x) \geq 1$ so long as $n <  \lfloor2\log M\rfloor$, i.e.
one needs at least  $\lfloor2\log M\rfloor$ steps to reach points of height less than $1$.

 Therefore, in a time interval of length $2\lfloor2\log M\rfloor$
 there can be only one stretch of times for which the points on the orbit are of height at least $M$. In other words the
 possible starting and end points of that time interval completely determine an element of
 $Q_{M,\lfloor2\log M\rfloor}$ which therefore has at most $\ll \log^2M$, say $\leq c_0\log^2M$,
 many elements. To obtain the lemma we note that $Q_{M,N}$ can be obtained
 by taking refinements of $\lfloor\frac{2N+1}{2\lfloor2\log M\rfloor+1}\rfloor\leq\frac{2N+1}{4\log M-1}$ many
 images and pre-images of $Q_{M,\lfloor2\log M\rfloor}$ and at most $2\lfloor2\log M\rfloor$ many
 of $\{X_{\geq M},X_{< M}\}$. We get that $Q_{M,N}$ has size
 $\ll_M (c_0\log^2M)^{\frac{2N}{4\log M-1}}$, which is at most $e^{\frac{2\log\log M}{\log M}N}$ once $M$ is large enough.

\subsection{Proof of Proposition \ref{cover-lemma}: covering $Z(V)$ by Bowen balls.}

Write $a=\begin{pmatrix}e^{\frac12}&0\\0& e^{-\frac12}\end{pmatrix}$, so that $T(x)=xa$.
Since $X_{< M}$ has compact closure, it suffices to restrict ourselves
 to a neighborhood $O$ of a point $x_0\in X_{<M}$.
 By taking the image under $T^N$ it also suffices to study the forward
 orbit as follows. We will show that for the set $V\subset[0,N-1]$ picked, the set
 \begin{multline*}
  Z_O^+=\Bigl\{x\in O\cap T^{-N} X_{< M}: \\
  \mbox{ for all $n\in[0,N-1]$ we have }T^n(x)\in X_{\geq M}\Leftrightarrow n\in V\Bigr\}
 \end{multline*}
 can be covered by
 $\ll_{M}2^{N-\frac{1}{2}|V|}$ forward-Bowen $N$-balls $x B_N^+$ where
 \[
  B_N^+=\bigcap_{n=0}^{N-1}a^{-n}B_\eta^{G} a^n.
  \]
 We may assume that the neighborhood we will consider is of the form
 \[
  O=x_0 B_{\eta/2}^{U^+}B_{\eta/2}^{U^-A}
 \]
 where $B_r^H$ denotes the $r$-ball of the identity in a subgroup $H<\SL_2(\R)$,
 $A$ denotes the diagonal subgroup, and $U^+$ resp.\ $U^-$ denote the unstable
 and stable horocyclic subgroups as in Section \ref{X compact}.

 Notice that by applying $T^n$
 to $O$ we get a neighborhood of $T^n(x_0)$ for which the $U^+$-part
 is $e^n$ times as big while the second part is still contained $B_{\eta/2}^{U^-A}$.
 By breaking the $U^+$-part into $\lceil e^n\rceil$ sets of the form $u^+_i B_{\eta/2}^{U^+}$
 for various $u_i^+\in U^+$ we can write $T_2^n(O)$ as a union of $\lceil e^n\rceil$ sets of the form
 $$
  T^n(x_0)u_i^+B_{\eta/2}^{U^+}a^{-n}
    B_{\eta/2}^{U^-A}a^n,
 $$
 i.e.\ we obtain neighborhoods of similar shape. If we take the preimage under $T^n$ of this set,
 we obtain a set contained in the forward Bowen $n$-ball $T^{-n}(T^n(x_0)u^+_i)B_n^+$. We will be iterating this procedure,
 but using the information that the orbit has to stay above height $M$ for a long
 time we will be able to cut down on the number of $u_i^+\in U^+$ needed to cover $Z_O^+$.

 In the proof of the claim we will use a partition of $[0,N]$ into sub-intervals
 of two types according to the set $V$. Notice that as in the proof of \S \ref{count_V} we can assume that $V$
 itself consists of intervals that are separated by $2\lfloor{2}\log M\rfloor$. For otherwise
 the set $Z_O^+$ is empty since no orbit under $T$ can leave $X_{\geq M}$ and return to it
 in a shorter amount of time.  We enlarge every such subinterval of $V$ by $\lfloor2\log M\rfloor$
 on both sides to obtain the first type of disjoint intervals $\mathcal{I}_1,\ldots,\mathcal{I}_k$.
 At the end points $0$ and $N$ we have required that $x,T^N(x)\in X_{< M}$  for all $x\in Z_O^+$. For this
 reason we can assume without loss of generality that all of these intervals are contained in $[0,N]$. (If this is not the case,
 we can enlarge the interval $[0,N]$ accordingly and absorb the change of the desired upper estimate in the multiplicative
 constant that depends on $M$ alone). The remainder of $[0,N]$ we collect into the intervals
 $\mathcal{J}_1,\ldots,\mathcal{J}_\ell$.

 We will go through the time intervals $\mathcal{I}_i$ and $\mathcal{J}_j$ in their respective order
 inside $[0,N]$. At each stage we will divide any of the  sets obtained earlier into $\lceil e^{|\mathcal{I}_i|}\rceil$- or
 $\lceil e^{|\mathcal{J}_j|}\rceil$-many sets, and in the case of $\mathcal{I}_i$ show that we do not have
 to keep all of them. More precisely, we assume inductively that for some $K\leq N$ we have
 $[0,K]=\mathcal{I}_1\cup\ldots\cup\mathcal{I}_i\cup\mathcal{J}_1\cup\ldots\cup\mathcal{J}_j$ and that
 all points in $Z_O^+$ can be covered by
 \[
  \leq 2 e^{|\mathcal{J}_1|+\cdots +|\mathcal{J}_j|+i\lfloor2\log_2M\rfloor+
     \frac{1}{2}(|\mathcal{I}_1|+\cdots+|\mathcal{I}_i|)}
 \]
 many preimages under $T^K$ of sets of the form
 \begin{equation}\label{onsided-set}
   T_2^K(x_0)u^+ B^{U^+}_{\eta/2}a^{-K}
    B_{\eta/2}^{U^-A}a^K.
 \end{equation}
 Note that for $K=N$ this gives the lemma since by construction
 $|\mathcal{I}_1|+\cdots+|\mathcal{I}_k|=2k\lfloor2\log M\rfloor+|V|$.

 For the inductive step it will be useful to assume a slightly stronger inductive assumption, namely that the multiplicative factor $2$ is only allowed if $[0,K]$ ends with the interval $\mathcal{J}_j$. Therefore, notice that if the next interval is $\mathcal{J}_{j+1}$ (i.e.\ $[0,K]$ ends with $\mathcal{I}_i$) then there is not much to show. In that case we keep all of the $\lceil e^{|\mathcal{J}_{j+1}|}\rceil\leq 2 e^{|\mathcal{J}_{j+1}|}$-many Bowen balls constructed above and obtain the claim.

 So assume now that the next time interval is $\mathcal{I}_{i+1}=[K+1,K+S]$. Here we will make use
 of the geometry of geodesics that visits $X_{\geq M}$ during that subinterval. Pick one of the sets
 \eqref{onsided-set} obtained in the earlier step and denote it by $Y$. By definition of $Z^+_O$
 we are only interested in points $y\in Y$ which satisfy
 \[
  T^n(y)\in X_{\geq M}\Leftrightarrow K+n\in V,
 \]
 or equivalently
 \begin{eqnarray*}
  \height(y),\height\bigl(T(y)\bigr),\ldots,\height\bigl(T^{\lfloor2\log M\rfloor}(y)\bigr)&<&M,\\
  \height\bigl(T^{\lfloor2\log M\rfloor+1}(y)\bigr),\ldots,\height\bigl(T^{S-\lfloor2\log M\rfloor}(y)\bigr)&\geq&M\\
  \height\bigl(T^{S-\lfloor2\log M\rfloor+1}(y)\bigr),\ldots,T^{S}(y))&<&M.
 \end{eqnarray*}
 If there is no such point in $Y$ there is nothing to show. So suppose $y,y'\in Y$ are such points.
 We will use the above restrictions on the heights to show that if
 \begin{equation}\label{y-star}
  y=T_2^K(x_0)u^+ u^+(t)v\mbox{ and }y'=T_2^K(x_0)u^+u^+(t')v'
 \end{equation}
 for $ u^+(t),u^+(t')\in B^{U^+}_{\eta/2}$ and $v,v'$ in the conjugate of $B_{\eta/2}^{U^-A}$, then
 $|t-t'|\ll 2^{-S/2}$.
 We can draw the geodesic orbits defined by $y$ and $y'$ in the upper half model of the hyperbolic plane
 and relate the conditions on $y,y'$ to geometric information about these geodesics. We choose the lifting of the paths
 in such a way that the time interval for which the height is above $M$ becomes the part of the geodesic
 where the imaginary part is above $M^2$.

 For the translation of the properties we will use the following
 observation: For two points $z_1,z_2\in\mathbb{H}$ on a geodesic line that are either both on the upwards part or both
 on the downwards part of the corresponding semi-circle their hyperbolic distance satisfies
 \begin{equation}\label{hyp-dist}
  |\log\operatorname{Im}(z_1)-\log\operatorname{Im}(z_2)|\leq d(z_1,z_2)\leq
    |\log\operatorname{Im}(z_1)-\log\operatorname{Im}(z_2)|+1.
 \end{equation}
 The lower bound actually gives the shortest distance between points with imaginary part $\operatorname{Im}(z_1)$
 and points with imaginary part $\operatorname{Im}(z_2)$. The upper bound gives the length of a path that first
 connects the point lower down, say $z_1$, to the point $z'$ immediately above with imaginary part $\operatorname{Im}(z_2)$
 and then moves horizontally to a point that is $\operatorname{Im}(z_2)$ far to the left or right of $z'$ towards $z_2$.
 For two points $z_1,z_2$ on the upwards or downwards part of a semi-circle this path actually goes through $z_2$.

 Applying the lower bound in \eqref{hyp-dist} to the points corresponding to
 $$
 y\mbox{ and }T_2^{\lfloor2\log M\rfloor+1}(y)
 $$
 whose hyperbolic distance is
 $\lfloor2\log M\rfloor+1$ we see that $\operatorname{Im}(y)\gg 1$
 (where in a slight abuse of notation we identify $y$ with the lifted point in $\mathbb{H}$). Similarly,
 we get from the upper bound for $y$ and $T_2^{\lfloor 2\log M\rfloor}(y)$ that
 $\operatorname{Im}(y)\ll 1$. Similar estimates hold for $T_2^S(y),y'$ and $T_2^S(y')$.

 We assume that the points $y,y'$ are lifted in such a way that $\operatorname{Re}(y)\in[-1/2,1/2]$
 and such that $y'$ is close to $y$. Denote by $\alpha_-,\alpha_+\in\R$ the backwards and forward
 limit points of the geodesic defined by $y$ on the boundary of $\mathbb{H}$ and similarly by $\alpha_-',\alpha_+'$
 the endpoints of the geodesic for $y'$.
 Then $|\alpha_-|<2+\frac{1}{2}$ since the lifting of the point $y$ was chosen
 such that the times of height $\geq M$ in $X$ correspond to imaginary part $\geq M^2$. For $y'$ this implies for small enough $\eta$ that
 $|\alpha_-'|<3$.

 Let $R=\frac{1}{2}|\alpha_+-\alpha_-|$ be the radius of the half circle defined by $y$ and define $R'$ similarly for $y'$.
 Then the above shows $R\ll|\alpha_+|\ll R$ once $M$ and so $R$ are large enough to make $\alpha_-$ negligible in
 comparison to $\alpha^+$. Similarly $R'\ll|\alpha_+'|\ll R'$.

 Applying \eqref{hyp-dist} twice, once for $y$ and the point $z$ on the same geodesic with imaginary part $R$,
 and once for $z$ and $T_2^S(y)$ we get
 \begin{equation}\label{S-R}
  |S -2\log R|\ll 1 \mbox{ and similarly } |S -2\log R'|\ll 1.
 \end{equation}
 Therefore, $R\ll R'\ll R$ and so $|\alpha_+|\ll|\alpha_+'|\ll|\alpha_+|$.

 Suppose $g=\begin{pmatrix}a&b\\ c&d\end{pmatrix}\in\SL(2,\R)$ defines $y=T_2^K(x_0)u^+ u^+(t)v$ in the sense
 that the natural action of $g$ maps the upwards vector at $i$ to the
 vector associated to $y$ for the lifting considered above. Then $\alpha_+=g(\infty)=\frac{a}{c}$ and
 $\alpha_-=g(0)=\frac{b}{d}$.
 Similarly, suppose $g'$ defines $y'=T_2^K(x_0)u^+u^+(t')v'$ such that $\alpha_+'=g'(\infty)$.
 Using this notation we summarize what we already know about these matrices
 \begin{equation}\label{estimate-9}
 \begin{split}
  \max(|a|,|b|,|c|,|d|)&\ll 1,\\
  R\ll |\alpha_+|=|\frac{a}{c}|&\ll R, \\
   R\ll |\alpha_+'|&\ll R,\mbox{ and }\\
  |\alpha_-|=|\frac{b}{d}|&\ll 1.
  \end{split}
 \end{equation}
 Here the first estimate follows since we know roughly the position of the lift corresponding to $y$
 which means that $g$ belongs to a compact subset of $\SL(2,\R)$.
 We claim the above implies that
 \begin{equation}\label{estimate-10}
  1\ll|d|,\ 1\ll|a|,\mbox{ and } |c|\ll |a|R^{-1}\ll R^{-1}.
 \end{equation}
 The first estimate follows since $|b|\ll|d|$ by the last estimate in \eqref{estimate-9} and since
 $g\in\SL(2,\R)$ belongs to a compact subset so that not both $b$ and $d$ are small.
 The second claim follows similarly from the second estimate in \eqref{estimate-9}.

 To simplify the following calculation we would
 like to remove the elements $v,v'$ (as in \eqref{y-star})
 from our consideration --- but to do this we need to see how this affects
 the above statements. Recall first that
 $v,v'\in B^{U^-A}_\eta$ and so $v(\infty)=v'(\infty)=\infty$.
 Therefore, the first three estimates above remain unaffected when
 changing $g$ resp.\ $g'$ on the right by  $v^{-1},(v')^{-1}$.
 Moreover, we have $|v^{-1}(0)|\ll\eta$ and so for small enough $\eta$
 that $1\ll|d|\ll |cv^{-1}(0)+d|$ which implies $|gv^{-1}(0)|\ll 1$.
 In other words, none of the estimates in \eqref{estimate-9} are affected (apart from possibly
 the values of the implicit constants) by the proposed transition
 from $g$ to $gv^{-1}$ resp.\ $g'$ to $g'(v')^{-1}$ and we can assume $v=v'=e$.

 Comparing the definitions of $y$ and $y'$ we get $g'=g u^+(t)^{-1}u^+(t')$.
 Therefore,
 $$
  \alpha_+'=g'(\infty)=\bigl(gu^+(t'-t)\bigr)(\infty)=\frac{\frac{a}{t'-t}+b}{\frac{c}{t'-t}+d}=
   \frac{a+b(t'-t)}{c+d(t'-t)}.
 $$
 Since $1\ll |a|$, $ u^+(t),u^+(t')\in B^{U^+}_{\eta/2}$, and so $|t'-t|\ll \eta$
 we can simplify the numerator and obtain together with the third estimate in \eqref{estimate-9}
 that for small enough $\eta>0$
 $$
  R\ll \bigl|\frac{a}{c+d(t'-t)}\bigr|\ll R,
 $$
 or equivalently
 $$
  R^{-1}\ll|c+d(t'-t)|\ll R^{-1}.
 $$
 Since $|c|\ll R^{-1}$ and $1\ll |d|$ by \eqref{estimate-10} this implies the estimate $|t'-t|\ll R^{-1}$.
 Now recall from \eqref{S-R} that $e^{S/2}\ll R$, so that we get the desired $|t'-t|\ll e^{-S/2}$

 Recall next that in the current time interval $\mathcal{I}_{i+1}$
 we divide $B_{\eta/2}^{U^+}$ into $\lceil e^S\rceil$ balls of the form $B_{e^{-S}\eta/2}^{U^+}$.
 Since all points $y'$ that belong to $Y\cap T^K(Z_O^+)$ satisfy the estimate $|t'-t|\ll e^{-S/2}$
 we see that only $\ll e^Se^{-S/2}= e^{S/2}$ can (after the correct thickening along $AU^-$)
 contain an element of $Y\cap T^K(Z_O^+)$. This implies the inductive
 claim if we assume $M$ is sufficiently large so that the upper bound we got
 is strictly bounded from above by $\frac12 e^{\lfloor 2\log M\rfloor+S/2}$.

This concludes the proof of Proposition \ref{cover-lemma}.

\subsection{Entropy and covers; proof of Theorem \ref{entropy-cusp}}
For the proof of Theorem~\ref{entropy-cusp} we need to relate entropy and covers via
Bowen balls. For this we need the following (well known) result, which is proved in Appendix~\ref{Appendix-B} below (for cocompact $\Gamma$ it follows from Brin and A.~Katok~\cite{Brin-Katok}).

\begin{Lemma}\label{cover-entropy-main}
 Let $\mu$ be an $A$-invariant measure on $X=\Gamma\backslash\SL(2,\R)$.
 For any $N\geq 1$ and $\epsilon>0$ let $BC(N,\epsilon)$ be the minimal
 number of Bowen $N$-balls needed to cover any subset of $X$ of measure bigger than $1-\epsilon$.
 Then
 \[
  h_\mu(T)\leq\lim_{\epsilon\rightarrow 0}\liminf_{N\rightarrow\infty}\frac{\log BC(N,\epsilon)}{2N}
 \]
 where $T$ is the time-one-map of the geodesic flow.
\end{Lemma}

\begin{proof}[Proof of Theorem \ref{entropy-cusp}]
 Note first that it suffices to consider ergodic measures. For if $\mu$ is not
 ergodic, we can write $\mu$ as an integral of its ergodic components $\mu=\int\mu_t \operatorname{d}\!\tau(t)$
 for some probability space $(T,\tau)$. Therefore, $\mu(X_{\geq M})=\int\mu_t(X_{\geq M})\operatorname{d}\!\tau(t)$
 but also $h_\mu(T)=\int h_{\mu_t}(T)\operatorname{d}\!\tau(t)$ by \cite[Thm.~8.4]{Walters},
 so that the desired estimate follows from the ergodic case.

 Suppose $\mu$ is ergodic. To apply Lemma \ref{cover-entropy-main} we need to show
 that most of $X$ can be covered by not too many Bowen $N$-balls.  Once $M>3$
 we have that every $T$-orbit visits $X_{<M}$, and so $\mu(X_{<M})>0$.
 By the ergodic theorem there exists for every $\epsilon>0$ some $K\geq 1$ such that
 $$
  Y=\bigcup_{k=0}^{K-1} T^{-k} X_{<M}\mbox{ satisfies }\mu(Y)>1-\epsilon.
 $$
 Moreover, also by the ergodic theorem
 \[
  \frac{1}{2N+1}\sum_{n=-N}^{N}1_{X_{\geq M}}(T^n(x))\rightarrow\mu(X_{\geq M})
 \]
 as $N\rightarrow\infty$ for a.e.\ $x\in X$.
 So for large enough $N$ the average on the left will be bigger than $\kappa=\mu(X_{\geq M})-\epsilon$
 for any $x\in X_1$ and some subset $X_1\subset X$ of measure $\mu(X_1)>1-\epsilon$.
 Clearly for any $N$ the set
 $$
 Z=X_1\cap T^N Y\cap T^{-N} Y
 $$
 has measure bigger than $1-3\epsilon$. Recall that we are interested in the asymptotics
 of the minimal number of Bowen $N$-balls needed to cover $Z$.
 Here $N\rightarrow\infty$ while $\epsilon$ and so also $K$ remain fixed.
 Since we can decompose $Z$ into $K^2$ many sets of the form
 $$
  Z'=X_1\cap T^{N-k_1}X_{<M}\cap T^{-N-k_2}X_{<M},
 $$
 it suffices to cover these, and for simplicity of notation we assume $k_1=k_2=0$.
 Next we split $Z'$ into the sets $Z(V)$ as in Proposition \ref{cover-lemma} for the various subsets $V\subset[-N,N]$.
\S \ref{count_V} shows that we need at most $\ll_M e^{\frac{2\log\log M}{\log M}N}$ many of these. Moreover,
 by our assumption on $X_1$ we only need to look at sets $V\subset[-N,N]$
 with $|V|\geq\kappa (2N+1)$. Therefore, Proposition~\ref{cover-lemma} gives
 that each of those sets $Z(V)$ can be covered by $\ll_{M}e^{(1-\frac{\kappa}{2})2N}$ many Bowen $N$-balls.
 Together we see that $Z$ can be covered by $\ll_{M,K} e^{\frac{2\log\log M}{\log M}N+(1-\frac{\kappa}{2})2N}$ Bowen $N$-balls.
 Applying Lemma \ref{cover-entropy-main} we arrive at
 $$
  h_\mu(T)\leq 1+\frac{\log\log M}{\log M}-\frac{\mu(X_{\geq M})-\epsilon}{2}
 $$
 for any $\epsilon>0$, which proves the theorem.
\end{proof}


\appendix

\section{Representations of binary quadratic forms by ternary forms}\label{Appendix-A}

 In this section we establish Proposition \ref{2-3reps}:
 \begin{prop*}
Let $Q$ be an non-degenerate, integral\footnote{I.e.\ $Q(\Zz^3)\subset\Zz$.}  ternary quadratic form on $\Z^3$, and let
$$q(x,y) = a_1 x^2+a_2  xy+a_3 y^2$$
be a non-degenerate binary quadratic form on $\Z^2$.  Let $f^2$ be the greatest square dividing $\gcd(a_1, a_2, a_3)$. Then the number $N(q)$ of embeddings
of $(\Z^2, q)$ into $(\Z^3, Q)$, modulo the action of $\so_Q(\Z)$, is $\ll_{Q,\epsilon} f \max(|a_1|,|a_2|,|a_3|)^{\epsilon}$.
\end{prop*}

We recall that an embedding of $(\Z^2, q)$ into $(\Z^3, Q)$ is a linear map $\iota: \Z^2 \rightarrow \Z^3$
with the property that $Q(\iota(\x)) = q(\x)$.  Such proposition was established for the first time by {Venkov} for $Q=x^2+y^2+z^2$ and extended by Pall to other quadratic forms \cite{Venkov,Pall}. The proposition can be deduced from Siegel's {\em mass formula}; here we present a direct and elementary argument inspired by the adelic proof of Siegel's mass formula as outlined by Tamagawa (cf.\ Weil's paper~\cite{Weil}).
\begin{Remark}

\begin{itemize}
\item[-] One may wonder what the dependency on $Q$ in the above bound looks like; this is for instance important to obtain equidistribution results when $Q$ is allowed to vary (see for instance \cite[Thm. 1.8]{EMV}). In the case where $Q$ is a multiple of the norm form on a lattice in the space of trace zero elements of a quaternion algebra whose associated order is an Eichler order, it can be shown that the dependency is of the shape $\ll_{\eps}|\disc(Q)|^{1/2+\eps}\dots$. It seem plausible that this holds in general.
\item[-]The argument provides, in fact, an upper bound for the the sum over a set of representatives $Q_{i},\ i=1,\ldots, g$ of the genus classes of $Q$, of the number of embeddings of $(\Zz^2,q)$ into $(\Zz^3,Q_{i})$ modulo $\SO_{Q_{i}}(\Zz)$.
\item[-]Finally it is easy to see that this argument carries over without significant changes to quadratic forms defined over a general number field.
\end{itemize}
\end{Remark}
\subsection{Reduction to local counting problems}\label{counting-local}
Fix an embedding $\iota: (q,\Z^2) \hookrightarrow (\Z^3, Q)$ and let $$L:=\iota(\Z^2)$$ be its image  (if no such embedding exists, we are obviously done.) Then any other embedding $\iota'$
is (by Witt's theorem; see \cite[IV.1.5, Theorem 3]{Serrebook}) of the form $g \circ \iota$, with $g \in \so_Q(\Q)$.  The stabilizer of $\iota$
inside $\so_Q(\Q)$ is trivial, for any isometry fixing $L$ pointwise would need to map $L^{\perp}$ to itself and so must be multiplication by $\pm 1$ on $L^{\perp}$;
the condition of determinant $1$ forces it to be the identity. The number of embeddings $N(L)$ (up to the action of $\so_Q(\Z)$) is therefore
precisely the number of cosets $\dot{g} \in \so_Q(\Z) \backslash \so_Q(\Q)$
so that $g L\subset \Z^3$.

Given a rational lattice $\Lambda\subset \Qq^3$, for any prime $p$ we denote by
$$\Lambda_{p}=\Lambda\otimes_{\Zz}\Zp$$
its closure inside $\Qp^3$. Let us recall that the map $$\Lambda\mapsto (\Lambda_{p})_{p}$$
is a bijection between the set of lattices in $\Qq^3$ and the set of sequences of lattices indexed by the primes $(\Lambda_{p})_{p}$, $\Lambda_{p}\subset\Qp^3$ such that $\Lambda_{p}=\Zp^3$ for a.e.\ $p$. Write $K_p=\SO_{Q}(\Zp)$ for the stabilizer of $\Zp^3$ inside $\SO_{Q}(\Qp)$ and let
$$\SO_{Q}(\Af)=\{ g_{f}=(g_{p})_{p},\ g_{p}\in\SO_{Q}(\Qp),\ g_{p}\in \SO_{Q}(\Zp)\hbox{ for a.e.\ $p$}\};$$
 the above bijection induces an action of $\SO_{Q}(\Af)$ on the set of rational lattices:
$$g_{f}.\Lambda\leftrightarrow g_{f}.(\Lambda_{p})_{p}:=(g_{p}\Lambda_{p})_{p}.$$

\begin{rem}The group $\SO_{Q}(\Af)$ is the group of {\em finite ad\`eles} of $\SO_{Q}$. The $\SO_{Q}(\Af)$-orbit of a lattice $\Lambda\in\Qq^3$ under this action is called the {\em $Q$-genus} of $\Lambda$. We will not need much of this terminology or discuss further properties of adelic groups here.
\end{rem}
The group $\SO_{Q}(\Qq)$ embeds diagonally into $\SO_{Q}(\Af)$. Now the stabilizer of $\Z^3$ in $\SO_{Q}(\Af)$ is $K_{f} = \prod_{p} \SO_{Q}(\Zp)$ and since $K_{f} \cap \so_Q(\Q) = \so_Q(\Z)$, $\so_Q(\Z) \backslash \so_Q(\Q)$
injects into $K_{f} \backslash \so_Q(\adele_f)$.

Consequently, letting $L_p=L\otimes_{\Zz}\Zp$ be the closure of $L$ inside $\Z_p^3$, we have
\begin{align*}
N(L) &\leq |\{ g_f  \in K_f \backslash \so_Q(\Af): g_f . L \subset \Z^3\}|\\
&\leq \prod_{p} |\{ g_p  \in \SO_{Q}(\Zp) \backslash \so_Q(\Q_p): g_p . L_p \subset \Z_p^3\}|
\\
&= \prod_{p} | \{g_p  \in \so_Q(\Q_p)/\SO_{Q}(\Zp): L_p \subset g_p \Z_p^3\}|=\prod_{p}N(L_{p}).
\end{align*}
with
$$N(L_p) =| \{g_p  \in \so_Q(\Q_p)/K_{p}: L_p \subset g_p \Zp^3\}|=|\{\Lambda \in \so_Q(\Q_p).\Zp^3: L_p \subset \Lambda\}|$$
being the number of lattices in $\Qp^3$, within the $Q$-isometry class of $\Zp^3$ that contain $L_{p}$. We have proven that
$$N(L)\leq \prod_{p}N(L_p)$$
 and thus have reduced our counting problem to a collection of local counting problems (as we will see below $N(L_{p})=1$ for a.e. $p$); a more careful analysis of what we have said so far is very closely related to the proof of the mass formula. In the present paper, however, we need only upper bounds.

\subsection{The anisotropic case and a reduction step}\label{unramified-reduction}
 We first introduce some notations. We denote by $$\peter{\bfx,\bfx'}=Q(\bfx+\bfx')-Q(\bfx)-Q(\bfx')$$ the bilinear form associated with $Q$; so $\peter{\bfx,\bfx}=2Q(\bfx)$.
The discriminant of $Q$ is set to be
$$\disc(Q)=\det(\peter{\bfx_i,\bfx_j})_{i,j\leq 3}$$ for $\{x_1,x_2,x_3\}$ any basis of $\Z^3$. Since $Q$ is integral
$\peter{\Z^3,\Z^3}\subset\Zz$, so $\disc(Q)$ is a non-zero integer.

We notice first that if $Q$ does not represent $0$ nontrivially over $\Qp$ (i.e.\ is anisotropic over $\Qp$), then $\SO_{Q}(\Qp)$ is  compact and
\begin{equation}\label{anisotropic-bound}
	 N(L_{p})\leq [\SO_{Q}(\Qp):\SO_Q(\Zp)]\ll_Q 1
\end{equation}
 This (favorable) situation can occur only if $p$ divides $\disc(Q)$.

We suppose now that $Q$ is isotropic over $\Qp$ for some prime $p\ |\ 2\disc (Q)$,
we will reduce the problem of bounding $N(L_p)$ to the case where the integral quadratic form is given by $Q(x,y,z)=xy+z^2$. We note that $\disc(xy+z^2)=2$. This reduction comes with the cost that we also have to replace $q$ by a different quadratic form $q'=up^{m_p}q$ with $u\in\Zp^*$ and $m_p\geq 0$. However, we only have to make this change for $p\ |\ 2\disc(Q)$ and $m_p$ will only depend on $Q$. Using these facts we will see in Subsection \ref{appaproof} that the bound for the number of representations of $q'$ by $xy+z^2$ will suffice for the proof of Proposition \ref{2-3reps}.

We claim that there exists a basis of $\Q_p^3$ over $\Q_p$ so that the quadratic form $Q$ with respect to the coordinates of this basis has  the form $up^{-\ell}(xy+z^2)$ for some $u\in\Zp^*$ and $\ell\in\{0,1\}$. Indeed as $Q$ is isotropic, there exists a hyperbolic plane in $\Q_p^3$. Complementing the basis of the hyperbolic plane with a vector of the orthogonal complement we arrive at a basis so that $Q$ has the form $xy+u p^{-\ell} z^2$ with $u\in\Zp^*$ and $\ell\in\Z$. If necessary we may replace the last basis vector by a multiple and can ensure that $\ell\in\{0,1\}$. Similarly we may divide the first basis vector by $up^{-\ell}$ and arrive at the claim.

Let $\Lambda$ be the $\Z_p$-lattice in $\Q_p^3$ spanned by the above basis. There exists some $k$ (depending only on $\Lambda$) so that $p^{k}\Z_p^3\subset\Lambda$. Let $\iota:(\Z_p^2,q)\to(\Z_p^3,Q)$
be an embedding of $q$. Then $p^{k}\iota : (\Z_p^2,p^{2k}q)\to (\Lambda,Q)$ and finally
$$
 p^{k}\iota : (\Z_p^2,u^{-1}p^{2k+\ell}q)\to (\Lambda,u^{-1}p^\ell Q)\simeq(\Z_p^3,xy+z^2)
$$
 are also
embeddings of quadratic lattices. We set $m_p=2k+\ell$ and $q'=u^{-1}p^{m_p}q$ and obtain
that there is an injection from the set of embeddings $\iota:(\Z_p^2,q)\to(\Z_p^3,Q)$
to the embeddings $\iota':(\Z_p^2,q')\to(\Z_p^3,xy+z^2)$.

\subsection{The case of an unramified lattice} \label{unramified-local}
The previous section reduces the proof of Proposition~\ref{2-3reps} to the problem of finding an upper bound for $N(L_p)$ where we may assume that either $p\, \ndiv\, 2\disc(Q)$ or that $Q(x,y,z)=xy+z^2$.
This will be done in the following two local counting lemmas which depend on whether $p=2$ or $p>2$:

Recall that for $p>2$ any quadratic form
$q$ on some rank two $\Zp$-lattice $L$ taking value in $\Zp$ may be written, in a suitable basis, in the form \begin{equation}\label{A-0}
 q(xe_1+ye_2)=up^ax^2+vp^by^2,\ u,v\in\Zp^\times,\ 0\leq a\leq b\in\Zz_{\geq 0}.
\end{equation}
To see this take an element $e_1 \in L$ such that the valuation of $q(e_1)$ is minimal and then take the orthogonal complement of $e_1$, cf.\ \cite[Sect.\ 8.3]{Cassels}.  We shall call the integers $a\leq b$ the {\em invariants} of the quadratic form (e.g.\ the invariants of $x^2+5 y^2$ over $\Z_5$ are $(0,1)$). This is a kind of $p$-adic analogue of the notion of successive minima. The invariants determine the quadratic form over $\Z_p$ -- up to isometry -- up  to $O(1)$ possibilities. We will prove the following lemma.

\begin{Lemma}\label{padic} Let $p>2$, let $Q$ be an isotropic quadratic form over $\Qp^3$
so that $p\,\ndiv\,\disc(Q)$.
Let $L\subset \Lambda$ be a rank $2$-sublattice such that $Q_{|L}$ has invariants $(a,b)$, then
$$N(L;\Lambda):=|\{\Lambda' \in \so_Q(\Q_p).\Lambda: L \subset \Lambda'\}|
\ll(b+1)^2
p^{\lfloor a/2\rfloor}$$ where the implied constant is  absolute. Moreover, if $(a,b)=(0,0)$, $N(L;\Lambda)=1$
\end{Lemma}

In the $2$-adic case, any quadratic form
$q$ on some rank $2$ $\Z_2$-lattice $L$ taking value in $\Z_2$ may be written, in a suitable basis either (cf. \cite[Lemma 2.1]{hanke} and \cite[Sect.\ 8.4]{Cassels}) in the form
\begin{equation}\label{2adicform1}
q(xe_1+ye_2)=u2^ax^2+v2^by^2,\ u,v\in\Z_2^\times,\ 0\leq a\leq b\in\Zz_{\geq 0},
\end{equation}
or in the form
\begin{equation}\label{2adicform2}q(xe_1+ye_2)=u2^bx^2+w2^{a}xy+v2^by^2,\ u,v,w\in\Z_2^\times,\ 0\leq a\leq b\in\Zz_{\geq 0}.
\end{equation}
In both cases we will refer to $a\leq b$ once more as the invariants of $q$.
We have the following lemma.
\begin{Lemma}\label{2adic} Consider $Q(x,y,z)=xy+z^2$ as a quadratic form over $\Qq_2^3$, let $\Lambda\subset\Qq_2^3$ be a lattice satisfying  $Q(\Lambda)\subset\Zz_2$ and which is maximal for this property. Let $L\subset \Lambda$ be a rank $2$-sublattice such that $Q_{|L}$ has invariants $(a,b)$, then
$$N(L;\Lambda)\ll(b+1)^2
2^{\lfloor a/2\rfloor}$$ where the implied constant is  absolute.
\end{Lemma}

The proof of these two lemma will use a geometric interpretation of the quotient $\SO_Q(\Qp^3)/\SO_Q(\Lambda)$.

\subsection{The Bruhat-Tits tree} Let $Q$ be an isotropic quadratic form such that $p\,\ndiv\,\disc(Q)$ or   $Q(x,y,z)=xy+z^2$.
Note that $\Lambda_0=\Zp^3$ has the property that $Q(\Lambda_0)\subset\Zp$ and that $\Lambda_0$ is maximal for this property.
We set $$\mcT_Q=\SO_Q(\Qp^3)\Lambda_0\simeq \SO_Q(\Qp^3)/K_p.$$
Even though this will not be used here, let us also mention that $\mcT_Q$
is the set of all lattices $\Lambda$ in $\Qp^3$ such that $$Q(\Lambda)\subset\Zp$$ and which are maximal for this property (see \cite[Cor. 4.17]{GI}).

We will need that $\mcT_Q$ has the structure of a $(p+1)$-regular tree (see \cite{BT}) in which
$\Lambda, \Lambda'$ are adjacent if and only if $\Lambda \cap \Lambda'$
has index $p$ in $\Lambda$ (or equivalently in $\Lambda'$).
More generally, the distance $d(\Lambda,\Lambda')$ between two vertices $\Lambda, \Lambda'$ satisfies
$$p^{d(\Lambda,\Lambda')}=[\Lambda:\Lambda\cap\Lambda']=[\Lambda':\Lambda\cap\Lambda'].$$
and the geodesic between $\Lambda$ and $\Lambda'$
consists of all $\Lambda''\in\mcT_Q$ satisfying $\Lambda \cap \Lambda' \subset \Lambda''$.

Let us describe the adjacency structure on $\mcT_Q$ more explicitly using the quadratic structure.
Given any lattice $\Lambda\in \mcT_{Q}$,
and any primitive $\bfv \in \Lambda$ (i.e.\ $\bfv\notin p\Lambda$) for which $\obfv=\bfv+p\Lambda\in\Lambda/(p\Lambda)$
is isotropic over $\Fp$ (i.e.\ $p\ |\ Q(\bfv)$)
we can define a lattice $\Lambda_{\obfv} \in \mcT_Q$, which only depends on the line through $\obfv$, as follows.
Since
\begin{equation}\label{elem}
 Q({a}\mathbf{v}+\mathbf{z})={a^2}Q(\mathbf{v})+Q(\mathbf{z})+ a\peter{\bfz,\bfv}
 \in\Zp
\end{equation}
and since the linear form $\peter{\cdot,\obfv}$ is not zero (even for $p=2$),
we may modify $\bfv$ by some element $p\mathbf{z}_0\in p\Lambda$ to ensure that $p^2\ |\ Q(\bfv+p\bfz_0)$. Here the element $\bfz_0$ is uniquely determined by $\bfv$ up to $\{ \mathbf{z} \in \Lambda: \peter{\bfz,\bfv} \equiv 0 \mbox{ mod $p$}\}$.
Therefore, the lattice
 $$\Lambda_{\obfv} := \frac1p\Z_p(\mathbf{v}+p\bfz_0)  + \{ \mathbf{z} \in \Lambda: \peter{\bfz,\bfv} \equiv 0 \mbox{ mod $p$}\}
 $$
depends only on $\obfv$, indeed only on the line through $\obfv$. Using \eqref{elem} we see quickly
that $Q(\Lambda_{\obfv})\subset\Z_p$.
Below we will always assume that $p^2\ |\ Q(\bfv)$ and set $\bfz_0=0$.

Under our assumptions on $Q$ this lattice $\Lambda_\obfv\in\mcT_Q$ is a neighbor of $\Lambda$, and
there are exactly $p+1 =| \mathbb{P}^1(\Fp)|$ such lines,
and thus every neighbor arises.


We will use also the following simple facts:
%


%
\begin{enumerate}
 \item For an isotropic $\obfv$ we have
$$
 \Lambda\cap\Lambda_\obfv=\Zp\bfv+\{\mathbf{z}\in\Lambda: \peter{\bfv,\mathbf{z}} \equiv 0 \mbox{ mod $p$}\}.
$$
 \item For $\obfv, \obfv'$ generating distinct isotropic lines the intersection
 $$\Lambda_{\obfv}\cap \Lambda_{\obfv'}= \{ \mathbf{z} \in \Lambda: \peter{\mathbf{v} , \mathbf{z}}\equiv \peter{\mathbf{v}' , \mathbf{z}}\equiv 0 \mbox{ mod $p$}\}=\Zp \bfw+p\Lambda$$
 is the preimage in $\Lambda$ of the orthogonal subspace $(\Fp\obfv+\Fp\obfv')^\perp\subset\Fp^3$.\label{preimageline}
 \item Given three isotropic vectors $\obfv,\obfv',\obfv''$ generating distinct lines and assuming $p>2$ we have
 $$\Lambda_{\obfv}\cap \Lambda_{\obfv'}\cap \Lambda_{\obfv''}=p\Lambda.$$\label{threeintersec}
\end{enumerate}

One establishes also the following generalization:
\begin{prop}\label{intersectprop}
Let $\Lambda$ lie at the mid-point of the geodesic between $\Lambda'$ and $\Lambda''$
(i.e. there is   $n\geq1$ such that $d(\Lambda,\Lambda')=d(\Lambda,\Lambda'')=n$, $d(\Lambda',\Lambda'')=2n$). There exists a primitive $\bfv\in \Lambda$ so that $Q(\bfv)\equiv 0(p^n)$ and $\bfw\in \Lambda$ with $Q(\bfw)\not\equiv 0(p)$ and $\peter{\bfv,\bfw}\equiv 0(p^n)$ so that
 $$\Lambda\cap \Lambda'=\{\bfz\in \Lambda,\ \peter{\bfz,\bfv}\equiv 0(p^n)\}=\Zp\bfv+\Zp\bfw+p^n\Lambda$$
 and
 $$\Lambda'\cap \Lambda''=\Zp\bfw +p^n\Lambda$$ is the preimage of the non-isotropic line defined by $w$
 under the projection $\Lambda\mapsto \Lambda/p^n\Lambda$. Moreover, for $m\leq n$, let $\Lambda'_m$ be the lattice on the segment $[\Lambda,\Lambda']$ at distance $m$ from $\Lambda$, then
 $$\Lambda\cap \Lambda'_m=\{\bfz\in \Lambda,\ \peter{\bfz,\bfv}\equiv 0(p^m)\}=\Zp\bfv+\Zp\bfw+p^m\Lambda\supset \Lambda\cap \Lambda'.$$
 \end{prop}

\subsection{Proof of Lemma \ref{padic}}
Let $p>2$ and $Q$ be as in the lemma. Define
$$\mcR(L):=\{\Lambda\in\mcT_Q,\ L\subset \Lambda\}\subset\mcT_{Q},\ N(L)=|\mcR(L)|.$$
In the notation of Lemma \ref{padic}, $N(L)=N(L;\Lambda)$  for any $\Lambda\in\mcT_Q$.

We start by remarking that $\mcR(L)$ is connected:
if $\Lambda, \Lambda'$ both contain $L$, then $L \subset \Lambda \cap \Lambda' \subset \Lambda''$ for any $\Lambda''$ on the geodesic path between $\Lambda$ and $\Lambda'$.

Let $q$ be as in \eqref{A-0}. Suppose $\mcR(L)$ is non-empty and let $\iota:(\Zp^2,q)\hookrightarrow (\Lambda,Q)$ be an isometric embedding with image
$L=\iota(\Zp^2)$ and let $e_1=\iota(1,0),\ e_2=\iota(0,1)$ so $$Q(e_1)=up^a,\ Q(e_2)=vp^b,\ \peter{e_1,e_2}=0.$$

\subsubsection{The case $(a,b)=(0,0)$}
We show $\mcR(L) = \{\Lambda\}$. If not, $L$ is also contained in a neighbor $\Lambda_{\obfv}$
of $\Lambda$. However, the induced quadratic form on the span of $\bar{e}_1, \bar{e}_2$
is nondegenerate, so this span cannot be $\obfv^{\perp}$ for an isotropic $\obfv \in \Lambda/p\Lambda$.
 So $N(L)=1$.

\subsubsection{The case $a=0$, $b\geq 1$} \label{case-p-0-1}
Suppose that $N(L)>1$. Then there is an isotropic $\obfv$ so that $\ov{e_1}$ belongs to $\obfv^\perp$. This shows that $e_1^\perp$ is a hyperbolic plane (first modulo $p$, and then since $p\neq 2$ also on $\Q_p^3$).


In other words,
$e_1^{\perp} \cap \Lambda$ is a rank two lattice generated by two isotropic vectors $\bfv,\bfv'$ (which are liftings of isotropic vectors generating $\ov{e_{1}}^\perp$) and then, there are exactly two neighboring lattices containing ${e_{1}}$, namely $\Lambda_{\obfv}$ and $\Lambda_{\obfv'}$; that there are at most two follows from Fact \refs{threeintersec}. Pursuing this reasoning, we see that the only lattices that can contain $e_{1}$ are the lattices
$$\Lambda_{n}:= \Zp p^{-n}\bfv+\Zp e_{1}+\Zp p^{n}\bfv',\ n\in\Zz$$
(which is a  geodesic in the tree determined by $e_{1}$).

Let us now see that for $n> b$, $\Lambda_{\pm 2n}$ does not contain $e_{2}$, which will show that $N(L)\leq 4b+3$. Suppose $e_{2}\in \Lambda_{n}$, then
$$e_{2}\in \Lambda\cap \Lambda_{2n}=\Zp e_{1}+p^n\Lambda_{n}$$
write $e_{2}=\alpha e_{1}+\bfz$, $\alpha\in\Zp,\ \bfz\in p^n\Lambda_{n}$ we obtain
 $$\peter{e_{1},e_{2}}=0\equiv \alpha (\modu p^n),\ Q(e_{2})=vp^b\equiv  \alpha^2\equiv 0(\modu p^n).$$
This is a contradiction for $n>b$.

\subsubsection{The case $a=1$}
We show  $N(L)\leq 2$: Suppose that $L\subset \Lambda_{\obfv}$ for some $\obfv$.
Since $\ov e_1\in\Lambda/p\Lambda$ is a non-zero isotropic vector contained in $\obfv^\perp$
it has to be a multiple of $\obfv$.
By symmetry between $\Lambda$ and $\Lambda_\obfv$, this also shows that $\Lambda$ is the only neighbor
of $\Lambda_\obfv$ which contains $L$. Since $\mcR(L)$ is a connected subset of the tree, this shows that $N(L)\leq 2$ as claimed.


\subsubsection{The case $a\geq 2$}\label{padic>=2}
Let $$L_1:=\Zp e'_1+\Zp e_2,\ L_2:=\Zp e_1+\Zp e'_2,\ e'_i=e_i/p,\ i=1,2,\ L_{1}+L_{2}=\frac1pL$$
these are rank $2$ lattices containing $L$, on which $Q$ is $\Zp$-valued with respective invariants  $(a-2,b)$,
$(a,b-2)$ and $(a-2,b-2)$. We will show that either $N(L)=1$ or
\begin{equation}\label{inclusion}
\mcR(L)\subset \mcR(L_1)\cup \mcR(L_2)\cup \bigcup_{\Lambda'\in\mcR(\frac{1}pL)}B(\Lambda',1)
\end{equation}
where
$B(\Lambda',d)=\{\Lambda''\in\mcT_{Q},\ d(\Lambda',\Lambda'')\leq d\}$
is the ball in the tree of radius $d$ centered at $\Lambda'$; it has
cardinality $1+(p+1)\frac{p^d-1}{p-1}\leq (1+\frac{3}p)p^{d}.$

Here is the proof of \eqref{inclusion}. Let $\Lambda\in\mcR(L)$. If $e_1\in p\Lambda$ or $e_2\in p\Lambda$, then $\Lambda\in\mcR(L_1)\cup \mcR(L_2)$. So suppose now $e_1,e_2\in\Lambda$ are both primitive vectors. By assumption, we have for $i=1,2$ (since $Q(e_i)\equiv 0(\modu p)$)
that $\ov{e_i}$ is a non-zero isotropic vector.
Since $\peter{e_1, e_2} = 0$, $\ov{e_1} $ and $ \ov{e_2}$ have to be co-linear; otherwise
the induced form on the reduction $\ov{\Lambda}$ would be identically zero on a plane.
Now $\Lambda_{\ov e_1}$ contains both $L_1$ and $L_2$;
so it belongs to $\mathcal{R}(\frac1pL)$. Thus $\Lambda$ is at distance
at most $1$ from $\mathcal{R}(\frac1pL)$.
%

%

Let us now see how to conclude the proof of Lemma \ref{padic}: for $r,s\in\mathbb{N}$, let
$$L_{r,s}:=\Z_p p^{-r}e_1+\Z_p p^{-s}e_2.$$
$Q$ takes integral values on $L_{r,s}$ for $r\leq\lfloor a /2\rfloor$, $s\leq\lfloor b/2\rfloor$.
In this notation \refs{inclusion} states
\[
 \mcR(L_{0,0})\subset \mcR(L_{1,0})\cup \mcR(L_{0,1})\cup \bigcup_{\Lambda'\in\mcR(L_{1,1})}B(\Lambda',1)
\]
We can now apply \refs{inclusion}
again to each of the terms on the right. With each application $r$ or $s$ or both increase by $1$.
In the latter case we obtain that the previous lattice $\Lambda'\in\mcR(L_{r,s})$ (to which \refs{inclusion} was applied) is at distance 1 from the new lattice $\Lambda''\in\mcR(L_{r+1,s+1})$. Also note that in the latter case both $a$ and $b$ are reduced by $2$, so that this case can only happen $\leq\lfloor a/2\rfloor$ many times.
Therefore, induction on $a+b$ shows that
\[
 \mcR(L)=\mcR(L_{0,0})\subset \bigcup\bigl\{B(L_{\lfloor a/2\rfloor,s}, \lfloor a/2\rfloor),B(L_{r,\lfloor b/2\rfloor},\lfloor a/2\rfloor) :
  \ 0\leq r,s\leq \lfloor b/2\rfloor\bigr\}.
\]
Each $L'=L_{\lfloor a/2\rfloor,s}$ resp.\ $L'=L_{r,\lfloor b/2\rfloor}$ has invariants $(0,b')$ or $(1,b')$ with $b'\leq b$ and by the previous sections $N(L')=O(b+1)$ in all cases.
Consequently
$$N(L)\ll\sum_{L'}\sum_{\Lambda'\in\mcR(L')}|B(\Lambda',\lfloor a/2\rfloor)|\ll(b+1)^2 p^{[a/2]}.$$

\qed

\subsection{ Proof of Lemma \ref{2adic}}
Recall that we assume that $Q(x,y,z)=xy+z^2$. Note that $(1,0,0),(0,1,0)$ and $(-1,1,1)$ are three isotropic vectors that are linearly independent modulo $2$, which define the neighbors of $\Z_2^3$. For every pair $f_1,f_2$ of these vectors we can find a third vector $f_3\in\Z_2^3$ so that $Q(xf_1+yf_2+zf_3)=xy+z^2$. Of the four non-zero non-isotropic vectors modulo $2$ the vector $k=(0,0,1)$ is special, it is the only element in the kernel of $\langle\cdot,\cdot\rangle$ modulo 2 and also satisfies $k\equiv f_3$ modulo $2$ for any basis $(f_1,f_2,f_3)$ as above. Below we will always use the letter $\ov{k}$ to denote the corresponding element in the lattice $\Lambda/2\Lambda$.

\subsubsection{The diagonal case \eqref{2adicform1}}
Suppose that in a suitable basis $q$ takes the form \refs{2adicform1}. This situation is similar to the proof of Lemma \ref{padic}. We only discuss the details where the two proofs differ.

\subsubsection{The case $(a,b)=(0,0)$}
We claim that $\Lambda\in\mcR(L)$ has at most one neighbor in $\mcR(L)$.
If one of $\ov{e_1}$ or $\ov{e_2}$ is not equal to $\ov{k}$, then we claim that $\mcR(L)$ contains
at most one neighbor of $\Lambda$. To see this suppose $\ov{e_1}\neq\ov{k}$
and $L\subset\Lambda_\obfv\cap\Lambda_{\obfv'}$. Then by Fact (2) $L$ is contained
modulo $2$ in the common kernel of $\peter{\cdot,\ov{v}}$ and $\peter{\cdot,\ov{v}'}$, which is
one-dimensional and actually equal to the span of $\ov{k}$ --- a contradiction. Therefore, $L\subset\Lambda\cap\Lambda_\obfv$
for at most one neighbhor $\Lambda_\obfv$ as claimed.

So suppose $\ov{e_1}=\ov{e_2}=\ov{k}$ and $w\in\Lambda$ is such that $Q(xe_1+y(e_1+2w))=ux^2+vy^2$ as in \eqref{2adicform1}. Since we also have
\[
 Q(xe_1+y(e_1+2w))=x^2Q(e_1)+y^2Q(e_1+2w)+xy(2Q(e_1)+2\langle e_1,w\rangle)
\]
and $2\ |\ \langle e_1,w\rangle$, it follows that $Q(xe_1+y(e_1+2w))$ is not as in \eqref{2adicform1}.
So we have seen that in all possible cases we have at most one neighbor of $\Lambda$ in $\mcR(L)$. However, this shows $N(L)\leq 2$ for $(a,b)=(0,0)$.

\subsubsection{The case $a=0$ and $b\geq 1$}
We claim that the main difference between the case of $p=2$ and $p>2$ lies in this case. Here we will see that $\mcR(L)$
is only contained in the set of elements of distance one to points on a geodesic. This is caused by the fact that if $\ov{e_1}= \ov{k}$ and $\ov{e_2}=0$,
then $\mcR(L)$ contains all neighbors of $\Lambda$ due to Fact (1) and since $\ov{k}$ is orthogonal to all three nonzero isotropic vectors in $\Lambda/2\Lambda$.

On the other hand, we have already seen above (in the case $a=0,b=0$) that if $\ov{e_1}\neq\ov{k}$ then only one neighbor
of $\Lambda$ can be in $\mcR(L)$. To prove that $\mcR(L)$ consists of points of distance one from a geodesic we only have to show that if $\ov{e_1}=\ov{k}$, then for
at least one neighbor $\Lambda'$ of $\Lambda$ we have $\ov{e_1}\neq\ov{k'}$ where $\ov{k'}\in\Lambda'/2\Lambda'$ is the corresponding special vector for $\Lambda'$. This follows if we can find some vector $w\in\Lambda'$ with $\langle\ov{e_1},\ov{w}\rangle\neq 0$.

To see this we simplify the notation and assume without loss of generality $\Lambda=\Z_2^3$. Let $e_1=(\alpha,\beta,\gamma)$ so that $\langle e_1,(1,0,0)\rangle=\beta$, $\langle e_1,(0,1,0)\rangle=\alpha$, and $\langle e_1,(0,0,1)\rangle=2\gamma$. Since $\ov{e_1}\neq 0$, one quickly sees that one of these inner products is not divisible by $4$. Without loss of generality we may assume $4\ndiv\beta$. Now consider the neighbor $\Lambda'=\frac12\Z_2\times 2\Z_2\times\Z_2$ of $\Lambda$. Then $w=(\frac12,0,0)\in\Lambda'$ satisfies $\langle e_1,w\rangle=\frac12\beta\not\equiv 0$ (mod $2$). Hence as claimed, $\ov{e_1}\neq\ov{k'}$ and so only one neighbor of $\Lambda'$, namely
$\Lambda$ itself, can belong to $\mcR(L)$.

It follows that there exists a line segment $I\subset\mcR(L)$ in a geodesic in $\mcT(Q)$ so that $\mcR(L)\subset\bigcup_{\Lambda\in I}B(\Lambda,1)$.
Arguing as in Subsection \ref{case-p-0-1}
we can bound the length of $I$ in terms of $b$ and obtain $N(L)\leq 3(4b+3)$.

\subsubsection{The case $a\geq 1$} The arguments for $p>2$ carry over to the remaining cases.


\subsubsection{The non-diagonal case \eqref{2adicform2}}
So supppose now $q$ is represented by the lattice $L=\Z_2 e_1+\Z_2 e_2\subset\Lambda$ with
$$Q(e_1)=u2^b,\ Q(e_2)=v2^b,\ \peter{e_1,e_2}=w2^a,\ u,v,w\in\Z_2^\times,\ 0\leq a\leq b.$$

\subsubsection{The case $a=0$} If $(a,b)=(0,0)$, then $\ov{e_{1}}$ and $\ov{e_{2}}$ are linearly independent in $\Lambda/2\Lambda$ since
otherwise $w=\peter{e_{1},e_{2}}\equiv 0$ (mod $2$). Also note that the plane generated by $\ov{e_{1}}$ and $\ov{e_{2}}$ does not contain any isotropic vector.
However, this implies that $e_{1},e_{2}$ cannot be both contained in any $\Lambda_{\obfv}$
for then $\obfv^{\perp}$ would contain $\ov{e_{1}}, \ov{e_{2}},\obfv$ three linearly independent vectors.

If now $(a,b)=(0,b\geq 1)$,  $\ov{e_{1}}$ and $\ov{e_{2}}$ are two linearly independent isotropic vectors and so $e_{1}$ can only be contained in $\Lambda_{\ov{e_{1}}}$. Similarly, $e_{2}$ is only contained in   $\Lambda_{\ov{e_{2}}}$.  So $L$ cannot be contained in any neighbor of $\Lambda$.

In conclusion for $a=0$ we have
$$N(L)=1.$$

\subsubsection{The case $a=1$}
In that case at least one of the vectors $\ov{e_{1}}$ and $\ov{e_{2}}$ must be a non-zero isotropic vector,
for otherwise $a\geq 2$. Suppose $\ov{e_1}\neq 0$. Then $\ov{e_1}\in\Lambda_\obfv$ only for $\ov{e_1}=\obfv$.
Therefore, $L$ can only have one neighbor in $\mcR(L)$ and so $N(L)\leq 2$.


\subsubsection{The case $a\geq 2$} We consider again the two rank $2$ lattices
$$L_1:=\Z_2 e'_1+\Z_2 e_2,\ L_2:=\Z_2 e_1+\Z_2 e'_2,\ e'_i=e_i/2$$
which contain $L$ and on which $Q$ is $\Z_2$-valued:
$$Q(e'_1)=u2^{b-2},\ Q(e'_2)=v2^{b-2},\ \peter{e'_1,e_2}= \peter{e_1,e'_2}=w2^{a-1}.$$

 We describe now the type and the invariants of $L_1$ --- by symmetry $L_2$ behaves the same way.

If $a=b$ we may solve the equation in $\beta\in\Z_2^\times$
$$0=\peter{e_2+\beta e'_1,e'_1}=w2^{a-1}+\beta u2^{b-1}$$
and so $Q_{|L_1}$ is of diagonal form \eqref{2adicform1} in the basis $\{e_2+\beta e'_1,e'_1\}$. Furthermore, since
 $$\peter{e_2+\beta e'_1,e_2+\beta e'_1}=2Q(e_2+\beta e'_1)=v2^{b+1}+\beta w2^{a-1}$$
it has invariants $(a-2,b-2)$.

If $a<b$, take $\beta=2^{b-a}$: in the basis $\{e_2+\beta e'_1,e'_1\}$, $Q_{|L_1}$ takes the non-diagonal form \refs{2adicform2} with $(a',b')=(a-1,b-2)$. Finally $Q_{|L_{1}+L_{2}}=Q_{|L/2}$ takes the form \refs{2adicform2} with $(a'',b'')=(a-2, b-2)$.

We then conclude exactly as in \S \ref{padic>=2} by proving that either $N(L)=1$ or
 \refs{inclusion} holds. This implies once more the desired bound.

\subsection{Proof of Proposition \ref{2-3reps}}\label{appaproof}

We now show how the previous subsections combine to the proof of Proposition \ref{2-3reps}.

Recall that we are bounding the number of representations $N(L)$ of the quadratic form $q(x,y)=a_1x^2+a_2xy+a_3y^2$ by the ternary quadratic form $Q$ up to $SO_Q(\Z)$.
For any $p$ let us write $a_p$ and $b_p$ for the invariants of $q$ over $\Z_p$ as in Section~\ref{unramified-local}.  Let $f^2|\gcd(a_1,a_2,a_3)$ be the greatest common square divisor of the coefficients of $q$. Then $a=v_p(f)$.

By the discussion in Section \ref{counting-local}-\ref{unramified-reduction}
we know that
\[
 N(L)\ll \prod_\stacksum{p}{Q\ p-\mathrm{isotropic}} N(L_{p}).
\]
Also recall from Section \ref{unramified-reduction} that for
bounding $N(L_p)$ for $p|\disc(Q)$ we may replace $Q$ by $xy+z^2$ and $q$ by a fixed multiple $q'$ of $q$, where the factor only depends on $Q$.  From this we see that Lemma \ref{padic}-\ref{2adic} also hold for $p|\disc(Q)$ for $q$ and $Q$, except that the implicit constant depends for those primes also on $Q$.

Notice that for any prime $p>2$ we have $a_p+b_p=v_p(\disc(q))$ and $a_p=v_p(\gcd(a_1,a_2,a_3))$.
For $p=2$ we have $v_2(\disc(q))=a+b+2$ in the diagonal case and $v_2(\disc(q))=2a$ in the non-diagonal case.
Also let $c\geq 1$ be the implied constant in Lemmas \ref{padic}.
Together with Lemma \ref{padic}-\ref{2adic} this gives
\[
 N(L)\ll \prod_{p|2\disc(q)} c (v_p(\disc(q))+1)^2  p^{v_p(f)}\ll_\epsilon
 f \max(a_1,a_2,a_3)^\epsilon,
\]
as desired.

\section{Entropy, Bowen balls and uniqueness of measure of maximal entropy}\label{Appendix-B}
\subsection{Statement of main results}
We recall some notations: we work in the space $X=\Gamma\backslash G$ with $G=\SL_2(\R)$, and let $T$
 denote the time-one-map of the geodesic flow, i.e.\ the map \[
 T:x\mapsto xa \qquad\text{with }a = \begin{pmatrix}e^{1/2}&0\\0&e^{-1/2}\end{pmatrix}.
 \]
 We define a Bowen $(N,\eta)$-ball in this space to be any set of the form $xB_{N,\eta}$ with $x\in X$ and
\[
  B_{N,\eta}=\bigcap_{n=-N}^N a^{-n}
  B_\eta^{G}(e)a^n
\]
(in the sections above $\eta$ remained fixed and was omitted from the notations, but here it will be convenient to be able to use Bowen balls of varying $\eta$).

If $\Gamma$ is cocompact, for all $\eta$ sufficiently small, the Bowen $(N, \eta)$-ball $xB_{N, \eta}$ coincides with the set
\begin{equation*}
xB_{N, \eta} = \left\{ y: d (T^n(x), T^n(y))< \eta \quad \text{for all $-N \leq n \leq N$} \right\}
.\end{equation*}
This is \textbf{not} true any more for noncompact quotients, where in general the right hand side can be significantly bigger than the left hand side which is the source of some complications.

The following theorem was proved for compact quotients by Bowen in \cite{Bowen}. It is certainly well known also in the finite volume case, and proofs using leafwise measures can be found e.g. \cite[Prop.~9.6]{Margulis-Tomanov} and the more recent lecture notes \cite[Thm.~7.9]{Pisa-notes}).

\begin{Theorem} \label{max-entropy}
Let $X=\Gamma \backslash \SL_2(\R)$ and $T: X \to X$ be as above.  Then for any $T$-invariant probability measure $\nu$ the entropy satisfies
$h_\nu(T)\leq 1$. Moreover, equality holds if and only if $\nu=\mu_X$ is the $\SL_2(\R)$-invariant
probability measure on~$X$.
\end{Theorem}

 We give here a direct proof not using leafwise measures, based on Lemma~\ref{cover-entropy} (which is identical to Lemma~\ref{cover-entropy-main} and was needed for the proofs in \S\ref{ergodicsection}), in the spirit of Bowen's proof (that in itself was inspired by a proof by Adler and Weiss \cite{Adler-Weiss} of the uniqueness of measure of maximal entropy in irreducible shifts of finite type).

\begin{Lemma}\label{cover-entropy}
Let $\mu$ be an $A$-invariant measure on $X=\Gamma \backslash \SL(2,\R)$. Fix $\eta > 0$ and $\epsilon\in(0,1)$. For any $N \geq 1$ we let $BC_\eta(N,\epsilon)$ be the minimal
number of Bowen $(N, \eta)$-balls needed to cover any subset of $X$ of measure bigger than $1-\epsilon$.
Then
\begin{equation}\label{cover entropy equation}
h_\mu(T)\leq \lim_{\epsilon \rightarrow 0}\liminf_{N \rightarrow \infty}\frac{\log BC_\eta(N,\epsilon)}{2N}.
\end{equation}
\end{Lemma}

\medskip
It is easy to see that for any $\eta,\eta'>0$ a Bowen $(N,\eta)$-ball can be covered by $\ll 1$ Bowen $(N,\eta')$-balls. Therefore,
\begin{equation}\label{one lim}
\liminf_{N \rightarrow \infty}\log BC_\eta(N,\epsilon)/2N
\end{equation} is independent of $\eta$.
One can show that if $\mu$ is ergodic, equality holds in \eqref{cover entropy equation}, and moreover that the quantity in \eqref{one lim} is independent of $\epsilon$. If $\mu$ is not ergodic, then in general equality in \eqref{cover entropy equation} fails: in this case $h _ \mu (T)$ is the average of the entropy of the ergodic components of $\mu$ and the right hand side of \eqref{cover entropy equation} gives the essential supremum of the entropies of the ergodic components of $\mu$. We shall not need either fact (nor will we prove them), but will use the following related estimates for $\mu$ ergodic:

\begin{Lemma}\label{cover a little cover a lot}
	Assume that $\mu$ is in addition ergodic for $T$. Then
	for any sufficiently small $\eta$ (depending only on $X$)
	and for any $\epsilon \in(0, 1)$ and any large enough $N$ (depending on~$\mu,\epsilon$), for any
	$ \epsilon _ 1 \in (0,\epsilon)$,  if $k$ is sufficiently large (depending on~$\epsilon _ 1, \epsilon, N, \mu, \eta$) then
\begin{equation*}
\log BC_\eta(kN, \epsilon _ 1) \leq k (1-2 \epsilon) \log BC_\eta (N, \epsilon) + 4 \epsilon Nk +qk.
\end{equation*}
Here $q$ is some absolute constant.
\end{Lemma}

For our proof of Theorem~\ref{max-entropy} it is crucial that the second error term ($qk$) does \emph{not} depend on $N$. Roughly speaking the lemma says, if we manage to cover some set of measure bigger than $1-\epsilon$ by relatively few Bowen $(N,\eta)$-balls, then a set of size $1-\epsilon'$ can also be covered
by relatively few Bowen $(Nk,\eta)$-balls.

The reader may wish to look now at the proof of Theorem \ref{max-entropy} in Subsection \ref{final-proof} to see how the above two lemmas are used in combination to imply the uniqueness of the measure of maximal entropy.

\subsection{Proof of Lemma~\ref{cover-entropy}}
In the proof we will need the notion of relative entropy for partitions: For two partitions
$\Pcr=\{S_1,\ldots,S_\ell\}$ and $\mathcal{Q}=\{Q_1,\ldots,Q_m\}$ of a probability space $(X,\mu)$
the relative entropy of $\Pcr$ given $\mathcal{Q}$ is defined by
\[
  H_\mu(\Pcr|\mathcal{Q})=-\sum_{i,j}\mu(S_i\cap Q_j)\log\frac{\mu(S_i\cap Q_j)}{\mu(Q_j)},
\]
and it is easy to see that it gives  the following additivity of entropy
\begin{equation}\label{additivity of entropy}
     H_\mu(\Pcr\vee\mathcal{Q})=H_\mu(\mathcal{Q})+H_\mu(\Pcr|\mathcal{Q}).
\end{equation}
We should also use the notation $\mathcal{P} (x)$ to denote the elements of the finite or countable partition $\mathcal{P}$ containing $x$.

\begin{proof}
 Let $\Pcr=\{Q,S_1,\ldots,S_\ell\}$ be a finite partition where $Q$ is the only
 unbounded set, all boundaries $\partial S_i$ are null sets which satisfy additionally
 $$
  \mu((\partial S_i) B_\kappa^{G})<C\kappa
 $$
 for some constant $C>0$ and all $\kappa>0$, and finally
 $h_\mu(T,\Pcr)>h_\mu(T)-\delta$. Here
  $$
    h_\mu(T,\Pcr)=\lim_{N\rightarrow \infty} \frac{H_{\mu}(\Pcr^{[-N,N]})}{2N+1}
  $$
 is the expression over which one needs to takes the supremum to define $h_\mu(T)$. Such a partition
 exists since (i) by the general theory of entropy $h_\mu(T)$ can be approximated by $h_\mu(T,\mathcal{P})$
 once $\mathcal{P}$ is a sufficiently fine partition, and (ii) one can find for every $x\in X$
 arbitrary small $r>0$ for which $\mu\bigl((\partial B_r(x))B_\kappa^{G}\bigr)<C\kappa$ for all $\kappa>0$
 (since for every $x$ the function $r \mapsto \mu(  B_r(x))$ is monotone increasing hence differentiable for a.e.\ $r$.)

 We claim that for most points $x \in X$
(we shall quantify this presently) it holds that
\begin{equation}\label{in Pcr}
\Pcr^{[-N,N]}(x) \supset xB_{N, 2\eta '} \qquad \text{for $\eta ' = \eta N^{-2}$},
\end{equation}
hence if $y \in xB_{N, \eta '}$, then $yB_{N, \eta '} \subset\Pcr^{[-N,N]}(x)$.
To show this, suppose $y=xh \not\in \Pcr^{[-N,N]}(x)$ for $h \in B _ {N, \eta '}$.  Then for some $n$ with $|n|\leq N$ the elements
$$
xa^n\mbox{ and }
xha^n
$$
belong to different elements of $\Pcr$. It follows that at least one of the elements $xa^n$
belong to $(\partial P) B_{2 \eta'}^{G}$ for some $P \in \Pcr$, $|n|\leq N$. Therefore, $x$ belongs to
\begin{equation}\label{boundary-set}
\bigcup_{n=-N}^N T^n \bigcup_{S \in\Pcr} (\partial S) B_{2\eta'}^{G}
\end{equation}
which has measure less than $2(2N+1)\ell C \eta N^{-2}\ll N^{-1}$. This proves the above claim.

 Roughly speaking $B_{N,\eta}$ has length $\eta$ in the direction of $A$ and $\eta e^{-N}$ along stable and unstable
 horocycle directions while $B_{N,\eta'}$ has $\eta N^{-2}$ and $\eta N^{-2}e^{-N}$ instead. From this one can
 easily deduce that one needs at most $\ll N^6$ many translates of $B_{N,\eta'}$ to cover $B_{N,\eta}$. Choose $f>\lim_{\epsilon\rightarrow 0}\liminf_{N\rightarrow\infty}\frac{\log BC(N,\epsilon)}{2N}$. Then
 for any
 $\epsilon>0$, there is some large $N\geq 1$ depending on $\epsilon$ such that the measure of the set in \eqref{boundary-set} is less than $\epsilon$,
 and moreover such that $1-\epsilon$ of the space
 can be covered by less than $e^{2Nf}$ many translates of the set $B_{N,\eta'}$.

 Say $y_1B_{N,\eta'},\ldots, y_kB_{N,\eta'}$
 (with $k\leq e^{2Nf}$) cover $X_1\subset X$ with $\mu(X_1)\geq 1-\epsilon$. If $x\in X_1$
 is not in the union in \eqref{boundary-set}. Since $x\in y_jB_{N,\eta'}$ for some $j$,
 it follows from \eqref{in Pcr} that
 $y_jB_{N,\eta'}\subset \Pcr^{[-N,N]}(y_j)$. In other words, it follows that
 $1-2\epsilon$ of the space can be covered by $e^{2Nf}$ elements of the partition $\Pcr^{[-N,N]}$.

  Let
  $P$ be the union of these partition elements and let $\mathcal{P}=\{P,X\setminus P\}\subset\sigma(\Pcr)$
  be the associated partition. Write $\mu_B=(\mu(B))^{-1}\mu|_B$ for the normalized restriction
  of the measure $\mu$ to a Borel set $B$. It follows now from \eqref{additivity of entropy} that
 \begin{multline*}
 H_\mu(\Pcr^{[-N,N]})=H_\mu(\mathcal{P})+H_\mu(\Pcr^{[-N,N]}|\mathcal{P})\\
  =H_\mu(\mathcal{P})+\mu(P)H_{\mu_P}(\Pcr^{[-N,N]})+\mu(X\setminus P)H_{\mu_{X\setminus P}}(\Pcr^{[-N,N]})\\
  \leq \log 2+2N f+4\epsilon N\ell
 \end{multline*}
 since the entropy of a partition with $K$ elements is at most $\log K$. For $N\rightarrow\infty$
 this shows that
 \[
  h_\mu(T)-\delta < h_\mu(T,\Pcr)\leq f+2\epsilon\ell,
 \]
 which implies the lemma since $\delta$ and $\epsilon$ were arbitrary.
 (Note that $\ell$ depends on $\delta$ but not on $\epsilon$.)
\end{proof}

\subsection{Proof of Lemma~\ref{cover a little cover a lot}}
We shall say a Bowen ball $y B _ {N, \eta}$ is \emph{injective} if the map $g \mapsto y g$ is injective on $B _ {N, \eta}$.
Let $\eta _ 0>0$ be such that $2 \eta _ 0$ is smaller than the length of any closed geodesic in $X$.
An easy compactness argument shows that if $\eta \leq \eta _ 0$ for any compact $F \subset X$ there is a $N _ 0$ so that if $N > N _ 0$ and $y \in F$ the Bowen ball $y B _ {N, \eta}$ is injective.
In the proof we shall also make use of shifted $(s,t;\eta)$-Bowen balls --- sets of the form $y B _{ s ,t; \eta}$ where $B _{ s ,t; \eta}: = \bigcap_ {i = s} ^ t a ^i B_\eta ^G a^{-i}$ and $(s,t;\eta)$ sub-Bowen balls which  are simply sets of the form $y B$ for some $B \subset B _ {s, t; \eta}$.
A shifted $(s,t;\eta)$-Bowen ball $y B _{ s ,t; \eta}$ (respectively, a $ (s, t; \eta)$ sub-Bowen ball $y B $) is injective if the map $g \mapsto yg$ is injective on $B _ {s,t; \eta}$ (or $B$).
We note the following important properties of shifted Bowen balls:
\begin{enumerate}[label=(Bowen-\arabic*).,align=left,leftmargin=*]
\item For any $s \leq t \leq r$, the intersection of an injective $(s, t; \eta)$ sub-Bowen ball with an injective $(t, r; \eta)$ sub-Bowen ball can be covered by at most q injective $(s, r; \eta)$ sub-Bowen balls;
\item For any $s \leq t \leq r$,  an injective $(s, t; \eta)$ sub-Bowen ball can be covered by at most $qe^{r-t}$ injective $(s, r; \eta)$ sub-Bowen balls.
\end{enumerate}

\begin{proof}[Proof of claims]
Both claims can easily be reduced to their special cases where $t=0$ and where we only consider Bowen balls
of the form $gB_{s,r;\eta}$ in $G$ instead of injective sub-Bowen balls in $X$.

For the proof of (Bowen-1) notice that there exists some $C>0$ so that
\begin{equation}\label{pos-bowen-claim}
 g_1B_{s,0;\eta} \subset g_1 B_{C\eta}^{U^+}B_{C\eta e^s}^{U^-}B_{C\eta}^A,
\end{equation}
where $B_r^H$ denotes the $r$-ball around the identity in a subgroup $H\subset\SL_2(\R)$.
Similarly,
\begin{equation}\label{pos-bowen-claim-2}
 g_2B_{0,r;\eta} \subset g_2 B_{C e^{-r}\eta}^{U^+}B_{C\eta}^{U^-}B_{C\eta}^A.
\end{equation}
We can now decompose each of the balls appearing on the right hand side of \eqref{pos-bowen-claim}--\eqref{pos-bowen-claim-2} into $\ll 1$ many balls with certain smaller radius and obtain that $g_1B_{s,0;\eta}\cap g_2B_{0,r;\eta}$ is the union of $\ll 1$ many sets of the form
$$
 O=(g_1 u^+_1B_{\eta/8}^{U^+}u_1^-B_{\eta e^s/8}^{U^-}a_1 B_{\eta/8}^A)\cap
  (g_2 u_2^+B_{\eta e^{-r}/8}^{U^+}u_2^-B_{\eta/8}^{U^-}a_2B_{\eta/8}^A).
$$
where $ u_1^+\in B_{C\eta}^{U^+},u_2^+\in B_{C\eta e^{-r}}^{U^+}, u_1^-\in B_{C\eta e^s}^{U^-},u_2^-\in B_{C\eta}^{U^-}, a_1,a_2\in B_{C\eta}^A$.
If $g\in O$ and $\eta_0$ is sufficiently small so that conjugation by an element of distance $C\eta_0$ does not increase the distance to the identity significantly, it follows that $O\subset g B_{(s,r;\eta)}$ which proves the first claim.

The second claim follows similarly by splitting the set $B_{s,0;\eta}$ as in \eqref{pos-bowen-claim}
into $\ll e^r$ many sets of the form
$$
O= g_1 u^+_1B_{\eta e^{-r}/8}^{U+}u_1^-B_{\eta e^s/8}^{U^-}a_1 B_{\eta/8}^A
$$
with $u^+\in B_{\ll\eta}^{U^+}$ and $u^-\in B_{\ll\eta e^s}^{U^-}$, and showing that for $g\in O$ we have $O\subset g B_{s,r;\eta}$.
\end{proof}

\begin{proof}[Proof of Lemma~\ref{cover a little cover a lot}]
Let $\eta\in(0,\eta_0)$ where $\eta_0$ is as defined above, and let $M$ be sufficiently large so that $\mu (X _ {\leq M}) > 1 - \epsilon / 2$ and similarly choose $M _ 1$  so that $\mu (X _ {\leq M _ 1}) > 1 - \epsilon _ 1 / 2$.
We require that $N$ be sufficiently large so that any $(N,\eta)$-Bowen ball $y B _ {N, \eta}$ intersecting $X _ {\leq M}$ is injective, and we choose $k_1$ so that a similar statement holds for any $(k_1N, \eta)$-Bowen ball intersecting $ X _ {\leq M _ 1}$.

\IGNORE{
Since the volume (according to Haar measure) of $B_{kN, \eta/2}$ is $ C \eta e ^ {-2kN}$ the injectivity statement above implies that for $k \geq k_1$ there can be at most $C' \eta^{-1} e ^ {2kN}$ disjoint $(kN, \eta/2)$-Bowen balls intersecting $X _ {\leq M _ 1}$. By maximality, the corresponding collection of $(kN, \eta)$-Bowen balls must cover $X _ {\leq M _ 1}$ and we get the simple estimate that for $k \geq k_1$
\begin{equation*}
BC_\eta(kN, \epsilon') \leq C' \eta^{-1} e ^ {2kN}.
\end{equation*}
}

Let $\Xi$ be a collection of $(N, \eta)$-Bowen balls of cardinality $BC_\eta (N, \epsilon)$ covering a subset of $X$ with $\mu$ measure at least $1 - \epsilon$. Then
\begin{equation*}
\Xi ' = \left\{ B \in \Xi: B \cap X _ {\leq M} \neq \emptyset \right\}
\end{equation*}
has $\mu \left (\bigcup_ {B \in \Xi'} B \right) \geq 1 - \frac {3 \epsilon }{ 2}$.
Let $Y = \bigcup_ {B \in \Xi '} B$. By the pointwise ergodic theorem, there is a $k _ 2 \geq k _ 1$ and a subset $Y _ 1 \subset X _ {\leq M _ 1}$ of $\mu$-measure $\geq 1 - \frac {3 \epsilon _ 1 }{ 4}$ so that
points in $Y_1$ spend most of their time in $Y$ in the following sense:
\begin{equation}\label{equation: many times}
\frac {1 }{ 2 n} \sum_ {s = - n } ^ {n-1} 1 _ Y (T ^ s (y)) > 1 - 2 \epsilon \qquad \text{for all $n \geq k _ 2 N$ and $y \in Y _ 1$}.
.\end{equation}

To complete the proof of Lemma~\ref{cover a little cover a lot} we will show that for any $k\geq k_3$ there is a collection $\Xi _ 1$ of $(kN, \eta)$-Bowen balls covering $Y _ 1$ of cardinality
\begin{equation*}
\absolute {\Xi _ 1} \ll
  N (2q)^{k} BC_\eta (N, \epsilon) ^ {k(1-2 \epsilon)} e ^ {(4 \epsilon k + 4)N}
\end{equation*}
Let $c$ be the implied multiplicative constant. Then for large enough $q'$ (depending only on $q$ and some absolute constants above) we have $cN(2q)^{k}e^{4 N}\leq e^{q'k}$ for  all sufficiently large $k$ (where the bound is allowed to depend on $N$).
Therefore, the existence of $\Xi _ 1$ as above will establish the lemma.

Fix $k \geq k _ 2 $ and let $y\in Y_1$. We partition the finite orbit $\{T^{-kN}(y),\ldots, T^{kN-1}(y)\}$ into the $2N$
finite orbits of the form $\{T ^ {-kN+\ell}(y), T^{(-k+2)N+\ell}(y),\ldots,T^{(k-2)N+\ell}(y)\}$ for $\ell\in\{ 0,\ldots, 2N-1\}$.  By equation \eqref{equation: many times} there must for any $y \in Y _ 1$ exist some $\ell (y) \in \left\{ 0, \dots, 2N -1 \right\}$ so that
\begin{equation*}
\frac {1 }{  k} \sum_ {s=0} ^ {k -1} 1 _ Y (T ^ {(-k+2s)N +\ell (y) } (y)) \geq 1 - 2 \epsilon.
\end{equation*}

Let $L= {\lceil  (1-2\epsilon) k \rceil}$. It follows that there are $0 \leq t_1 < t_2 \dots< t_ L< k$
with $T^{(-k+2t_i)N+\ell(y)}(y)\in Y$. Furthermore, there exist injective $(N, \eta)$-Bowen balls $B _ 1, \dots, B _ L \in \Xi$ so that
\begin{equation*}
y \in \bigcap_ {i = 1} ^ L T^{-(-k+2t_i)N-\ell(y)}B_i.
\end{equation*}
Recall that $\Xi$ has $BC_\eta(N, \epsilon)$ many elements.
We now apply the
properties (Bowen-1) and (Bowen-2), and  we conclude that
the set of all $y \in Y_1$ with a given value of $\ell (y)$ and $t_1, \dots, t_L$ can be covered by
$$
 \ll BC_\eta(N, \epsilon)^{k(1-2\epsilon)+1} e^{4 N k\epsilon +2N}q^{k+1}
$$
injective $(kN,\eta)$-Bowen balls. Since there are at most $2N 2^{k}$ choices of $\ell (y)$ and $t_1, \dots, t_L$ we are done.
\end{proof}

\subsection {Proof of Theorem~\ref{max-entropy}}\label{final-proof}
We begin with the observation that the $\SL (2, \R)$-invariant measure $\mu _ X$ on $X$ achieves the upper bounds stated on the entropy, and moreover is ergodic under $T$. Let $\nu \neq \mu _ X $ be another $T$-invariant probability measure and without loss of generality we may assume that $\nu$ is  singular with respect to $\mu _ X$ (which is the case e.g.\ if $\nu$ is also ergodic), and let $\eta _ 0$ be as in the proof of Lemma~\ref{cover a little cover a lot}.

Let $f$ be a nonnegative, continuous, compactly supported function so that
\begin{equation} \label{for sandwich}
\int f \,d \mu _ X  < \int_ 0 ^ 1 dt \int f (xa_t) \,d \nu,
\end{equation} $R$ some real number strictly between the left-hand side and right-hand side of \eqref{for sandwich} and set
\begin{equation*}
Y _ T = \left\{ x: \frac1T \int_ 0 ^ T f (x a _ t) dt > R \right\}
.\end{equation*}
By construction $Y _ T$ is compact, and  (for $\epsilon > 0$ arbitrary) by the pointwise ergodic theorem if $T$ is large enough $\mu _ X (Y _ T) < \epsilon$ and $\nu (Y _ T) > 1 - \epsilon$. In fact, if $T$ is large enough, for any sufficiently large $N$ it holds that
\begin{equation} \label{Y sub T slightly thickened}
\mu _ X (Y_T B_{N, \eta_0}) <2 \epsilon
.\end{equation}
Fix such a $T$, and chose $N$ so that the \eqref{Y sub T slightly thickened} holds and moreover any $(N, \eta _ 0)$-Bowen ball intersecting $Y _ T$ is injective.

Now choose a maximal collection of disjoint $(N, \eta _ 0/2)$-Bowen balls intersecting $Y_T$. Each of these balls has $\mu _ X$ volume $\gg_{\eta_0}e^{-2N}$ (the implicit constant is independent of $\epsilon$ and $N$). In view of \eqref{Y sub T slightly thickened}, it follows that the cardinality of this collection is $\ll_{\eta_0} \epsilon e^{2N}$, and by maximality the corresponding collection of $(N, \eta _ 0)$-Bowen balls covers $Y _ T$.
As $\nu(Y_T)>1-\epsilon$  we obtain
 $BC_{\eta_0} (N, \epsilon, \nu) \ll_{\eta_0} \epsilon e ^ {2 N}$ (note that since we are simultaneously discussing two measures we have added $\nu$ to the notation $B C (\cdot)$).

Roughly speaking the above upper bound should lead  to $h_\nu(T)<1$ by using Lemma \ref{cover-entropy}:  most of the space with respect to $\nu$ is covered by relatively few, namely $\leq C \epsilon e^{2N}$, Bowen $(N,\eta)$-balls. However, as \eqref{cover entropy equation} first takes the limit as $N\to\infty$ this inequality is not directly implying $h_\nu(T)<1$. To overcome this we introduce an
$\epsilon'\in(0,\epsilon)$ and will use Lemma \ref{cover a little cover a lot}
to obtain the bound on the covering number for $\epsilon'$ and $kN$.
Indeed applying Lemma~\ref{cover a little cover a lot} we conclude that for any $\epsilon '\in(0,\epsilon)$ if $k$ is sufficiently large
\begin{multline*}
\log B C _ {\eta _ 0} (kN, \epsilon ', \nu) \leq k (1-2 \epsilon)
(2N+ \log (C\epsilon)) + 4 \epsilon kN + q k \\
\leq k (1-2 \epsilon)2N +\frac12k\log (C\epsilon) + 4 \epsilon kN + q k= 2Nk + \Bigl(q + \frac12\log(C \epsilon)\Bigr)k,
\end{multline*}
where we also assumed $\epsilon<1/4$ and $C\epsilon<1$.
Hence we obtain for any $\epsilon '\in(0,\epsilon)$ that
\begin{equation*}
\liminf_ {k \to \infty} \frac {1 }{ 2kN} \log B C _ {\eta _ 0} (kN, \epsilon ', \nu) \leq 1 + \frac {2q + \log( C\epsilon)}{4N}
.\end{equation*}
However, for sufficiently small $\epsilon$ the right hand side is $<1$. Hence by Lemma~\ref{cover-entropy} we get $h _ \nu (T)<1$. Therefore, $m_X$ is the only probability measure on $X$ with $h_{m_X}(T)\geq 1$.

\bibliographystyle{plain}
\def\cprime{$'$}

\end{document}